\documentclass[11pt]{article}

\usepackage{amsfonts,amsmath,amssymb}
\usepackage{amsthm}
\usepackage{graphicx}

\newtheorem{theorem}{Theorem}[section]
\newtheorem{thm}[theorem]{Theorem}
\newtheorem{prop}[theorem]{Proposition}
\newtheorem{lem}[theorem]{Lemma}
\newtheorem{claim}[theorem]{Claim}
\newtheorem{cor}[theorem]{Corollary}
\newtheorem{rem}[theorem]{Remark}

\newtheorem{obs}[theorem]{Observation}
\newtheorem{defn}[theorem]{Definition}

\newtheorem{conjecture}[theorem]{Conjecture}

\newtheorem{example}[theorem]{Example}

\begin{document}
\global\long\def\f{\mathcal{F}}
\global\long\def\l{\mathcal{L}}
\global\long\def\pn{\mathcal{P}\left(\left[n\right]\right)}
\global\long\def\g{\mathcal{G}}
\global\long\def\s{\mathcal{S}}
\global\long\def\m{\mathcal{M}}
\global\long\def\mp{\mu_{p}}
\global\long\def\j{\mathcal{J}}
\global\long\def\d{\mathcal{D}}
\global\long\def\Inf{}
\global\long\def\p{\mathcal{P}}
\global\long\def\mpo{\mu_{p_{0}}}
\global\long\def\h{\mathcal{H}}
\global\long\def\n{\mathbb{N}}
\global\long\def\a{\mathcal{A}}
\global\long\def\b{\mathcal{B}}
\global\long\def\ex{\mathrm{ex}}
\global\long\def\c{\mathcal{C}}
\global\long\def\sor{Sor}
\global\long\def\fo{Fo}
\global\long\def\i{\mathcal{I}}
\global\long\def\e{\mathbb{E}}
\global\long\def\t{\mathcal{T}}
\global\long\def\u{\mathcal{U}}

\newcommand{\bn}[0]{\bigskip\noindent}
\newcommand{\mn}[0]{\medskip\noindent}
\newcommand{\nin}[0]{\noindent}

\oddsidemargin=0.15in \evensidemargin=0.15in \topmargin=-.5in
\textheight=9in \textwidth=6.25in

\setcounter{tocdepth}{1}

\title{The Junta Method for Hypergraphs and the Erd\H{o}s-Chv\'{a}tal Simplex Conjecture}

%

\author{
Nathan Keller\thanks{Department of Mathematics, Bar Ilan University, Ramat Gan, Israel.
{\tt Nathan.Keller@biu.ac.il}. Research supported by the Israel Science Foundation (grants no.
402/13 and 1612/17) and by the Binational US-Israel Science Foundation (grant no. 2014290).}
\mbox{ } and Noam Lifshitz\thanks{Einstein Institute of Mathematics, Hebrew University of Jerusalem, Israel.
{\tt noamlifshitz@gmail.com}. Research supported by the Adams Fellowship Program of the Israel Academy
of Sciences and Humanities.}
}

\date{}

\maketitle

\begin{abstract}
Numerous problems in extremal hypergraph theory ask to determine the maximal size of a $k$-uniform hypergraph on $n$ vertices
that does not contain an `enlarged' copy $H^+$ of a fixed hypergraph $H$. These include well-known problems such as the
Erd\H{o}s-S\'{o}s `forbidding one intersection' problem and the Frankl-F\"{u}redi `special simplex' problem.

We present a general approach to such problems, using a `junta approximation method' that originates from analysis of Boolean functions. We prove that any $H^+$-free hypergraph is essentially contained in a `junta' -- a hypergraph determined by a small number of vertices -- that is also $H^+$-free, which effectively reduces the extremal problem to an easier problem on juntas. Using this approach, we obtain, for all $C<k<n/C$, a solution of the extremal problem for a large class of $H$'s, which includes the aforementioned problems, and solves them for a large new set of parameters.

We apply our method also to the 1974 Erd\H{o}s-Chv\'{a}tal simplex conjecture, which asserts that for any  $d < k \leq \frac{d}{d+1}n$, the maximal size of a $k$-uniform family that does not contain a $d$-simplex (i.e., $d+1$ sets with empty intersection such that any $d$ of them intersect) is ${{n-1}\choose{k-1}}$. We prove the conjecture for all $d$ and $k$, provided $n>n_0(d)$.
\end{abstract}


\renewcommand{\baselinestretch}{0.65}\normalsize
\tableofcontents
\renewcommand{\baselinestretch}{1.0}\normalsize



\section{Introduction and Review}
\label{sec:intro}

\subsection{Background}
\label{sec:sub:intro:background}

Extremal combinatorics concerns the problem of determining the maximal or the minimal size of a combinatorial
object that has certain properties. In the last decades, extremal combinatorics has grown tremendously,
partly due to strong connections with other fields, such as number theory, harmonic analysis, geometry, and
theoretical computer science.

A specific class of problems that played a central role in this development are Tur\'{a}n-type questions,
which concern the maximal size of a graph or a hypergraph that does not contain a specific `forbidden'
sub-structure.

One of the first results of this class was obtained in 1907 by Mantel: he showed that any graph on $n$ vertices
with more than $\left\lfloor \frac{n}{2}\right\rfloor \cdot \left\lceil \frac{n}{2}\right\rceil $ edges must contain
a triangle. In 1941,  Tur\'{a}n~\cite{turan1941basic} posed the general question of determining
the maximal number of edges $ex(n,H)$ in a graph (or a hypergraph) on $n$ vertices that does not contain a copy of
a `forbidden' graph (or hypergraph, respectively) $H$. Tur\'{a}n himself solved the problem in the case where $H$ is
a complete graph, and Erd\H{o}s, Stone, and Simonovits~\cite{erdos1966limit,erdos1946structure} essentially solved
the problem for all forbidden graphs except bipartite graphs, for which it has been the subject of intensive research in
the last decades (see the ICM'2010 talk of Sudakov~\cite{sudakov2010ICM} for an excellent survey of these topics).
In recent years, much progress was made also in Tur\'{a}n-type theorems in random structures, which assert that
if a random graph $G$ on $n$ vertices is constructed by choosing each edge with a
probability $p$ that is beyond a certain threshold, then $G$ contains a copy of the forbidden graph $H$ `almost
surely' (see~\cite{conlon2016combinatorial,schacht2016extremal}).

Unlike the case of graphs, the Tur\'{a}n problem for $k$-uniform hypergraphs (i.e., hypergraphs in which each edge
contains exactly $k$ vertices, or in other words, families of $k$-element subsets of $[n]=\{1,2,\ldots,n\}$)
turned out to be much more difficult. Here, even for the most basic
extension of Mantel's theorem -- $3$-uniform hypergraphs without a complete $3$-uniform hypergraph on 4 vertices --
$ex(n,H)$ is not known, not even approximately (see the survey~\cite{keevash2011hypergraph}). One well-known problem
of this class is the $(6,3)$-problem, which asks for an upper bound on the size of a $3$-uniform hypergraph with no
three edges whose union contains at most 6 vertices. The Ruzsa-Szemer\'{e}di theorem~\cite{ruzsa1978triple}, which
asserts that any such hypergraph has $o(n^2)$ edges, yields immediately a short proof of the classical Roth's
theorem~\cite{roth1953certain} asserting that any subset of $[n]$ of positive density contains a $3$-element
arithmetic progression.

Another extensively-studied class of extremal problems is that of \emph{intersection problems} (see~\cite{frankl2016survey}),
in which the restriction on the set family concerns the sizes of intersections between its members. This field was
initiated in 1961 by Erd\H{o}s, Ko, and Rado~\cite{erdos1961intersection}, who showed that for all $k \leq n/2$,
the maximal size of a family $\f \subseteq {{[n]}\choose{k}}$ (i.e., a family of $k$-element subsets of $[n]$)
in which every two sets have a non-empty intersection
is ${n-1}\choose{k-1}$. Numerous intersection theorems were proved, and some of them were applied in other fields. For
example,
the Ahlswede-Khachatrian
theorem~\cite{ahlswede1997complete} which finds the maximal size of $\f \subseteq {{[n]}\choose{k}}$ in which any two sets
intersect in at least $t$ elements, is a crucial component in the hardness-of-approximation theorem for the `vertex
cover' problem, proved by Dinur and Safra~\cite{dinur2005hardness}.

A setting that includes many of the Tur\'{a}n-type problems, as well as many of the intersection problems, is that of
\emph{Tur\'{a}n problems for expansion}~\cite{mubayi2016survey}: For any hypergraph $\h$ with edges of size at most $k$,
the \emph{$k$-expansion of $\h$}, denoted by $\h^{+}$, is the $k$-uniform hypergraph obtained from $\h$ by adding to
each of its edges distinct new vertices. The Tur\'{a}n problem for expansion asks to determine the
maximal size of a $k$-uniform hypergraph on $n$ vertices that does not contain a copy of $\h^{+}$, for some fixed
hypergraph $\h$.

The simplest example of such a problem is the aforementioned Erd\H{o}s-Ko-Rado theorem, which corresponds to $\h$
that consists of two disjoint edges. Well-known open problems that fall into this framework include, among others:
\begin{itemize}
\item The 1965 Erd\H{o}s \emph{matching conjecture}~\cite{erdHos1965problem}, which asks for the maximal size of a $k$-uniform
hypergraph that does not contain $t$ pairwise disjoint edges;

\item The 1975 Erd\H{o}s-S\'{o}s \emph{forbidding one intersection} problem~\cite{erdHos1975problems}, which asks for the
maximal size of a $k$-uniform hypergraph that does not contain two edges whose intersection is of size exactly $t-1$;

\item The 1987 Frankl-F\"{u}redi \emph{special simplex} problem~\cite{frankl1987exact}, in which the forbidden configuration is $d+1$ edges $E_1,\ldots,E_{d+1}$
such that there exists a set $S=\{v_1,\ldots,v_{d+1}\}$ for which $E_i \cap S = S \setminus \{v_i\}$ for any $i$ and the
sets $\{E_i \setminus S\}$ are pairwise disjoint.
\end{itemize}


Tur\'{a}n problems for expansion were studied intensively in the last decades. Besides results for specific problems, general results on such
problems were obtained using various methods, most notably the \emph{delta-system}, the \emph{stability}, and the recently proposed \emph{random sampling from the shadow} methods, which we briefly describe in Section~\ref{sec:sub:intro:techniques} below.

Despite these advances, our understanding in this field is still very limited (see~\cite{furedi2014exact,mubayi2016survey}). In particular, except
for the simplest case where $\h$ consists of $t$ pairwise disjoint edges and a few very special cases of other problems, all known results (including those obtained by the aforementioned delta-system method) hold only where $\h$ and $k$ are fixed and $n$ is very large with respect to $k$. Moreover, even for a fixed $k$ and a sufficiently large $n$ where general techniques are available, these techniques only yield scattered exact results that solve specific Tur\'{a}n problems but not general exact results for large families of forbidden graphs. Furthermore, there are no structural results describing `large' $\h^{+}$-free families, except for the most basic cases, like intersecting families where such results are known ever since the 1967 Hilton-Milner theorem~\cite{hilton1967some} and a very satisfactory characterization was obtained by Frankl~\cite{frankl1987erdos} already in 1987.

\subsection{Our results}
\label{sec:sub:intro:results}

In this paper we present a new approach to Tur\'{a}n problems for expansion that applies for `large' values of $k$ (specifically,
$C <k<n/C$, where $C$ depends only on $\h$) and allows obtaining general exact results for large classes of forbidden hypergraphs,
along with structural results that characterize the almost-extremal families (i.e., the families whose size is close to the maximum
size). We obtain our results using a method we
call \emph{the junta method}, which takes its origin from the work of Dinur and Friedgut~\cite{dinur2009intersecting} and
is motivated by results in analysis of Boolean functions. We describe the method in Section~\ref{sec:sub:intro:techniques},
after the description of our results.

\subsubsection{A general structure theorem for large $\h^+$-free families}

Our most general result is a structure theorem which asserts that for any $\h$ and any $C<k<n/C$, every `large' $\h^+$-free family can be
approximated by a simpler family called `junta' that is also $\h^+$-free. Before we present the result, a few words on juntas and their
origin are due.
\begin{defn}
Let $n,k\in\mathbb{N}$, and let $J\subseteq\left[n\right]$ satisfy $|J|<k$. A hypergraph $\j \subseteq {{\left[n\right]}\choose{k}}$
is said to be a $J$-junta if the answer to the question of whether $A$ belongs to $\j$ depends only on $A\cap J$. (Formally: If $A,B\subseteq {{\left[n\right]}\choose{k}}$
are sets such that $A\cap J=B\cap J$, then $A \in \j \Leftrightarrow B \in \j$). A $j$-junta is a $J$-junta for some set $J$ of size $j$. Informally, a hypergraph is called a {}`junta' if it is a $j$-junta for some `constant' $j$.
\end{defn}
While the definition of `junta' may look unnatural in the context of hypergraphs, it is natural in the Boolean functions context it comes from. A Boolean function $f:\{0,1\}^n \rightarrow \{0,1\}$ is called a `junta' if it depends on a small number of variables. Being a natural `simple' structure of a Boolean function, juntas were studied extensively in the last 20 years and numerous `approximation-by-junta' theorems were obtained.
The first (and probably the best-known) of these is the 1998 Junta theorem of Friedgut~\cite{friedgut1998boolean} which states that any function with a small `total influence' (which is the same as a subset of the vertices of the discrete cube graph $\{0,1\}^n$ that has a small edge-boundary) can be approximated by a constant-sized junta. Approximation by juntas in the spirit of Friedgut's theorem or by structures alike (such as `pseudo-juntas')
plays an important role in some of the most prominent results in analysis of Boolean functions (e.g.,~\cite{friedgut1999sharp,hatami2012structure}), as well as in the central applications of Boolean functions analysis to theoretical computer science (e.g., to machine learning~\cite{odonnell2007learning} and hardness of approximation~\cite{dinur2005hardness}). The proofs of these results rely on tools from harmonic analysis and functional analysis (see, e.g.,~\cite{beckner1975inequalities,bonami1970etude,gross1975logarithmic}).

Tools from analysis of Boolean functions have already been applied to extremal combinatorics in a number of works (e.g.,~\cite{ellis2012triangle,ellis2011intersecting}). In particular, juntas were used by Friedgut~\cite{friedgut2008measure} for obtaining a `stability version' of the aforementioned Ahlswede-Khachatrian theorem~\cite{ahlswede1996complete}, that characterizes the `almost-extremal' families in which every two sets intersect in at least $t$ elements. The closest to this paper is a result of Dinur and Friedgut~\cite{dinur2009intersecting} which proves that any intersecting family is essentially contained in an intersecting junta. In the terminology of Tur\'{a}n problems for expansion, the result of~\cite{dinur2009intersecting} asserts that in the case where $\h$ consists of two disjoint edges, any $\h^+$-free family can be approximated by an $\h^+$-free junta.

Our general theorem asserts that the same holds for \emph{any} constant-size forbidden hypergraph $\h$, provided that $k=k(n)$ is `not too large'.
\begin{thm}
\label{thm:Junta-approx-theorem}For any fixed hypergraph $\h$, there exist constants $C,j$ such that the following holds. Let $n \in \mathbb{N}$ and let $C<k < n/C$. Suppose that $\f \subseteq {{[n]}\choose{k}}$ is free of $\h^+$. Then there exists an $\h^+$-free junta
$\j\subseteq  {{\left[n\right]}\choose{k}}$ which depends on at most $j$ coordinates, such that
\[
\left|\f\backslash\j\right|\le\max\left( e^{-k/C},C\frac{k}{n}\right) \cdot |\j|.
\]
\end{thm}
In particular, for $k$ sub-linear in $n$ that tends to infinity with $n$, Theorem~\ref{thm:Junta-approx-theorem} implies that the asymptotically largest $\h$-free families are juntas.

Theorem~\ref{thm:Junta-approx-theorem} may be viewed in light of the classical theorem of Frankl and F\"{u}redi~\cite{frankl1987exact} from 1987 that determines the asymptotic size of the extremal $\h^+$-free families $\f \subseteq {{[n]}\choose{k}}$, for any forbidden hypergraph $\h$ (Theorem~\ref{thm:Frankl-Furedi} below). While the result of Frankl and F\"{u}redi holds only for $n$ very large with respect to $k$ (specifically, for $k \leq O(\log \log n)$) and specifies only the asymptotic size of the family, our result holds for $C < k < n/C$ and specifies the family's approximate structure (see Section~\ref{sec:sub:proof:junta} for a detailed comparison).

A few words regarding the tightness of Theorem~\ref{thm:Junta-approx-theorem} are due. For $k=O(\log n)$, the dominant term in the r.h.s.~is $e^{-k/C} \cdot |\j|$. This term is sharp, as can be seen in the following example, that concerns the aforementioned `forbidding one intersection' problem. Let $\h$ consist of two edges that intersect in a single vertex; thus, an $\h^+$-free hypergraph is a family that does not contain two sets that intersect in a single element. Define $\f \subseteq {{[n]}\choose{k}}$ as
\[
\f = \{A \subseteq {{\{1,2,\ldots,n/2\}}\choose{k}}: \{1,2\} \subset A\} \bigcup \{A' \subseteq {{\{n/2+1,\ldots,n\}}\choose{k}}: \{n-1,n\} \subset A'\}.
\]
It is easy to see that $\f$ is $\h^+$-free, while for any constant-sized $\h^+$-free junta $\j$ approximating $\f$ we have $|\f \setminus \j| \geq e^{-O(k)}{{n}\choose{k}}$ (since each such junta must miss an $\Omega(1)$ fraction of one of the `halves' of $\f$).

For $k = \Omega(\log n)$, the dominant term in the r.h.s.~is $(k/n)\cdot |\j|$. We conjecture that this term can be replaced by $(k/n)^r \cdot |\j|$ for any constant $r$, assuming $n \geq n_0(r)$. See Proposition~\ref{prop:uncap cross k large} for a step in this direction.

\subsubsection{Forbidden hypergraphs for which the extremal family is a star}

Using our general structure theorem, we solve the Tur\'{a}n problems for expansion for several families of forbidden hypergraphs. Here, instead of considering each problem (i.e., each forbidden hypergraph) separately as was done in previous works, we focus on the \emph{structure of the extremal solutions}, and obtain a characterization of all forbidden hypergraphs $\h$ for which the extremal example is a specific family, in terms of intrinsic properties of the hypergraph $\h$. The extremal examples we chose to consider are those which appear in many of the extensively studied Tur\'{a}n problems for expansion -- in particular, the problems mentioned above.
\begin{defn}
For $t,s \in \mathbb{N}$ such that $s \leq t \leq k$, the \emph{$(t,s)$-star} is the family $\f= \{A \in {{[n]}\choose{k}}: |A \cap T| \geq s\}$, for some
specific $T \in {{[n]}\choose{t}}$. In particular, the $(t,t)$-star is the family $\s_T= \{A \in {{[n]}\choose{k}}: T \subseteq A\}$ and the $(t,1)$-star is
the family $\s'_T= \{A \in {{[n]}\choose{k}}: T \cap A \neq \emptyset\}$.
\end{defn}

Our first exact result is a complete characterization of the forbidden hypergraphs $\h$ for which the extremal family is a $(t,t)$-star. It is clear that
in order for the $(t,t)$-star to be extremal, it is necessary that the $(t,t)$-star is free of $\h^{+}$ and that any family that properly contains a $(t,t)$-star contains a copy of $\h^{+}$. It is easy to show (see Lemma~\ref{lem:inclusion maximal-> intrinsic}) that these necessary conditions are equivalent to the following property of $\h$: The kernel $K(\h)$ (i.e., the intersection of all edges of $\h$) is of size $t-1$, and there exists a set $S$ of size $2t-1$ that is included in all edges of $\h$ but one. Our theorem shows that these necessary
conditions are in fact sufficient.

Denote the number of edges of $\h$ by $|E(\h)|$ and the number of vertices that belong to at least two edges of $\h$ (the \emph{center} of $\h$) by $|C(\h)|$.
\begin{thm}\label{thm:exact solution for t-stars}
For any $m \in \mathbb{N}$ there exists a constant $C(m)$, such that the following holds. Let $n,k \in \mathbb{N}$ be such that $C< k< n/C$. Then for any hypergraph $\h$ with $\max(|E(\h)|,|C(\h)|) \leq m$, the following conditions are equivalent:
\begin{enumerate}
\item The maximal size of an $\h^+$-free family $\f \subseteq {{[n]}\choose{k}}$ is ${{n-t}\choose{k-t}}$.

\item The extremal $\h^+$-free families are all the $(t,t)$-stars, and only them.

\item The kernel $K(\h)$ is of size $t-1$, and there exists $S$ of size $2t-1$ that is included in all edges of $\h$ but one.
\end{enumerate}
\end{thm}

Theorem~\ref{thm:exact solution for t-stars} applies to various previously-studied problems, and in particular, provides a complete solution for the aforementioned `special simplex' and `forbidding one intersection' problems in the range $C < k < n/C$. Previous results on these problems apply only for a fixed $k$ and $n$ very large with respect to $k$ (see below).

In the cases where Theorem~\ref{thm:exact solution for t-stars} applies, we also obtain a stability result which describes the structure of families whose size is close to the maximum possible (as is common in many recent results in extremal combinatorics, e.g.,~\cite{keevash2008shadows,mubayi2007structure}). The result asserts that any `large' $\h^+$-free family is essentially contained in a $(t,t)$-star.
\begin{thm}[Stability for Theorem~\ref{thm:exact solution for t-stars}]
\label{thm:Stab-for-Stars}
For any constants $m,r$, there exists a constant $C(m,r)$ such that the following holds. Let $n,k \in \mathbb{N}$ be such that $C< k< n/C$ and let $\h$ be a hypergraph that satisfies the conditions of Theorem~\ref{thm:exact solution for t-stars}.

For any $0<\epsilon<1/C$ and for any $\h^+$-free family $\f \subseteq {{[n]}\choose{k}}$ with $\left|\f\right|\ge \left(1-\epsilon\right){{n-t}\choose{k-t}}$, there exists a $(t,t)$-star $\s_T$ such that
\[
\left|\f\backslash\s_T\right|\le \epsilon^{r} {{n-t}\choose{k-t}}.
\]
\end{thm}

\medskip Our second exact result concerns the case where the extremal family is a $(t,1)$-star. This case is somewhat more complex than the case of $(t,t)$-stars since in this case, the obviously necessary conditions for the $(t,1)$-stars to be extremal turn out to be insufficient, as demonstrated by the following example.

\begin{example}
\label{or s extremal_intro} Let $C$ be a sufficiently large constant, let $C<k<n/C$, and let $\h^+$ be the $k$-expansion of the
hypergraph $\h=\left\{ \left\{ 1,2\right\}, \left\{ 1,4\right\}, \left\{ 1,5\right\} ,\left\{ 2,6\right\} , \left\{ 2,7\right\}, \left\{ 3\right\} \right\}$.
It is easy to verify (see Section~\ref{sec:sub:proof:porcupine}) that the $\left(2,1\right)$-star $\s'_{\{1,2\}}$ is $\h^+$-free and is maximal under inclusion among the $\h^+$-free families. However, the family $\f=\left\{ A\in {{[n]}\choose{k}}\,:\,\left|A\cap\left\{ a,b,c\right\} \right|=1\right\}$ (for arbitrary distinct $a,b,c \in [n]$), which is larger than $\s'_{\{1,2\}}$, is $\h^+$-free as well.
\end{example}

To avoid Example~\ref{or s extremal_intro} and its relatives, we bound our discussion to hypergraphs $\h$ that satisfy a stronger condition: instead of requiring that no family that properly contains the $(t,1)$-star $\s'_{T}$ is $\h^+$-free, we require that no family that contains $\left\{ A\,:\,\left|A\cap T \right|=1\right\} $ and is not contained in $\s'_T$, is $\h^+$-free. (Note that this condition is closely related to the notion of \emph{cross-cuts}, introduced by Frankl and F\"{u}redi in~\cite[Sec.~5]{frankl1987exact}.)
It is easy to show (see Lemma~\ref{lem:inclusion maximal-> intrinsic-1}) that these conditions are equivalent to the following property of $\h$: There exists a set of size $t$ that intersects all the edges of $\h$ except for one in a single element, while there is no set of size $t$ that intersects all edges of $\h$. Our theorem shows that these
conditions are sufficient.

For a hypergraph $H$, and a set of vertices $S \subseteq V(H)$, we denote $\mathrm{span}_H(S)= \{E \in E(H):E \cap S \neq \emptyset\}$, and $\mathrm{span}^1_H(S)= \{E \in E(H):|E \cap S|=1 \}$.
\begin{thm}\label{thm:OR}
For any $m \in \mathbb{N}$, there exists a constant $C(m)$, such that the following holds. Let $n,k \in \mathbb{N}$ be such that $C< k< n/C$.

Suppose that for some hypergraph $\h$ with $\max(|E(\h)|,|C(\h)|) \leq m$ and for some $t \in \mathbb{N}$, there exists $T \subseteq V(\h)$ with $|T|=t$, such that $\left|\mathrm{span}_{\h}^{1}\left(T\right)\right|=\left|E(\h)\right|-1$ and that there is no set of vertices $T'$ of size $t$,
such that $\mathrm{span}_{\h}\left(T'\right)=\h$.

Then the maximal size of a family $\f \subseteq {{[n]}\choose{k}}$ that is free of $\h^+$ is ${{n}\choose{k}}-{{n-t}\choose{k}}$, and the unique extremal examples are the $(t,1)$-stars.
\end{thm}
Theorem~\ref{thm:OR} is not a complete characterization of the forbidden hypergraphs for which the $(t,1)$-star is an extremal family. Nevertheless, it provides a sufficient condition that holds for many well-known problems. In particular, Theorem~\ref{thm:OR} provides a complete solution (in the range $C<k<n/C)$ for the aforementioned `Erd\H{o}s matching conjecture' and for the `forbidden paths/cycles' problems
studied by Kostochka, Mubayi, and Verstra\"{e}te~\cite{kostochka2015turan} (see below). In the case of the Erd\H{o}s matching conjecture, a stronger result was obtained
by Frankl~\cite{frankl2013improved} in 2013. In the case of forbidden paths and cycles, the best previous results on these problems apply only for a fixed $k$ and $n$ very large with respect to $k$ (as we describe below).

A stability result for Theorem~\ref{thm:OR} holds as well (see Theorem~\ref{thm:Solution for t-porcupines}).

\subsubsection{The Erd\H{o}s-Chv\'{a}tal simplex conjecture}

We apply our approach to a well-known conjecture of Erd\H{o}s~\cite{erdos1971topics} and Chv\'{a}tal~\cite{chvatal1974extremal} from 1974 which generalizes the Erd\H{o}s-Ko-Rado theorem~\cite{erdos1961intersection}.

\medskip A \emph{$d$-simplex} is a family of $d+1$ sets that have empty intersection, such that the intersection of any $d$ of them is nonempty.
\begin{conjecture}[Chv\'{a}tal, 1974]
Let $d<k\le\frac{d}{d+1}n$ , and let $\f \subseteq {{[n]}\choose{k}}$ be a family that does not contain a $d$-simplex.
Then $\left|\f\right|\le {{n-1}\choose{k-1}}$, with equality if and only if $\f$ is a $(1,1)$-star.
\end{conjecture}
To be precise, Erd\H{o}s~\cite{erdos1971topics} raised the conjecture in 1971 in the specific case $d=3$, and three years later, Chv\'{a}tal~\cite{chvatal1974extremal} raised the conjecture for any $d$.
The conjecture has a long history of partial results. In 1976, Frankl~\cite{frankl1976sperner} showed that the Erd\H{o}s-Chv\'{a}tal conjecture holds for $k\ge\frac{d-1}{d}n$. In 1987, Frankl and F\"{u}redi~\cite{frankl1987exact} showed that that it holds for all $n\ge n_{0}\left(k,d\right)$. In 2005, Mubayi and Verstra\"{e}te~\cite{mubayi2005proof} settled completely the case $d=3$, improving over several previous
results~\cite{bermond1977chvatal,chvatal1974extremal,frankl1981chvatal} and resolving the original conjecture of Erd\H{o}s. In 2010, Keevash and Mubayi~\cite{keevash2010set} showed that Chv\'{a}tal's conjecture holds for $\zeta n\le k\le\frac{n}{2}-O_{\zeta,d}\left(1\right)$, for all $\zeta>0$.

\medskip

Our methods do not apply directly to the Erd\H{o}s-Chv\'{a}tal  conjecture, since the $d$-simplex is not an expansion of a fixed hypergraph $\h$. However, for $C<k<n/C$ (where $C=C(d)$), Theorem~\ref{thm:exact solution for t-stars} implies an even stronger result: any family $\f \subseteq {{[n]}\choose{k}}$ with more than ${{n-1}\choose{k-1}}$ edges contains a `special simplex' -- a copy of $\h^+$, where the edges of $\h$ are all the $d$-element subsets of a $(d+1)$-element set -- which in particular is a $d$-simplex. Adding to this the previous results on the Erd\H{o}s-Chv\'{a}tal conjecture, only the `large $k$' case is left. We solve it using problem-specific techniques to obtain the following:
\begin{thm}
For any $d \in \mathbb{N}$ there exists $n_0(d)$ such that the following holds. Let $n>n_0(d)$, let $d<k\le\frac{d}{d+1}n$ , and let $\f \subseteq {{[n]}\choose{k}}$ be a family that does not contain a $d$-simplex. Then $\left|\f\right|\le {{n-1}\choose{k-1}}$, with equality if and only if $\f$ is a $(1,1)$-star.
\end{thm}
\noindent This proves the Erd\H{o}s-Chv\'{a}tal simplex conjecture for all $k$, given that $n$ is sufficiently large as function of $d$. Note that the aforementioned Frankl-F\"{u}redi result~\cite{frankl1987exact} proves the conjecture for a fixed $k$, where $n$ goes to infinity. Our theorem proves the conjecture in the
entire range $d<k\le\frac{d}{d+1}n$ (i.e., $k$ is allowed to grow with $n$), provided that $n$ is sufficiently large only as function of $d$ (which is assumed to be fixed, like in all previous works on the conjecture).

\subsubsection{Other applications}
\label{sec:app}

As mentioned above, Theorems~\ref{thm:exact solution for t-stars} and~\ref{thm:OR} apply to a number of well-studied Tur\'{a}n-type problems. In this subsection we present briefly several of these applications, mainly focusing on the problems mentioned above, and compare them to previous results on the respective problems.


\mn \emph{The Erd\H{o}s-S\'{o}s `forbidding one intersection' problem.} Posed in 1975~\cite{erdHos1975problems}, the problem asks for the maximal size $f(n,k,t)$ of a $k$-uniform hypergraph on $n$ vertices that does not contain two edges whose intersection is of size exactly $t-1$. Equivalently, the problem asks what is the maximal size of an $\i_t^+$-free $k$-uniform hypergraph, where the forbidden hypergraph $\i_t$ consists of two edges that share exactly $t-1$ vertices.

This problem, for different regimes of the relation between $n,k,$ and $t$, was studied in numerous works, including the
Frankl-Wilson~\cite{frankl1981intersection} theorem (which considers the case where multiple intersection sizes that satisfy some modular conditions are forbidden) and the Frankl-R\"{o}dl~\cite{frankl1987forbidden} theorem (which considers the case where $t$ is linear in $n$). Recent works considered variants of the problem in different settings, e.g., for permutations with forbidden intersections~\cite{ellis2014forbidding,keevash2017frankl}.

We consider the `forbidding one intersection' problem in the regime where $t$ is constant and $n$ is large. In this regime, the problem is related to the classical problem of determining the maximal size of a $t$-intersecting family (i.e., a family in which every two sets intersect in at least $t$ elements), proposed by Erd\H{o}s, Ko, and Rado~\cite{erdos1961intersection}. For the `$t$-intersecting' problem, Ahlswede and Khachatrian~\cite{ahlswede1997complete} showed that the maximal $t$-intersecting family has size $\max_r{|\f_{n,k,t,r}|}$, where $\f_{n,k,t,r}=\{A \in {{[n]}\choose{k}}: |A \cap [t+2r]| \geq t+r\}$, proving a conjecture of Frankl~\cite{frankl1978nkt}. In particular, for all $n \geq (k-t+1)(t+1)$, the extremal size is ${{n-t}\choose{k-t}}$ (which was proved much earlier by Wilson~\cite{wilson1984exact}).

In 1977, Frankl~\cite{frankl1977singleton} showed that for all $k \geq 4$ and $n>n_0(k)$, we have $f(n,k,2)={{n-2}\choose{k-2}}$. At the same year, Frankl~\cite{frankl1977constructive} determined $f(n,k,t)$ up to a constant factor for all $k \geq 3t-2$ and $n \geq n_0(k)$. Generalizing the two aforementioned results, Frankl and F\"{u}redi~\cite{frankl1985forbidding} showed in 1985 that if $k\ge2t$ and $n\ge n_{0}\left(k,t\right)$, then any $\i_{t}^+$-free family $\f \subseteq {{[n]}\choose{k}}$ satisfies $\left|\f\right|\le{{n-t}\choose{k-t}}$, with equality if and only if $\f$ is a $(t,t)$-star. This means that, perhaps surprisingly, forbidding only intersections of size $t-1$ yields the same result as forbidding all intersections of size at most $t-1$.

It is known that the condition $k\ge2t$ is sharp, in the sense that the $(t,t)$-star is no longer an extremal example if $k<2t$. For example, Frankl~\cite{frankl1983Steiner} showed that when $k=2t-1$ and $k-t$ is a prime, $f(n,k,t)$ is attained by the Steiner system $S(n,2k-t-1,t)$, if such a system exists for these values of $n,k,t$.

In 2006, Keevash, Mubayi and Wilson~\cite{keevash2006set} solved the problem completely for $k=4,t=2$, and any value of $n$.
The problem of determining the minimal value of $n_{0}=n_{0}\left(k,t\right)$ such that for any $n>n_0$ the size of the extremal family is
$\left|\f\right|\le{{n-t}\choose{k-t}}$ is wide open.

Theorem~\ref{thm:exact solution for t-stars}, together with the Frankl-F\"{u}redi result~\cite{frankl1985forbidding}, imply the following:
\begin{cor}
For each $t\in\mathbb{N}$, there exists $C(t)$ such that the following holds. Let $n \in \mathbb{N}$, let $2t\le k\le n/C$, and let $\f \subseteq {{[n]}\choose{k}}$ be a family that does not contain two edges with intersection of size $t-1$. Then $|\f| \le {{n-t}\choose{k-t}}$, with equality
if and only if $\f$ is a $(t,t)$-star.
\end{cor}

\mn \emph{The Frankl-F\"{u}redi `special simplex' problem.} The \emph{special $d$-dimensional simplex} $\s_{d}$ is a hypergraph
that consists of $d+1$ sets $A_{1},\ldots,A_{d+1}$, such that $A_{i}\cap\left[d+1\right]=\left[d+1\right]\backslash\left\{ i\right\} $
for any $i\in\left[d+1\right]$, and the sets $\left\{ A_{i}\backslash\left[d+1\right]\right\} _{i\in\left[d+1\right]}$ are pairwise disjoint. The special simplex problem asks, what is the largest size of a family $\f \subseteq {{[n]}\choose{k}}$ that does not contain a copy of $\s_d$.

Frankl and F\"{u}redi, who posed the problem in 1987~\cite{frankl1987exact}, proved that if $k\ge d+3$ and $n\ge n_{0}\left(k,d\right)$,
then any $\s_{d}$-free family $\f \subseteq {{[n]}\choose{k}}$ satisfies $\left|\f\right|\le {{n-1}\choose{k-1}}$, with equality if and only if $\f$ is a $(1,1)$-star. They conjectured that the assertion holds for $k=d+1,d+2$ as well, but could prove the conjecture only for $d=2$.

In 1999, Cs\'{a}k\'{a}ni and Kahn~\cite{csakany1999homological} showed that one may take $n_{0}\left(3,2\right)=6$ using a homological approach. They conjectured that the Frankl-F\"{u}redi result holds for all $n \geq (d + 1)(k - d + 1)$. However, no progress on the problem was obtained since the Cs\'{a}k\'{a}ni-Kahn work.

Theorem~\ref{thm:exact solution for t-stars}, together with the Frankl-F\"{u}redi result~\cite{frankl1985forbidding}, imply the following:
\begin{cor}
\label{cor:Csakani-Kahn}For each $d \in \mathbb{N}$, there exists $C(d)$ such that the following holds. Let $n \in \mathbb{N}$, let $d+3\le k\le n/C$,
and let $\f \subseteq {{[n]}\choose{k}}$ be a family that does not contain a copy of the special simplex $\s_{d}$. Then
$\left|\f\right|\le {{n-1}\choose{k-1}}$, with equality if and only if $\f$ is a $(1,1)$-star.
\end{cor}

\mn \emph{The Tur\'{a}n hypergraph problem for paths and cycles.} A $t$-path is a hypergraph $P_t = \{e_1, e_2, \ldots, e_t\}$, such that $|e_i \cap e_j | = 1$ if $|j - i| = 1$ and $e_i \cap e_j = \emptyset$ otherwise. A $t$-cycle is obtained from a $(t-1)$-path $\{e_1, e_2, \ldots, e_{t-1}\}$ by adding an edge $e_t$ that shares one vertex with $e_1$, another vertex with $e_{t-1}$, and is disjoint from the other edges.
The Tur\'{a}n hypergraph problem for paths (resp., cycles)
asks, what is the maximum number $ex_k(n, P_t)$
(resp., $ex_k(n, C_t)$)
of edges in a family $\f \subseteq {{[n]}\choose{k}}$ that does not contain a $t$-path
(resp., $t$-cycle). We present here our result in the case of paths and compare it with previous work; the situation in the case of cycles is similar.

The Erd\H{o}s-Gallai theorem~\cite{erdHos1959paths} from 1959 shows that $ex_2(n, P_t) \leq \frac{t-1}{2}n$, and this is tight whenever $t | n$. In 1977, Frankl~\cite{frankl1977singleton} determined $ex_k(n, P_2)$ (which is identical to the `forbidding singleton intersection' problem) for $n$ sufficiently large as function of $k$. The next exact result on the problem was obtained only recently: In 2014, F\"{u}redi, Jiang and Seiver~\cite{furedi2014exact} determined $ex_k(n, P_t)$ for all $k \geq 4$, $t \geq 3$ and $n$ sufficiently large, using the Deza-Erd\H{o}s-Frankl `delta-system method'~\cite{deza1978intersection}. In an independent work from 2015, Kostochka, Mubayi, and Verstra\"{e}te~\cite{kostochka2015turan} presented the method of `random sampling from the shadow' and used it to obtain a similar result for $k \geq 3$, $t \geq 4$ (except for $(k,t)=(3,4)$), and $n$ sufficiently large. All these results (except for the Erd\H{o}s-Gallai theorem) apply only for $n$ sufficiently large as function of $k$.

Theorem~\ref{thm:OR}, together with the F\"{u}redi-Jiang-Seiver result~\cite{furedi2014exact}, imply the following:
\begin{cor}
\label{cor:paths}For each $t \geq 1$, there exists $C(t)$ such that the following holds. Let $n \in \mathbb{N}$, let $4 \le k\le n/C$,
and let $\f \subseteq {{[n]}\choose{k}}$ be a family that does not contain a copy of the path $P_{2t+1}$. Then
$\left|\f\right|\le {{n}\choose{k}} - {{n-t}\choose{k}}$, with equality if and only if $\f$ is a $(t,1)$-star.
\end{cor}
The case of paths of even length does not follow directly from Theorem~\ref{thm:OR} since the structure of the conjectured extremal examples in that case is more complex (see~\cite{kostochka2015turan}). It seems, however, that with some more technical effort a similar result for the even case can be obtained as well.




\mn \emph{Families of hypergraphs with no cross matching.} While our main results consider single families that are free of some forbidden configuration, in order to establish them we prove various results in the `cross' setting that studies families $\f_1,\f_2,\ldots,\f_h$ that are `together' free of some hypergraph $\h$ with $h$ edges. Formally, to treat this setting one defines an \emph{ordered} hypergraph to be an ordered tuple of edges. The $k$-expansion of an ordered hypergraph is defined accordingly, and families $\f_1,\f_2,\ldots,\f_h$ are said to be \emph{cross free} of the ordered hypergraph $\h$ if there exists no copy $(E_1,...,E_h)$ of $\h$ with $E_1\in \f_1,\ldots,E_h\in \f_h$. We note that the `cross' setting appears frequently in inductive proofs of statements in extremal hypergraph theory and numerous results in this setting were proved in the last decades (see, e.g.,~\cite{ahlswede1977contributions,frankl1992some,furedi1995cross} and the survey~\cite{frankl2016survey}).


One of the results we prove in this direction concerns a conjecture raised (independently) by Aharoni and Howard~\cite{aharoni2017rainbow} and by Huang, Loh and Sudakov~\cite{huang2012size} in 2012, as a strengthening of the Erd\H{o}s matching conjecture.

A $t$-matching is a $k$-uniform hypergraph that consists of $t$ pairwise disjoint edges.
\begin{conjecture}
\label{Conjecture: cross matching}Let $k\le n/t$, and let $\f_1,\ldots,\f_t \subseteq {{[n]}\choose{k}}$ be families that are cross free of an ordered $t$-matching. Then
\[
\min\left\{ \left|\f_{1}\right|,\ldots,\left|\f_{t}\right|\right\} \le \max \left({{n}\choose{k}}-{{n-t+1}\choose{k}}, {{kt-1}\choose{k}} \right).
\]
Equality is attained if all the families $\f_{1},\ldots,\f_{t}$ are equal either to the same $(t-1,1)$-star, or to the same $(kt-1,k)$-star.
\end{conjecture}
Huang, Loh, and Sudakov~\cite{huang2012size} proved their conjecture in the case $n\ge3tk^{2}$, and used this result to improve the state-of-the-art for the Erd\H{o}s matching conjecture. (The application to the matching conjecture was later superseded by a result of Frankl~\cite{frankl2013improved}).
Using our techniques (see Section~\ref{sec:Baby Case}), along with the Huang-Loh-Sudakov result~\cite{huang2012size}, we obtain:
\begin{cor}\label{cor:aharoni-howard}
Conjecture~\ref{Conjecture: cross matching} holds for all $k < n/C$, where $C$ depends only on $t$.
\end{cor}
The main tool here is a `cross-setting' version of Theorem~\ref{thm:Junta-approx-theorem} in the case where $\h$ is a matching.

\mn \emph{Problems in which the extremal example is a more general junta.} We anticipate that with additional effort, our junta method will be applicable not only in cases where the extremal example is a $(t,t)$-star or a $(t,1)$-star, but to any Tur\'{a}n problem for expansion in which the extremal example is a junta. An example of such a result is the following theorem, proved by the authors and David Ellis, that requires several additional methods and hence is presented in a separate paper~\cite{ellis2018stability}.
\begin{thm}
Denote by $\f_{n,k,t,r}$ the family $\{A \in {{[n]}\choose{k}}: |A \cap [t+2r]| \geq t+r\}$, as above.

For any $t \in \mathbb{N}$ and any $\epsilon>0$, there exist $C=C(t)$ and $n_0=n_0(t,\epsilon)$ such that the following holds for all $n>n_0$. Let $C \leq k \leq (\frac{1}{2}-\epsilon)n$ and let $\f \subseteq {{[n]}\choose{k}}$ be a family that does not contain two sets which intersect in exactly $t-1$ elements. Then $|\f| \leq \max_r (|\f_{n,k,t,r}|)$, with equality if and only if $\f$ is isomorphic to one of the families $\f_{n,k,t,r}$.
\end{thm}
Combined with the Frankl-F\"{u}redi result~\cite{frankl1985forbidding} that resolves the case $2t \leq k \leq C$, this provides a complete solution to the `forbidding one intersection' problem in the regime of a constant $t$, for almost all the range in which the conjectured extremal example is a junta.

\medskip Problems for which the conjectured extremal example is not a junta (like maximal hypergraphs that do not contain a clique or a copy of the Fano plane, see~\cite{keevash2011hypergraph}), seem out of reach for our method.

\subsection{Techniques}
\label{sec:sub:intro:techniques}

Before presenting an outline of our method, we briefly describe the main previous techniques that were used to study Tur\'{a}n problems for expansion. As we indicate, two of them are used in our argument as well.

\subsubsection{Overview of previous techniques}

\medskip \noindent \textbf{The delta-system method.} The arguably most common method for attacking Tur\'{a}n problems for expansion is the delta-system method, proposed by Deza, Erd\H{o}s and Frankl~\cite{deza1978intersection} and developed by F\"{u}redi~\cite{furedi1983kernel} and others. The driving force behind the method (in the way it is used today) is a lemma of F\"{u}redi~\cite{furedi1983kernel} which allows partitioning any $\h^+$-free family into several parts that have a rich structure (called `$(s,J)$-homogeneous hypergraphs') and a `remainder' that contains only a few edges. After the partition is performed, a two-step procedure is applied: First, the restriction on the original family is translated to a restriction on each of the $(s,J)$-homogeneous parts and on the interactions between them. Second, these restrictions are used to deduce an upper-bound on the sizes of the $(s,J)$-homogeneous parts, which eventually leads to an upper bound on the size of the original family.

The delta-system method was used by Frankl and F\"{u}redi (\cite{frankl1987exact}, see Theorem~\ref{thm:Frankl-Furedi} below) to derive an asymptotic upper bound on the size of $\h^+$-free families for any fixed forbidden hypergraph $\h$, and is being continuously used to find exact solutions to various specific problems (see, e.g.,~\cite{furedi2014linear,furedi2014hypergraph,furedi2015turan,furedi2014exact}).

\medskip \noindent \textbf{Random sampling from the shadow.} This method, introduced by Kostochka, Mubayi and Verstra\"{e}te~\cite{kostochka2015turan,kostochka2015problems,kostochka2014turan} in 2015, revolves around the idea of studying a family $\f \subset {{[n]}\choose{k}}$ by inspecting its lower shadow $\partial \f$, defined as the family of all sets of size $k-1$ that are contained in a set of $\f$. The family $\partial \f$ is multi-colored by $n$ colors, where a set $A \in \partial \f$ is colored by the set of colors $\{i\in [n]: A\cup \{i\} \in \f\}$. Then, in order to upper bound the size of a family $\f$ that is free from some forbidden structure $H$, the method proceeds by three steps. First it is shown that the restriction on $\f$ implies that $\partial \f$ is free of a certain multicolored configuration. This is used to upper bound the size of $\partial \f$ using Ramsey-type theorems. Finally, the upper bound on the size of $\partial \f$ is leveraged into an uppper bound on the size of $\f$.

The method of `random sampling from the shadow' was used by Kostochka, Mubayi and Verstra\"{e}te to attack the Tur\'{a}n problem for expansion in the cases of linear paths and cycles~\cite{kostochka2015turan}, trees~\cite{kostochka2015problems}, and graphs with crosscuts of size two~\cite{kostochka2014turan}.

\medskip \noindent \textbf{The stability method.} This method allows leveraging stability results (i.e., results which describe the structure of almost extremal families with respect to some Tur\'{a}n-type problem) into `exact' results (which determine the extremal families exactly). The stability method was introduced by Simonovits~\cite{simonovits1968method} and developed by Mubayi~\cite{mubayi2007intersection} and others. The method proceeds by three steps. First an approximate result is proved, which shows that any family that satisfies the restriction cannot be `much larger' than the desired bound. Then the approximate result is used to obtain a stability statement which characterizes the almost-extremal families. Finally, the stability result is leveraged to obtain an exact extremal result.

The stability method was used in numerous recent works. In particular, Keevash and Mubayi~\cite{keevash2010set} used it in their result on the Erd\H{o}s-Chv\'{a}tal conjecture, and Kostochka, Mubayi, and Verstra\"{e}te~\cite{kostochka2015turan} used it in their recent result on the aforementioned Tur\'{a}n hypergraph problem for paths and cycles.

\subsubsection{Outline of our technique}

\medskip \noindent \textbf{The junta method.} As mentioned above, our main tool is the `junta method', which takes its origin from the work of Dinur and Friedgut~\cite{dinur2009intersecting} on approximation of intersecting families by juntas. The heart of the junta method is the general structure theorem asserting that if a family $\f$ is free of $\h^+$ for some hypergraph $\h$ with $h$ edges, then $\f$ is essentially contained in a junta that is also free of $\h^+$. The proof of this structural result (Theorem~\ref{thm:Junta-approx-theorem} above), which spans most of the paper, can be very roughly sketched as follows (see Section~\ref{sec:overview} for a detailed proof-sketch).
\begin{itemize}
\item We consider the division of each family into a `junta' part and a `quasi-random' part which is far from being approximated by a junta (the formal notion for this is `uncapturable', introduced by Dinur and Friedgut~\cite{dinur2009intersecting}). We seek to approximate each $\h^+$-free family by its `junta' part. For this, we have to show that if $\f$ is $\h^+$-free then its `uncapturable' part is small, and its `junta' part is $\h^+$-free.

\item To prove this, we first show that any sufficiently large uncapturable family $\g$ contains a copy of $\h_1^+$, for \emph{any} hypergraph $\h_1$ with $h$ edges. This part requires defining two new notions of quasi-randomness (which we call `fairness' and `quasiregularity') and studying their inter-relations, as well as studying the relation between an $\h^+$-free family and its shadow (following the aforementioned method of `random sampling from the shadow'). We then deduce that if $\f$ is  $\h^+$-free, then its `junta' part is $\h^+$-free as well.
\end{itemize}

Once the general structure theorem is established, we use it to derive exact solutions of hypergraph Tur\'{a}n-type problems in a two-step procedure which generally follows the aforementioned `stability method':
\begin{enumerate}

\item We characterise all the largest juntas that are free of $\h^+$. This is the easiest step, as juntas are very convenient structures to work with.

\item We show that if an $\h^+$-free family $\f \subseteq {{[n]}\choose{k}}$ is a small perturbation of an extremal junta $\j$ , then $\left|\f\right|\le\left|\j\right|$. This is obtained using a `bootstrapping' lemma which asserts that if $\g_1,\ldots,\g_h$ are families that are cross free of a copy of $\h^+$, and in addition, $\g_1,\ldots,\g_{h-1}$ are `very large', then $\g_h$ must be `very small'.
\end{enumerate}
In words, Theorem~\ref{thm:Junta-approx-theorem} shows that the juntas are the largest $\h^+$-free families, up to some `noise'. Step~(1) lets us find an $\h^+$-free junta $\j$ of the largest size, and so by Theorem~\ref{thm:Junta-approx-theorem}, the largest $\h^+$-free hypergraph is a `noisy version' of $\j$. Then, Step~(2) tells us that any $\h^+$-free  noisy version of $\j$ has a smaller size than $\j$, and so $\j$ is the extremal family.

\subsection{Follow-up works}

Since the initial version of this paper appeared online, our methods were used in diverse settings, and gave rise to results in several directions. In this subsection we give a very brief description of three of these follow-up works.

\medskip \noindent \emph{`Forbidding one intersection' for permutations: relation to representation-theoretic techniques.} A family $\f$ of permutations, i.e., $\f \subset \mathcal{S}_{n}$, is called $t$-\emph{intersecting} if every two permutations in $\f$ agree on at least $t$-elements. In a seminal paper from 2011 that used representation theory of the symmetric group and algebraic methods in graph theory, Ellis, Friedgut and Pilpel~\cite{ellis2011intersecting} determined the largest $t$-intersecting families of permutations, provided that $n\ge n_{0}\left(t\right)$. The beautiful proof technique of~\cite{ellis2011intersecting} was shown by Ellis~\cite{ellis2014forbidding} to fail for the corresponding `forbidding intersection problem', for any $t \ge 3$.

In a recent work by Ellis and the second author~\cite{EL18+}, they combined the junta method with representation-theoretic methods to derive a significantly shorter and more robust proof of the main result of~\cite{ellis2011intersecting} and also to extend the result into a solution of the `forbidding one intersection' problem for permutations. They proved:
\begin{thm}[Ellis and Lifshitz]
\label{thm:extremal-forbidden}
For any $t \in \mathbb{N}$, there exists $n_0(t)$ such that the following holds for any $n \geq n_0$. Let $\f \subset \mathcal{S}_n$ be a $(t-1)$-intersection-free family. Then $|\f| \leq (n-t)!$, with equality if and only if $\f$ is a $(t,t)$-star (i.e., the set of all permutations that agree with a fixed bijection of size $t$).
\end{thm}

\medskip \noindent \emph{A `removal lemma' for expanded hypergraphs of large uniformity, via new sharp threshold results.} The celebrated hypergraph removal lemma, proved by Gowers~\cite{Gowers07} and independently by Nagle, R\"{o}dl, Schacht, and Skokan~\cite{NRS06,RS04}, asserts that for any fixed $k$ and for any fixed $k$-uniform hypergraph $\mathcal{H}$, if a $k$-uniform hypergraph $\f$ on $n$ vertices contains only few copies of $\mathcal{H}$, then it can be made $\mathcal{H}$-free by removing `few' of its edges. The removal lemma is derived from a regularity lemma for hypergraphs, similarly to the Ruzsa-Szemer\'{e}di triangle removal lemma~\cite{ruzsa1978triple} derived from Szemer\'{e}di's regularity lemma for graphs~\cite{szemeredi1975arithmetic}. While this settles the `hypergraph removal problem' in the case where $k$ and $\mathcal{H}$ are fixed, the result is meaningless when $k$ is large (say, when $k > \log \log \log n$). Friedgut and Regev~\cite{friedgut2018cheese} used eigenvalue techniques to obtain a hypergraph removal lemma for arbitrarily large $k$, in the special case where $\mathcal{H}$ consists of two pairwise disjoint edges.

In a recent work~\cite{lifshitz2018removal}, the second author combined the junta method with a novel sharp threshold theorem, to generalize the result of~\cite{friedgut2018cheese} into a removal lemma for a wide class of expanded hypergraphs, where $k$ is allowed to be as large as linear in $n$. The sharp threshold theorem of~\cite{lifshitz2018removal} is a robust version of the classical sharp threshold theorem of Friedgut and Kalai~\cite{friedgut1996every} which essentially asserts that for any monotone Boolean function $f$ which satisfies a certain symmetry condition, the expectation of $f$ with respect to the biased measure $\mu_p$ on the discrete cube increases rapidly from close to 0 to close to 1 within a small interval of $p$'s. In~\cite{lifshitz2018removal} it is shown that a variant of the Friedgut-Kalai theorem holds even if the function $f$ is not monotone but only almost monotone, in the sense that there are only few $x,y$ with $\forall i:x_i\le y_i$, such that $f(x)>f(y)$.

\medskip \noindent \emph{Tur\'{a}n-type results for expansion of large graphs, via sharp threshold results for sparse sets.} In all Tur\'{a}n-type results considered in this paper, while the uniformity of the expanded hypergraph (i.e., $k$) may grow up to linearly with $n$, the basic hypergraph $\mathcal{H}$ is  of constant size.

In a recent paper of Keevash, the second author, Long, and Minzer~\cite{KLLM18}, they showed that the junta method can be extended to cases where the size of $\mathcal{H}$ depends on $n$, and used it to show that the extension of the Erd\H{o}s matching conjecture~\cite{erdHos1965problem} to cross intersection (i.e., Conjecture~\ref{Conjecture: cross matching} above) holds for all $n \geq Ckt$, where $C$ is a universal constant. This generalizes results of Huang, Loh, and Sudakov~\cite{huang2012size} and of Frankl~\cite{frankl2013improved}. In order to extend the junta method, the authors of~\cite{KLLM18} developed new sharp threshold theorems for sparse families with respect to a biased product measure on the discrete cube, which generalize the classical sharp threshold theorem of Bourgain~\cite{friedgut1999sharp} and prove a conjecture of Kahn and Kalai~\cite{kahn2007thresholds}.

\medskip

A common feature of the two latter works is that both required developing new sharp threshold theorems. The `sharp threshold phenomenon', which effectively says that a Boolean function exhibits sharp threshold behavior unless it can be approximated by a junta, plays a central role in the junta method via the lemma of Dinur and Friedgut~\cite{dinur2009intersecting} (Lemma~\ref{lem:DF} below), whose proof relies on the Friedgut-Kalai sharp threshold theorem. Hence, it comes by no surprise that extensions of the junta method go hand in hand with new threshold results, and we anticipate that the interplay between these two classes of results will bring more advances in both directions.

\subsection{Organization of the paper}
After presenting definitions and notations in Section~\ref{sec:prelim}, we consider in Section~\ref{sec:Baby Case} a `baby case' of the general problem, in which the forbidden hypergraph is a matching. In this special case, we prove a cross-version of our general results, thus establishing Corollary~\ref{cor:aharoni-howard} above. We present this case first since on the one hand, its proof contains all components of the general proof, and on the other hand, each component is easier than in the general case, so that the proof is easier to follow. (In addition, components of the baby-case proof are used later in the general-case proof.) In Section~\ref{sec:overview} we present a detailed overview of the proof in the general case. Sections~\ref{sec:fairness}--\ref{sec:uncap-contains} are devoted to proving that any `large' uncapturable family contains a copy of $\h^+$, for any constant-size hypergraph $\h$. In Section~\ref{sec:bootstrapping} we establish the bootstrapping lemmas required in the last part of the proof. In Section~\ref{sec:proof} we collect all ingredients together to prove Theorems~\ref{thm:Junta-approx-theorem},~\ref{thm:exact solution for t-stars},~\ref{thm:Stab-for-Stars}, and~\ref{thm:OR}. Finally, the proof of the Erd\H{o}s-Chv\'{a}tal simplex conjecture is presented in Section~\ref{sec:Chvatal}.

\section{Preliminaries}
\label{sec:prelim}

In this section we introduce some notations and a few lemmas that will be used multiple times in the sequel.

\subsection{Notations}

The following notations will be used throughout the paper.

\mn \textbf{Alphabet.} As the paper contains quite a lot of `literal' notations, we tried to make these notations as consistent as possible. In particular:
\begin{itemize}
\item $n$ -- the `universe': All sets we consider in the paper are subsets of $[n]=\{1,2,\ldots,n\}$, where $n$ is a `very large' number and we are interested in the asymptotic dependence of the results on $n$.

\item $k$ -- most of the sets considered in the paper are of size $k$. Unlike most previous papers, $k$ can go to infinity with $n$ (and usually can be up to linear in $n$).

\item $r$ -- appears only in expressions like $\epsilon^r$ and $\left(\frac{k}{n}\right)^r$, which always denote `a negligibly small quantity'. We usually show that claims hold for any value $r \in \mathbb{N}$ (with some parameters depending on $r$) and sometimes use the ability to apply the statements with different values of $r$.

\item $h$ -- this is always the number of edges in the forbidden hypergraph $\h$.

\item $t,d$ -- usually denote the `main' parameter of $\h$, such as the number of edges in a forbidden $t$-matching, or the `dimension' of a forbidden simplex.

\item $C,s,m,u,v$ -- denote auxiliary parameters, which can be large (depending on parameters like $h,t,r$) but are always constant, i.e., do not tend to infinity with $n$.

\item $\zeta,\delta$ -- denote auxiliary parameters, which can be small but do not tend to zero as $n \rightarrow \infty$.

\item $i,j,l$ -- denote indices.
\end{itemize}

\mn \textbf{$k$-element sets.} For a set $S \subseteq [n]$ and $k \leq |S|$, we denote by ${{S}\choose{k}}$ the family of all $k$-element subsets of $S$. For a set family $\f \subseteq [n]$, the sub-family of $\f$ that consists of all $k$-element sets is denoted by $\f^{(k)} = \{A \in \f: |A|=k\}$. We sometimes abbreviate `$k$-element set' to `$k$-set'.

\mn \textbf{Measure.} For a family $\f$ of $k$-element sets, i.e., $\f \subseteq {{[n]}\choose{k}}$, the \emph{measure} of $\f$ is said to be $\mu(\f) = \frac{|\f|}{{{n}\choose{k}}}$ (which means that we consider the uniform measure on the set ${{[n]}\choose{k}}$).

\mn \textbf{A frequently used measure lower bound.} Multiple assertions in the paper hold for families $\f$ that satisfy
\[
\mu(\f) \geq \max\left( e^{-k/C},C\frac{k}{n}\right),
\]
where $C$ is a fixed constant. We note that with respect to this bound, $e^{-k/C}$ is the dominant term when $k\le C\log n$, while the term $C\frac{k}{n}$ is dominant when $k>(C+\epsilon)\log n$, for any $\epsilon>0$ and sufficiently large $n$.

\mn \textbf{Biased measure.} For a family $\f \subseteq \pn$ and for $0<p<1$, the $p$\emph{-biased measure }of $\f$ is defined by
\[
\mu_{p}\left(\f\right)=\sum_{A\in\f}p^{\left|A\right|}\left(1-p\right)^{n-\left|A\right|}.
\]
Intuitively, the measure $\mu_{p}$ defines a probability distribution on sets $A\subseteq\left[n\right]$, where each element in $\left[n\right]$ is chosen to be in $A$ independently at random with probability $p$. The measure $\mu_{p}\left(\f\right)$ is the probability that $A \in \f$.

\mn \textbf{Monotone families.} A family $\a \subseteq \pn$ is called \emph{monotone} if $(A \in \a \wedge A \subseteq B) \Rightarrow B \in \a$. The \emph{monotone closure} $\f^{\uparrow}$ of a family $\f\subseteq\pn$ is the family of all sets in $\pn$ that contain some element of $\f$. That is, $\f^{\uparrow} = \{S \subseteq [n]: \exists A \in \f, A \subseteq S\}$. Obviously, this family is monotone.

\mn \textbf{Complexes.} A family $\c \subseteq\p\left(n\right)$ is said to be a \emph{simplicial complex} (or, in short, a complex) if
$(A \in \a \wedge B \subseteq A) \Rightarrow B \in \a$. We note that such families are also called \emph{down-sets}. The \emph{down-closure} $\f_{\downarrow}$ of a family $\f\subseteq\pn$ is the family of all sets in $\pn$ that are included in an element of $\f$. That is, $\f_{\downarrow} = \{S \subseteq [n]: \exists A \in \f, S \subseteq A\}$. Obviously, this family is a complex.

\mn \textbf{Generalized Binomial coefficients.} For $x \in \mathbb{R}$ and $k\in\mathbb{N}$, we denote ${{x}\choose{k}} = \frac{x\left(x-1\right)\cdots\left(x-k+1\right)}{k!}$.

\mn \textbf{Random variables notation.} Throughout the paper, we use bold letters to denote random variables and regular letters to denote fixed (i.e., non-random values). For example, $\mathbf{S}=S$ means that the randomly chosen set $\mathbf{S}$ equals a fixed set $S$.

\mn \textbf{Uniform sampling.} For a set $S$ (e.g., $S = {{[n]}\choose{k}}$), the notation $\mathbf{X} \sim S$ means that $\mathbf{X}$ is drawn uniformly from the elements of $S$.

\mn \textbf{Asymptotic notation.} For functions $f,g: \mathbb{R}_{+} \rightarrow \mathbb{R}_{+}$ and a parameter $d$, the notation $f=O_d(g)$ means that there exists a constant $C$ that depends only on $d$ such that $f(x) \leq Cg(x)$ for all $x>0$.

\mn \textbf{Disjoint union.} We sometimes use the notation $A \sqcup B$ for the union of $A$ and $B$, where $A,B$ are known to be disjoint.

\subsection{Hypergraph notations}

While the `forbidden' hypergraph in our paper is always simple, it will be convenient for us to use the notion of
multi-hypergraphs.

\mn \textbf{Multi-hypergraphs.} A \emph{multi-hypergraph} is a hypergraph that may contain multiple edges. Given a hypergraph $\h'$, we define the multi-hypergraph $s\cdot\h'$ to be the hypergraph that contains each edge of $\h'$ $s$ times.

\mn \textbf{$k$-expansion of a multi-hypergraph.} Let $\h_1$ be a multi-hypergraph, and suppose that all the edges of $\h_1$ are of size at most $k$ . The $k$-\emph{expansion} of $\h_1$ is the hypergraph $\h_1^{+}\subseteq{{[n]}\choose{k}}$ obtained from $\h_1$ be enlarging each of its edges by adding to it distinct new vertices. For example, the hypergraph $\left(t\cdot\emptyset\right)^{+}$ is a matching, and the hypergraph \emph{$\left(2 \cdot \left\{ \left[t-1\right]\right\} \right)^{+}$} has two hyperedges, which are $k$-element sets whose intersection is of size $t-1$.\emph{ }

\mn \textbf{$d$-expanded hypergraph} We say that a hypergraph $\h\subseteq{{[n]}\choose{k}}$ is $d$-\emph{expanded} if it is the expansion of a $d$-uniform multi-hypergraph. For example, matchings are the only $0$-expanded hypergraphs, and the hypergraph $\left(2 \cdot \left\{ \left[t-1\right]\right\} \right)^{+}$ is $\left(t-1\right)$-expanded.

It is important to note that the following are equivalent:
\begin{itemize}
\item The hypergraph $\h\subseteq{{[n]}\choose{k}}$ is the expansion of a `fixed' hypergraph (in the sense defined in the introduction);

\item There exist constants $d,h$, such that $\h$ is a $d$-expanded hypergraph with $h$ edges.
\end{itemize}
Hence, the multi-hypergraph notation allows us to have the base hypergraph (whose expansion is forbidden) uniform.

\mn \textbf{Slight abuse of notation.} In the introduction, we denoted the base hypergraph by $\h$, and its $k$-expansion -- which is the actually forbidden
hypergraph -- by $\h^+$. For sake of simplicity, in the sequel we use the notation $\h$ for the forbidden $k$-uniform hypergraph, and usually denote the
base multi-hypergraph by $\h_1$, so that $\h=\h_1^+$. In addition, we sometimes use the term \emph{$d$-expanded hypergraph}, without specification of $d$, to denote a forbidden hypergraph of the form $\h=\h_1^+$, where $\h_1$ is \emph{constant-sized}. Formally, this means that for any fixed $m \in \mathbb{N}$, the claims hold for all hypergraphs $\h=\h_1^+$ with $\max(|E(\h_1)|,|C(\h_1)|) \leq m$, where the parameters in the assertions depend on $m$ (see, e.g., Theorem~\ref{thm:exact solution for t-stars}).

\mn \textbf{Degree and center.} The \emph{degree} of a vertex $v$ in a hypergraph $\h$ is the number of edges that contain $v$. The \emph{center} of $\h$ is the set of all vertices of $\h$ of degree at least 2.

\mn \textbf{Kernel.} The \emph{kernel} of a hypergraph $\h$, denoted by $K(\h)$, is the set of vertices that belong to all edges of $\h$.

\mn \textbf{Size.} The size of a hypergraph $\h$ is the number of its edges.

\subsection{The `cross' setting}

While our main questions concern a single family, in several steps of the proof we have to consider the so-called `cross' version of the problems.

\mn \textbf{Cross freeness.} Let $\mathcal{T}=\left(A_{1},\ldots,A_{h}\right)$ be an \emph{ordered} hypergraph, and let $\f_{1}\subseteq {{\left[n\right]}\choose{|A_1|}},\f_{2}\subseteq {{\left[n\right]}\choose{|A_2|}},\ldots,\f_{h}
\subseteq {{\left[n\right]}\choose{|A_h|}}$ be set families. We say that the families $\f_{1},\f_{2},\ldots,\f_{h}$
are \emph{cross free} of $\mathcal{T}$ if there is no copy of $\mathcal{T}$ of the form $\left(B_{1},\ldots,B_{h}\right)$, such that
$B_{i} \in \f_{i}$ for all $i \in [h]$.

\mn \textbf{Cross expansion.} Let $\mathcal{T}=\left(A_{1},\ldots,A_{h}\right)$ be an ordered hypergraph, and suppose that $\left|A_{i}\right|\le k_{i}$ for
all $i\in\left[h\right]$. We define the $\left(k_{1},\ldots k_{h}\right)$-expansion of $\mathcal{T}$ analogously to the $k$-expansion of a single hypergraph.
Namely, we define its edge set to be of the form $\left(A_{1}\cup D_{1},\ldots,A_{h}\cup D_{h}\right)$, for pairwise disjoint sets $D_{1},\ldots,D_{h}$ that are disjoint from all the sets in $\mathcal{T}$, such that $|A_i \cup D_i|=k_i$, for all $i \in [h]$.

\subsection{Slices}

Let $\f\subseteq\pn$ be a family. Then any set $S\subseteq\left[n\right]$ induces a partition of $\f$ into $2^{\left|S\right|}$ \emph{slices} $\left\{ \f_{S}^{B}\right\} _{B\subseteq S}$, according to the intersection with $S$:
\[
\f_{S}^{B}:=\left\{ A\backslash B\,:\, (A\in\f) \wedge (A\cap S=B)\right\}.
\]
We treat the families $\f_{S}^{B}$ as residing in the `universe' ${{[n] \setminus S}\choose{k-|B|}}$.
For example, if $\f$ is the family
\[
\mathrm{Maj}_{\left[3\right]}=\left\{ A\in{{[n]}\choose{k}}\,:\,\left|A\cap\left[3\right]\right|\ge2\right\} ,
\]
then $\f_{\left\{ 1\right\} }^{\left\{ 1\right\} }$ is the $\left(2,1\right)$-star
\[
\left\{ A\in {{\left[n\right]\backslash\left\{ 1\right\} }\choose{k-1}}\,:\, A\cap\left\{ 2,3\right\} \ne\emptyset\right\} ,
\]
 and $\f_{\left\{ 1\right\} }^{\emptyset}$ is the $\left(2,2\right)$-star
\[
\left\{ A\in {{\left[n\right]\backslash\left\{ 1\right\} }\choose{k}}\,:\,\left\{2,3\right\}\subseteq A\right\} .
\]


\subsection{Juntas}

For a set $J \subseteq [n]$ (which will usually be of `constant' size), for a family $\j \subseteq \p(J)$, and for $k \geq |J|$, the \emph{$k$-uniform junta} generated by $\j$ is
\[
\left\langle \j\right\rangle = \{A\in{{[n]}\choose{k}}: A\cap J\in\j\}.
\]
Any family $\j\subseteq\p\left(J\right)$ gives rise to two natural $k$-uniform juntas. The first is $\left\langle \j\right\rangle$,
and the second is $\left(\j^{\uparrow}\right)^{\left(k\right)}$. Of course, we always have $\left\langle \j\right\rangle \subseteq\left(\j^{\uparrow}\right)^{\left(k\right)}$.
The following lemma asserts that
\[
\mu\left(\left(\j^{\uparrow}\right)^{\left(k\right)}\backslash\left\langle \j\right\rangle \right)=o\left(\mu\left(\left\langle \j\right\rangle \right)\right).
\]

\begin{lem}
\label{lem:measures of juntas}Let $J \subseteq [n]$, where $j=|J|$ is a constant. Let $\j\subseteq\p\left(J\right)$ be a family, and let $l<k$ be the minimal size of an element of $\j$. Then
\begin{equation}
\left|\left(\j^{\uparrow}\right)^{\left(k\right)}\right|=\left|\j^{\left(l\right)}\right|{{n}\choose{k-l}}+O_{j}\left(\frac{k}{n}\right)^{l+1}{{n}\choose{k}},\label{eq:upper junta measure}
\end{equation}
 and
\begin{equation}
\mu\left(\left(\j^{\uparrow}\right)^{\left(k\right)}\backslash\left\langle \j\right\rangle \right)=O_{j}\left(\frac{k}{n}\right)^{l+1}.\label{eq:lower junta measure}
\end{equation}
\end{lem}
\begin{proof}
We prove only (\ref{eq:upper junta measure}) as (\ref{eq:lower junta measure}) is similar. We have
\[
\left|\left(\j^{\uparrow}\right)^{\left(k\right)} \right|=\sum_{\{A\subseteq J\,:\, |A|\ge l+1\}}\left|\left(\left(\j^{\uparrow}\right)^{\left(k\right)}\right)_{J}^{A}\right|+\sum_{\{A\subseteq J\,:\, |A|=l\}}\left|\left(\left(\j^{\uparrow}\right)^{\left(k\right)}\right)_{J}^{A}\right|.
\]
Hence, the assertion of the lemma is an immediate consequence of the following easy inequalities:
\[
\sum_{\{A\subseteq J\,:\,\left|A\right|\ge l+1\}}\left|\left(\left(\j^{\uparrow}\right)^{\left(k\right)}\right)_{J}^{A}\right| = \sum_{\{A\subseteq J\,:\,\left|A\right|\ge l+1\}}{{n-j}\choose{k-\left|A\right|}}=O_{j}\left(\frac{k}{n}\right)^{l+1}{{n}\choose{k}},
\]
and
\[
\sum_{A\in\j^{\left(l\right)}}\left(\left(\j^{\uparrow}\right)^{\left(k\right)}\right)_{J}^{A}=\left|\j^{\left(l\right)}\right|{{n-j}\choose{k-l}}=
\left|\j^{\left(l\right)}\right|{{n}\choose{k-l}}+O_{j}\left(\frac{k}{n}\right)^{l+1}{{n}\choose{k}}.
\]

\end{proof}

\subsection{Shadows and the Kruskal-Katona Theorem}

The \emph{shadow} of a family $\f\subseteq{{[n]}\choose{k}}$, denoted by $\partial\left(\f\right)$, is the family $\f_{\downarrow}^{\left(k-1\right)}$ of all $(k-1)$-sets that are contained in some element of $\f$. Sets in $\partial\left(\f\right)$ are called subedges of $\f$. Similarly, the $t$-shadow of $\f$, denoted by $\partial^{t}\left(\f\right)$, is the family $\f_{\downarrow}^{\left(k-t\right)}$.

The classical Kruskal-Katona theorem allows to obtain bounds on the size of $\f$ in terms of the size of its shadows. In this subsection we present two such bounds that will be useful for us in the sequel.

\mn We start with two classical corollaries of the Kruskal-Katona theorem.
\begin{thm}[Lov\'{a}sz]
\label{Theorem:Kruskal-Katona-Lovasz} Let $\f\subseteq{{[n]}\choose{k}}$ be a family, and let $x \in \mathbb{R}$ be such that $\left|\f\right|={{x}\choose{k}}$. Then $\left|\partial\left(\f\right)\right|\ge {{x}\choose{k-1}}$.
\end{thm}
\begin{prop}
\label{cor:standard Kruskal-Katona}Let $\f\subseteq{{[n]}\choose{k}}$
be a monotone family, let $t\in\n$, and let $k'>k$.
\begin{enumerate}
\item If $\left|\f^{\left(k\right)}\right|\ge{{n-t}\choose{k-t}}$, then $\left|\f^{\left(k'\right)}\right|\ge {{n-t}\choose{k'-t}}$.
\item If $\left|\f^{\left(k\right)}\right|\ge{{n}\choose{k}}- {{n-t}\choose{k}}$,
then $\left|\f^{\left(k'\right)}\right|\ge{{n}\choose{k'}}-{{n-t}\choose{k'}}$.
\end{enumerate}
\end{prop}
The following two lemmas translate Theorem \ref{Theorem:Kruskal-Katona-Lovasz}
and Proposition \ref{cor:standard Kruskal-Katona} into statements about
measures of monotone families, that will be more convenient for us to work
with.

The first lemma bounds the measure of a family $\f\subseteq{{[n]}\choose{k}}$
in terms of the measure of the family $\left(\f_{\downarrow}\right)^{\left(d\right)}$,
for a constant $d$.
\begin{lem}
\label{lem:Kruskal Katona going down}For any constants $d,r$, there
exists a constant $C=C(d,r)$, such that the following holds.
Let $C<k<n/C$, and let $\f\subseteq{{[n]}\choose{k}}$
be a family. Suppose that $\mu\left(\left(\f_{\downarrow}\right)^{\left(d\right)}\right)\le\epsilon$.
Then $\mu\left(\f\right)\le O_{d,r}\left(\epsilon^{r}\right)$.
\end{lem}
Note that Lemma \ref{lem:Kruskal Katona going down} asserts that $\mu(\f)$ is significantly smaller than the
measure of the family $\left(\f_{\downarrow}\right)^{\left(d\right)}$.

In the proof of the lemma we use the following classical result, that can be found in~\cite[Sec.~2.1]{Anderson89}
\begin{prop}[Local LYM Inequality]
Let $\g\subseteq\pn$ be a complex. Then $\mu\left(\g^{\left(l\right)}\right)\ge\mu\left(\g^{\left(k\right)}\right)$,
for any $l\le k$.
\end{prop}

\noindent \emph{Proof of Lemma~\ref{lem:Kruskal Katona going down}. }
Write $\f_{i}=\left(\f_{\downarrow}\right)^{\left(i\right)}$, for
any $i\le k$, and write
\[
\left|\f_{d}\right|={{x}\choose{d}}.
\]
Note that ${{x}\choose{d}} = \Omega_{d} \left(\left(\frac{x}{n}\right)^{d} \right){{n}\choose{d}}$. On the other hand, by a repeated application of Theorem \ref{Theorem:Kruskal-Katona-Lovasz}, we have
\[
\left|\f_{dr}\right|\le {{x}\choose{dr}}=O_{d,r}\left(\left(\frac{x}{n}\right)^{dr} \right){{n}\choose{dr}}.
\]
Hence, to complete the proof, it is sufficient to show that $\mu\left(\f\right)\le\mu\left(\f_{dr}\right)$.
This follows immediately from the Local LYM inequality, substituting $\g=\f_{\downarrow}$.

\medskip The second lemma considers a monotone family $\f$ and bounds $\mu(\f^{\left(k\right)})$ from below in terms of the measure $\mu\left(\f^{\left(l\right)}\right)$ for
$l\le k$. This bound applies in a different regime: we require that $\frac{k}{n}$ is bounded away from $0$ and $1$, and on the other hand, we allow $l$ to be as large as linear in $n$ (instead of the constant value of $d$ in Lemma \ref{lem:Kruskal Katona going down}).
\begin{lem}
\label{lem:Kruskal-Katona going up} For any constants $\zeta>0$ and $r \in \mathbb{N}$,
there exists a constant $C(\zeta,r)>0$, such that the following holds. Let
$\zeta n<k\le\left(1-\zeta\right)n$, let $\epsilon>0$ and let $l<k/C$.
Suppose that $\f\subseteq\pn$ is a monotone family that satisfies $\mu\left(\f^{\left(k\right)}\right)\le\epsilon$.
Then $\mu\left(\f^{\left(l\right)}\right)\le O_{\zeta,r}\left(\epsilon^{r}\right)$. \end{lem}
\begin{proof}
We may assume that $\epsilon\le\frac{k}{n}$, for otherwise the lemma
holds trivially. So suppose that $\epsilon\le\frac{k}{n}$, and let
$t\in\mathbb{N}$ be such that ${{n-t-1}\choose{k-t-1}}\le\left|\f^{\left(k\right)}\right|\le{{n-t}\choose{k-t}}$.
By Proposition~\ref{cor:standard Kruskal-Katona}, we have $\left|\f^{\left(l\right)}\right|\le {{n-t}\choose{l-t}}$.
Hence, if $t>l$ then $\f^{\left(l\right)}=\emptyset$, and the lemma holds trivially. Suppose that $t \leq l$. Provided that $C$ is sufficiently large, we obtain
\begin{align*}
\epsilon & \ge\mu\left(\f^{\left(k\right)}\right)\ge\frac{{{n-t-1}\choose{k-t-1}}}{{{n}\choose{k}}}=\left(\frac{k}{n}\right)\cdots\left(\frac{k-t}{n-t}\right)\\
 & \ge\left(\frac{k-t}{n-t}\right)^{t+1}\ge\left(\frac{k/2}{n}\right)^{t+1}>\left(\frac{l}{n}\right)^{\frac{t}{r}}.
\end{align*}
 On the other hand, we have
\begin{align*}
\mu\left(\f^{\left(l\right)}\right) \le\frac{{{n-t}\choose{l-t}}}{{{n}\choose{l}}}=\left(\frac{l}{n}\right)\cdots\left(\frac{l-t+1}{n-t+1}\right) \le\left(\frac{l}{n}\right)^{t}<\epsilon^{r}.
\end{align*}
 This completes the proof of the lemma.
\end{proof}

\subsection{Application of analysis of Boolean functions}

As written above, the junta method stems from analysis of Boolean functions and relies upon several Fourier-analytic results. We use these results via a lemma, essentially proved by Dinur and Friedgut~\cite[Lemma~3.2]{dinur2009intersecting}. The lemma requires the definition of capturability that will be discussed thoroughly in Section~\ref{sec:Baby Case}.
\begin{defn}
Let $s\ge0$ be an integer, and let $\epsilon\in\left(0,1\right)$. A family $\f\subseteq{{[n]}\choose{k}}$ is said to be $\left(s,\epsilon\right)$-capturable, if there exists a set $S$ of size at most $s$, such that $\mu\left(\f_{S}^{\emptyset}\right)\le\epsilon$.
Otherwise, it is said to be $\left(s,\epsilon\right)$-uncapturable.
\end{defn}

\begin{lem}[Dinur and Friedgut]
\label{lem:DF}For any constants $\zeta\in\left(0,\frac{1}{2}\right)$ and $r\in\mathbb{N}$,
there exists a constant $s\left(\zeta,r\right)$ such that the following
holds. Let $n,k\in\mathbb{N}$ and $p\in\left(\zeta,1\right)$ be
numbers such that $\frac{k}{n}\le\frac{p}{2}$, and let $\f\subseteq{{[n]}\choose{k}}$
be a family that satisfies $\mu_{p}\left(\f^{\uparrow}\right)\le1-\zeta$.
Then $\f$ is $\left(s,\left(\frac{k}{n}\right)^{r}\right)$-capturable.
\end{lem}
Since the statement of the lemma we use is rather different from its statement in~\cite{dinur2009intersecting}, and as the lemma plays a central role in our argument, we provide its proof for the sake of completeness.

The proof uses the notion of \emph{influences} (which is one of the central notions in analysis of Boolean functions), and its relation to threshold phenomena. Note that each vector $x=(x_1,\ldots,x_n) \in \{0,1\}^n$ naturally corresponds to the set $\{i:x_i=1\} \subset [n]$, and thus, a Boolean function $f:\{0,1\}^n \to \{0,1\}$ naturally corresponds to the family $\{x: f(x)=1\} \subseteq \p([n])$.
\begin{defn}
Let $f:\{0,1\}^n \to \{0,1\}$ be a Boolean function and let $0<p<1$. The influence of the $i$'th coordinate on $f$ with respect to the biased measure $\mu_p$ is
\[
I^p_i[f]= \Pr_{x \sim \mu_p}[f(x) \neq f(x \oplus e_i)],
\]
where $x \oplus e_i$ is obtained from $x$ by flipping the $i$'th coordinate. The total influence of $f$ is the sum of its influences, namely,
\[
I^p[f] = \sum_{i=1}^n I^p_i[f].
\]
\end{defn}

\begin{lem}[Margulis~\cite{Mar74}, Russo~\cite{Rus81}]\label{Lem:Russo}
Let $f:\{0,1\}^n \to \{0,1\}$ be a monotone non-decreasing function. Then
\[
\frac{d \mu_p(\{x:f(x)=1\})}{dp} = I^p[f].
\]
\end{lem}
The lemma implies that if $f$ is a monotone Boolean function and $p_0$ is chosen such that $\mathbb{E}_{x \sim \mu_{p_0}}[f(x)]=1/2$, then $f$ has a \emph{coarse threshold} (i.e., the increase of $\mathbb{E}_{x \sim \mu_{p}}[f(x)]$ from close to $0$ to close to $1$ as $p$ increases is not sharp) if and only if the total influence $I^p[f]$ is small for $p \sim p_0$.

\medskip The notion of capturability has a natural analogue for monotone Boolean functions.
\begin{defn}
For $p \in (0,1)$ and a monotone function $f:\{0,1\}^n \to \{0,1\}$, we say that $f$ is $(s,\epsilon)$-capturable with respect to $\mu_p$ if there exists a set $J \subset [n]$ of size $\le s$ such that
\[
\mu_{p}\left[f_{J\to\overline{0}}\right]:= \mathbb{E}_{x_2 \sim\left(\left\{ 0,1\right\} ^{\left[n\right]\setminus J},\mu_{p}\right)}[f(0,x_2)]
<\epsilon.
\]
\end{defn}
The following classical theorem of Friedgut~\cite{friedgut1998boolean} asserts that if the total influence $I^p[f]$ is small for $p \sim p_0$, then $f$ can be well-approximated with respect to $\mu_p$ by a constant-sized junta, and thus, is \emph{capturable}.
\begin{thm}[Friedgut's junta theorem]\label{Thm:F-Junta}
For any $\epsilon,\zeta>0$ and for any $K$, there exists $j \in \mathbb{N}$, such that
the following holds. Let $\zeta<p<1-\zeta$, and let $f:\{0,1\}^n \to \{0,1\}$ satisfy $I^{p}\left[f\right]\le K$. Then
there exists a $j$-junta $g:\{0,1\}^n \to \{0,1\}$ such that
\[
\Pr_{x\sim\mu_{p}}\left[f\left(x\right)\ne g\left(x\right)\right]<\epsilon.
\]
\end{thm}
The following corollary allows deducing capturability from a coarse threshold assumption in a single step.
\begin{cor}
\label{cor:Friedgut junta}For any $\epsilon,\zeta'>0$ there exists $j=j(\epsilon,\zeta') \in \mathbb{N}$, such that the following holds. Let $p,q \in (0,1)$ satisfy $\zeta'<q<\left(1-\zeta'\right)p.$ Let $f\colon\left\{ 0,1\right\} ^{n}\to\left\{ 0,1\right\} $
be monotone and suppose that $\mathbb{E}_{x \sim \mu_{q}}\left[f(x)\right]>\zeta'$ and that
$\mathbb{E}_{x \sim \mu_{p}}\left[f(x)\right]<1-\zeta'.$ Then $f$ is $(s,\epsilon)$-capturable with respect to $\mu_q$.
\end{cor}
Note that the assumption $(\mathbb{E}_{x \sim \mu_{q}}\left[f(x)\right]>\zeta') \wedge (\mathbb{E}_{x \sim \mu_{p}}\left[f(x)\right]<1-\zeta')$
for $q,p$ bounded away from each other, is a coarse threshold assumption. Thus, the corollary tells us that a coarse threshold implies capturability.

\begin{proof}
Let $\delta=\delta\left(\epsilon,\zeta'\right)$ be sufficiently small.
By Lemma~\ref{Lem:Russo} and the mean value theorem, there exists $\bar{p} \in\left[q,\frac{p+q}{2}\right]$ such that
\[
I^{\bar{p}}\left[f\right]\le\frac{2}{p-q}\le\frac{2}{(\zeta')^2}.
\]
By Theorem~\ref{Thm:F-Junta} there exist $j=j(\zeta',\delta)=j(\epsilon,\zeta')$, a set $J$ of size $\le j$,
and a function $g\colon\left\{ 0,1\right\}^n \to\left\{ 0,1\right\}$ that depends only on the coordinates in $J$,
such that
\begin{equation}
\Pr_{x\sim\mu_{\bar{p}}}\left[f\left(x\right)\ne g\left(x\right)\right]<\delta.\label{eq:closeness}
\end{equation}
(Note that the theorem can be applied since $\zeta' \leq \bar{p} \leq 1-\zeta'$ by the choice of $\bar{p}$.)
Using the monotonicity of $f$, Equation~\eqref{eq:closeness}, and the fact that the function $t \mapsto \mu_{t}\left(f\right)$ is monotone non-decreasing, we obtain
\begin{align*}
\delta>\Pr_{x\sim\mu_{\bar{p}}}\left[f\left(x\right)\ne g\left(x\right)\right] & \ge\Pr_{x_{1}\sim\left(\left\{ 0,1\right\} ^{J},\mu_{\bar{p}}\right),x_{2}\sim\left(\left\{ 0,1\right\} ^{\left[n\right]\setminus J},\mu_{\bar{p}}\right)}\left[g\left(x_{1},x_2\right)=0,f\left(x_{1},x_{2}\right)=1\right]\\
 & \ge\Pr_{x_{1}\sim\left(\left\{ 0,1\right\} ^{J},\mu_{\bar{p}}\right),x_{2}\sim\left(\left\{ 0,1\right\} ^{\left[n\right]\setminus J},\mu{}_{\bar{p}}\right)}\left[g\left(x_{1},x_2\right)=0,f\left(0,x_{2}\right)=1\right]\\
 & =\left(1-\mu_{\bar{p}}\left(g\right)\right)\mu_{\bar{p}}\left(f_{J\to\overline{0}}\right)\\
 & \ge\left(1-\mu_{\bar{p}}\left(f\right)-\delta\right)\mu_{\bar{p}}\left(f_{J\to\overline{0}}\right)\\
 & \ge\left(\zeta'-\delta\right)\mu_{q}\left(f_{J\to\overline{0}}\right),
\end{align*}
where the equality in the middle holds since $g$ depends only on $x_1$ and $x_1,x_2$ are independent. Rearranging, we obtain
\[
\mu_{q}\left(f_{J\to\overline{0}}\right)\le\frac{\delta}{\zeta'-\delta}<\epsilon,
\]
provided that $\delta$ is sufficiently small.
\end{proof}

In addition, we need the following claim from~\cite{dinur2009intersecting}, which follows easily from the Kruskal--Katona theorem and a Chernoff bound.
\begin{lem}[\cite{dinur2009intersecting}, Claim 2.5]
\label{lem:claim 2.5} Let $n,k,r$ be integers such that $r<k<n$, and let $p,\eta \in (0,1)$ satisfy $p\left(1-\eta\right)\ge\frac{k-r}{n-r}.$
Let $\mathcal{F}\subseteq\binom{\left[n\right]}{k}$ be a family with
$\left|\mathcal{F}\right|\ge\binom{n-r}{k-r}$. Then
\[
\mu_{p}\left(\mathcal{F}^{\uparrow}\right)>p^{r}\left(1-\mathrm{exp}\left(-\eta^{2}pn/2\right)\right).
\]
\end{lem}

Now we are ready to present the proof of Lemma~\ref{lem:DF}.

\begin{proof}[Proof of Lemma~\ref{lem:DF}]
Let $\zeta,r,n,k,p$ satisfy the assumption of the lemma, and let $\f\subseteq{{[n]}\choose{k}}$
be a family that satisfies $\mu_{p}\left(\f^{\uparrow}\right)\le1-\zeta$. We want to show that $\f$ is $\left(s,\left(\frac{k}{n}\right)^{r}\right)$-capturable, for some $s=s(\zeta,r)$.

Set $q=\frac{p}{1.5},\epsilon=q^{r+2}, \zeta'=\min(q^{r+2},\zeta/1.5),$ and $\eta=0.01.$
Let $j=j\left(\zeta',\epsilon\right)$ be sufficiently large for Corollary \ref{cor:Friedgut junta} to hold. Furthermore, assume that $n$ is sufficiently large with respect to $j,r$, such that we have
\begin{equation}\label{Eq:Aux_revised1}
q^{r+1}\left(1-\mathrm{exp}\left(-\eta^{2}q\left(n-j\right)/2\right)\right)>q^{r+2},
\end{equation}
and
\begin{equation}\label{Eq:Aux_revised2}
\frac{\binom{n-j-r-1}{k-r-1}}{\binom{n-j}{k}}\le\left(\frac{k}{n}\right)^{r}.
\end{equation}
Note that we can assume w.l.o.g.~that $n$ is larger than any $C=C(j,r)$, by making sure that $s \geq C$ and noting that the lemma holds trivially when $n<s+k$.

\medskip

By Corollary~\ref{cor:Friedgut junta}, applied to the monotone family $\mathcal{F}^{\uparrow}$ with $p,q,\zeta',\epsilon$ as defined above, either we have $\mu_{q}\left(\mathcal{F}^{\uparrow}\right) \leq \zeta' \leq q^{r+2}$, or there exists a set $J$ of size $\le j$ with
\begin{equation}\label{Eq:Aux_revised3}
\mu_{q}\left(\left(\mathcal{F}^{\uparrow}\right)_{J}^{\emptyset}\right)\le \epsilon=q^{r+2}.
\end{equation}
(Note that Corollary~\ref{cor:Friedgut junta} can be applied to $\mathcal{F}^{\uparrow}$ with these parameters, as the conditions $\zeta'<q<(1-\zeta')p$ and $\mu_p(\mathcal{F}^{\uparrow}) \leq 1-\zeta'$ follow from the assumptions of the lemma and the definition of $\zeta'$.)
In fact,~\eqref{Eq:Aux_revised3} always holds, as in the former case~\eqref{Eq:Aux_revised3} holds with $J=\emptyset.$ Writing
\[
\mathcal{G}=\mathcal{F}_{J}^{\emptyset} \subseteq {{[n] \setminus J}\choose{k}},
\]
we obtain that
\begin{equation}\label{Eq:Aux_revised4}
\mu_{q}\left(\mathcal{G}^{\uparrow}\right) \leq q^{r+2}.
\end{equation}
On the other hand, Lemma~\ref{lem:claim 2.5}, applied to $\mathcal{G}$ with the parameters $n-j,k,r+1,q,\eta$ (in place of $n,k,r,p,\eta$, respectively),
implies that either
\[
\left|\mathcal{G}\right|< \binom{n-j-r-1}{k-r-1}
\]
or
\[
\mu_{q}\left(\mathcal{G}^{\uparrow}\right)>q^{r+1}\left(1-\mathrm{exp}\left(-\eta^{2}q\left(n-j\right)/2\right)\right)>q^{r+2},
\]
where the last inequality follows from~\eqref{Eq:Aux_revised1}. (Note that Lemma~\ref{lem:claim 2.5} can indeed be applied, as the assumption $q(1-\eta)\geq \frac{k-r-1}{n-j-r-1}$ it requires follows from the assumption $\frac{k}{n} \leq \frac{p}{2}$, provided $n$ is sufficiently large with respect to $j$.) As the latter contradicts~\eqref{Eq:Aux_revised4}, the former must hold. This completes the proof of the lemma, as by~\eqref{Eq:Aux_revised2},
\[
\mu\left(\mathcal{G}\right)<\frac{\binom{n-j-r-1}{k-r-1}}{\binom{n-j}{k}}\le\left(\frac{k}{n}\right)^{r}.
\]
\end{proof}

\section{\label{sec:Baby Case}Families that are cross free of a matching}

In this section we demonstrate our junta method on the `baby case' where $\h$ is a matching. Our aim is to prove the following theorem:
\begin{thm}\label{thm:Proof of the aaroni howrd conjecture}
For any constant $t \in \mathbb{N}$, there exists a constant $C=C\left(t\right)$, such that the
following holds. Let $k< n/C$, and let $\f_{1},\ldots,\f_{t}\subseteq{{[n]}\choose{k}}$
be families that are cross free of an ordered matching. Then
\[
\min_{i=1}^{t} \left|\f_{i}\right| \leq {{n}\choose{k}}-{{n-t+1}\choose{k}},
\]
with equality if and only if there exists an $\left(t-1,1\right)$-star $\u$, such that $\f_{1}=\cdots=\f_{t}=\u$.
\end{thm}

\subsection{Proof overview}
\label{sec:sub:match:overview}

Since in the case of a matching, there is no difference between cross containment of an ordered matching and cross containment of an un-ordered matching, we suppress the term `ordered' for sake of convenience.

As mentioned in the introduction, our proof consists of the following steps.
\begin{enumerate}
\item We first show that if $\f_{1},\ldots,\f_{t}\subseteq{{[n]}\choose{k}}$
are families that are cross free of a matching, then there exist
juntas $\g_{1},\ldots,\g_{t}$ that are also cross free of a matching,
such that each family $\f_{i}$ is essentially contained in the junta
$\g_{i}$.
\item We then show that if $\g_{1},\ldots,\g_{t}$ are juntas that
are cross free of a matching, then $\min_{i=1}^{t}\left|\g_{i}\right|\le{{n}\choose{k}}-{{n-t+1}\choose{k}}$.
Furthermore, we show that if `near equality' holds, then the juntas
$\g_{1},\ldots,\g_{t}$ are included in the same $\left(t-1,1\right)$-star.
\item The above steps show that if $\f_{1},\ldots,\f_{t}$ are families
that are cross free of a matching and satisfy $\min\left\{ \left|\f_{i}\right|\right\} \ge{{n}\choose{k}}-{{n-t+1}\choose{k}}$,
then all the families $\f_{i}$ are small alterations of an
$\left(t-1,1\right)$-star $\mathcal{U}=\{A: A \cap U \neq \emptyset\}$ (for some $|U|=t-1$).
The final step is to leverage this stability result into an exact result
with the help of a bootstrapping lemma.
\end{enumerate}

\medskip

The main step of the proof is the first one. The basic idea here is to decompose our family $\f$ into parts that are `easier to understand'. For this, we use the notion of \emph{capturability}, introduced by Dinur and Friedgut~\cite{dinur2009intersecting}.
\begin{defn}
Let $s\ge0$ be an integer, and let $\epsilon\in\left(0,1\right)$. A family $\f\subseteq{{[n]}\choose{k}}$ is said to be $\left(s,\epsilon\right)$-capturable, if there exists a set $S$ of size at most $s$, such that $\mu\left(\a_{S}^{\emptyset}\right)\le\epsilon$.
Otherwise, it is said to be $\left(s,\epsilon\right)$-uncapturable.
\end{defn}
Note that an $\left(s,\epsilon\right)$-uncapturable family is $\left(s,\epsilon'\right)$-uncapturable for any $\epsilon'<\epsilon$, and that for any $B \subset S \subset [n]$, if $\f$ is $\left(s,\epsilon\right)$-uncapturable, then $\f_S^B$ is $\left(s-|S|,\epsilon\right)$-uncapturable.

\medskip \noindent The proof of Step~1 consists of two parts.

\mn \textbf{Step 1(a).} We associate to each family $\f\subseteq{{[n]}\choose{k}}$,
each $s\in\mathbb{N}$, and $\epsilon'>0$ which is not too small, a set $J\subseteq\left[n\right]$, and a family $\j\subseteq\p\left(J\right)$,
such that for each $B\in\j$,
\begin{enumerate}
\item The family $\f_{B}^{B}$ is $\left(s,\epsilon'\right)$-uncapturable.
\item The family $\f$ is essentially contained in the junta $\j^{\uparrow}$.
\end{enumerate}
Intuitively, this means that that the family $\f$ consists of a negligible part that lies outside of $\j^{\uparrow}$, together with the parts
$\left\{ \f_{B}^{B}\right\} _{B\in\j}$ that are $\left(s,\epsilon'\right)$-uncapturable, and as such, are easier to understand.

This part will be used in subsequent sections as well.

\mn \textbf{Step 1(b).} We show that if the families $\f_{1},\ldots,\f_{t}$ are cross free of a matching, then the associated juntas from Step 1(a) are cross free of a matching as well.

\medskip

The proof of Step~2 is quite straightforward. The proof of Step~3 is also divided into two steps:

\mn \textbf{Step 3(a).} We show a `bootstrapping lemma' which asserts that if $\b_{1},\ldots,\b_{t}$ are some families that are
cross free of a matching, such that $\b_1,\ldots,\b_{t-1}$ are `very large' (formally, $\mu\left(\b_{i}\right)\ge1-\epsilon$),
then the last family must be `very small' (formally, $\mu\left(\b_{t}\right)\le O\left(\epsilon^{2}\right)$).

\mn \textbf{Step 3(b).} With the bootstrapping lemma in hand, we consider the parts of the families $\f_1,\ldots,\f_t$ that lie outside of
the $(t-1,1)$-star $\mathcal{U}$ mentioned above. (Formally, these are the  sets $\{\left(\f_{i}\right)_{U}^{\emptyset}\}$). We assume w.l.o.g. that the
largest among these `outside parts' is $\left(\f_{t}\right)_{U}^{\emptyset}$ and denote its measure by $\epsilon$.
Applying the bootstrapping lemma to the families
\[
\left(\f_{1}\right)_{U}^{\left\{ i_{1}\right\} },\ldots,\left(\f_{t-1}\right)_U^{\left\{ i_{t-1}\right\} },\left(\f_{t}\right)_{U}^{\emptyset},
\]
(which are cross free of a matching), we deduce that there exists $\ell \in [t-1]$ such that $\mu\left(\left(\f_{\ell}\right)_{U}^{\left\{ i_{\ell}\right\} }\right)\le1-\Omega\left(\sqrt{\epsilon}\right)$.

Informally, this shows that if one starts with $t$ copies of the $(t-1,1)$-star $\mathcal{U}$ and then tries to enlarge one of the families
by adding to it a family of measure $\epsilon$, then in order to keep the families cross free of a matching she will have to remove a set of measure at least
$\Omega(\sqrt{\epsilon})$ from one of the other families. Hence, any $t$ families $\f_1,\ldots,\f_t$ that are cross free of a matching and are
`small perturbations' of a $(t-1,1)$-star $\mathcal{U}$, satisfy $\min\left\{ \left|\f_{i}\right|\right\} \le |\mathcal{U}|$, with equality
only for $t$ copies of $\mathcal{U}$. The formal derivation here is a simple calculation.

\medskip

This section is organized as follows. The proof of Step~1 spans Sections~\ref{sec:sub:match:app-juntas}--\ref{sec:sub:match:completing-step-1}, where in Subsection~\ref{sec:sub:match:app-juntas} we present the junta associated to each family, in Subsection~\ref{sec:sub:match:uncap} we show that any uncapturable families $\f_{1},\ldots,\f_{t}$ cross contain a matching, and in Subsection~\ref{sec:sub:match:completing-step-1} we show that if the original families are cross free of a matching then so are the approximating juntas. In Subsection~\ref{sec:sub:match:solving-for-juntas} we prove Step~2, namely, that if $\g_{1},\ldots,\g_{t}$ are juntas that are cross free of a matching, then one of them must be `not larger' than a $(t-1,1)$-star, along with a stability version. The proof of Step~3 spans Sections~\ref{sec:sub:match:bootstrap-1} and~\ref{sec:sub:match:bootstrap-2}, where in Subsection~\ref{sec:sub:match:bootstrap-1} we show that if some families $\b_{1},\ldots,\b_{t}$ are cross free of a matching and $\b_{1},\ldots,\b_{t-1}$ are `very large' then $\b_t$ must be `very small', and in Subsection~\ref{sec:sub:match:bootstrap-2} we use this bootstrapping lemma to deduce that the $(t-1,1)$-star is a `local maximum' for the measure of families that are cross free of a matching. Finally, we complete the proof of Theorem~\ref{thm:Proof of the aaroni howrd conjecture} in Section~\ref{sec:sub:match:completing-the-proof}.

\subsection{Approximation by juntas}
\label{sec:sub:match:app-juntas}

In this subsection we present Step 1(a), namely, the association of a ``nice-behaved'' junta to any family $\f$.
\begin{prop}
\label{lem:associated junta lemma} Let $r,s\in\mathbb{N}$ be constants, and denote $C=\left(2s\right)^{r}$.
For any $k<n$, for any $\epsilon\ge\left(\frac{k}{n}\right)^{r}$, and for any family $\f\subseteq{{[n]}\choose{k}}$,
there exists a set $J\subseteq\left[n\right]$ of size at most $C$ and a family $\j\subseteq\p\left(J\right)$, such that:
\begin{enumerate}
\item For each $B\in\j$, the family $\f_{B}^{B}$ is $\left(s,\epsilon\left(\frac{n}{k}\right)^{\left|B\right|}\right)$-uncapturable,

\item We have
\[
\mu\left(\f\backslash\j^{\uparrow}\right)\le C\epsilon.
\]
\end{enumerate}
\end{prop}

\begin{rem}
Note that the sizes of all $B\in\j$ are naturally bounded from above. Indeed, we have $|B| \leq \log_{k/n} \epsilon \leq r$ for any $B \in \j$, as
no family can be $(s,\epsilon')$-uncapturable for $\epsilon' \geq 1$.
\end{rem}

\begin{proof}[Proof of Proposition~\ref{lem:associated junta lemma}]
We construct the junta $\j$ inductively, where the induction variable is $r$ (which measures `how small' $\epsilon$ is allowed to be).

For $r=0$, we set $J=\emptyset,\j=\emptyset$. Condition~(1) holds vacuously, and Condition~(2) holds since
\[
\mu\left(\f\backslash\j^{\uparrow}\right)=\mu\left(\f\right)=1\le\epsilon.
\]
For $r>0$, we first consider the case where $\f$ itself is $\left(s,\epsilon\right)$-uncapturable. In this case,
we take $J=\emptyset$ and $\j=\left\{ \emptyset\right\}$. Condition~(2) holds as $\mu\left(\f\backslash\j^{\uparrow}\right)=0$, and
Condition~(1) holds by hypothesis, since $\f=\f_{\emptyset}^{\emptyset}$.

Hence, we may assume that $r>0$ and $\f$ is $\left(s,\epsilon\right)$-capturable. This means that there exists a set $S$ of size $s$, such that
$\mu\left(\a_{S}^{\emptyset}\right)\le\epsilon$. We now apply the induction hypothesis to each of the families $\left\{ \f_{\left\{ i\right\} }^{\left\{ i\right\} }\right\} _{i\in S}$, with $\epsilon \frac{n}{k}$ in place of $\epsilon$.

The induction hypothesis implies that there exist sets $\{J_{i}\}_{i \in S}$ of size $\left|J_{i}\right|=\left(2s\right)^{r-1}$,
and families $\j_{i}\subseteq\p\left(J_{i}\right)$, such that:
\begin{itemize}
\item For each $B\in\j_{i}$, the family $\left(\f_{\left\{ i\right\} }^{\left\{ i\right\} }\right)_{B}^{B}$
is $\left(s,\epsilon\frac{n}{k}\left(\frac{n}{k}\right)^{\left|B\right|}\right)$-uncapturable,
\item We have
\[
\mu\left(\f_{\left\{ i\right\} }^{\left\{ i\right\} }\backslash\j_{i}^{\uparrow}\right)\le\left(2s\right)^{r-1}\frac{\epsilon n}{k}.
\]
\end{itemize}
Let $J=\bigcup_{i \in S} J_{i}\cup S$ and $\j = \{A_i \cup \{i\}: A_{i} \in \j_{i}\}$. We claim that $\j$ is the desired junta.

\medskip

First, note that since $|S| \leq s$, we have
\[
\left|J\right|\le\sum_{i \in S}\left|J_{i}\right|+\left|S\right|\le s+s\left(2s\right)^{r-1}\le\left(2s\right)^{r}=C,
\]
and thus the size of $J$ is as asserted.

To see that Condition~(1) holds, let $B\in\j$, and let $i\in S$ be such that $B\backslash\left\{ i\right\} \in \j_{i}$. Then by the definition of $\j_{i}$, the family $\f_{B}^{B}=\left(\f_{\left\{ i\right\} }^{\left\{ i\right\} }\right)_{B\backslash\left\{ i\right\} }^{B\backslash\left\{ i\right\} }$ is indeed $\left(s,\epsilon\left(\frac{n}{k}\right)^{\left|B\right|}\right)$-uncapturable.

Finally, to see that Condition~(2) holds,  
note that by a union bound,
\begin{align*}
\mu\left(\f\backslash\j^{\uparrow}\right) & =\Pr_{\mathbf{A}\sim{{[n]}\choose{k}}}\left[\mathbf{A}\in\f\backslash\j^{\uparrow}\right]\\
 & \le\Pr_{\mathbf{A}\sim{{[n]}\choose{k}}}\left[\mathbf{A}\in\f\backslash\j^{\uparrow}\mbox{ and }\mathbf{A}\cap S=\emptyset\right]
 +\sum_{i\in S}\Pr_{\mathbf{A}\sim{{[n]}\choose{k}}}\left[\mathbf{A}\in\f\backslash\j^{\uparrow}\mbox{ and }\mathbf{A}\cap S\supseteq\left\{ i\right\} \right]\\
 & \le\mu\left(\f_{S}^{\emptyset}\right)+\sum_{i\in S}\Pr_{\mathbf{A}\sim{{[n]}\choose{k}}}\left[i\in \mathbf{A}\right]\mu\left(\f_{\left\{ i\right\} }^{\left\{ i\right\} }\backslash\j_{i}^{\uparrow}\right)\\
 & \le\epsilon+\sum_{i\in S}\frac{k}{n}\left(2s\right)^{r-1}\cdot \epsilon \frac{n}{k}\le\left(2s\right)^{r}\epsilon=C\epsilon.
\end{align*}
\noindent This completes the proof.
\end{proof}

\subsection{Uncapturable families cross contain a matching}
\label{sec:sub:match:uncap}

We now turn to Step~1(b) which shows that if the families $\f_{1},\ldots,\f_{t}$ are cross free of a matching then the associated juntas $\j_{1}^{\uparrow},\ldots,\j_{t}^{\uparrow}$ defined in Proposition \ref{lem:associated junta lemma} are cross free of a matching
as well. In this subsection, we prove that any uncapturable families $\f_{1},\ldots,\f_{t}$
cross contain a matching. As we shall see, this proposition will allow us to show that existence of a matching in the juntas $\j_{1}^{\uparrow},\ldots,\j_{t}^{\uparrow}$ implies existence of a matching in the original families $\f_{1},\ldots,\f_{t}$.
\begin{prop}
\label{prop:uncapturable} For any constants $r,t\in\mathbb{N}$, there exists $s=s(r,t)$
such that the following holds. Let $k_{1},\ldots,k_{t}<\frac{n}{2t}$,
and let $\f_{1}\subseteq {{\left[n\right]}\choose{k_{1}}},\ldots,\f_{t}\subseteq {{\left[n\right]}\choose{k_{t}}}$
be families that are cross free of a matching. Then there exists $i \in [t]$ such that the family
$\f_{i}$ is $\left(s,\left(\frac{k_{i}}{n}\right)^{r}\right)$-capturable.
\end{prop}

The proof of the proposition consists of two steps.
\begin{enumerate}
\item We first show that there exists $i\in\left[t\right]$ such that the biased $\mu_{1/t}$ measure of the monotonization
$\f_{i}^{\uparrow}$ is bounded away from $1$.
\item We then apply a lemma of Dinur and Friedgut~\cite{dinur2009intersecting} (Lemma~\ref{lem:DF} above) which shows that the above
statement implies that the family $\f_{i}$ is capturable. \end{enumerate}
\begin{prop}
\label{prop:cross-matching} Let $\f_{1},\ldots,\f_{t} \subseteq \pn$ be families that are cross free of a matching, and
let $p_{1},\ldots,p_{t}\in\left(0,1\right)$ be such that $p_{1}+\cdots+p_{t}\le1$. Then
\[
\sum_{i=1}^{t}\mu_{p_{i}}\left(\f_{i}\right) \leq t-1.
\]
\end{prop}
\begin{proof}[Proof of Proposition \ref{prop:cross-matching}]
For $\a,\b \subseteq\pn$, we denote by $\a\sqcup\b$ the family
of all sets of the form $D_{1}\cup D_{2}$ for some pairwise disjoint
sets $D_{1}\in\a,D_{2}\in\b$.
\begin{claim}
\label{Claim:coupling fun lemma}
Let $\epsilon_{1},\epsilon_{2}>0$ and let $p,q\in\left(0,1\right)$ be numbers such that $p+q\le1$.
Let $\mathcal{A}, \mathcal{B} \in \p\left(\left[n\right]\right)$ be families such that $\mu_{p}\left(\mathcal{A}\right)\ge1-\epsilon_{1}$ and
$\mu_{q}\left(\mathcal{B}\right)\ge1-\epsilon_{2}$. Then $\mu_{p+q}\left(\mathcal{A\sqcup B}\right)\ge1-\epsilon_{1}-\epsilon_{2}$.
\end{claim}
\begin{proof}
We define a coupling of $\mu_{p},\mu_{q}$ and $\mu_{p+q}$. For each
$i\in\left[n\right]$, let $\mathbf{x_{i}}$ be chosen uniformly and independently
at random from $\left[0,1\right]$. We say that a coordinate $i\in[n]$
is of \emph{type 1} if $0\leq \mathbf{x_{i}}\le p$, and
that $i \in [n]$ is of \emph{type 2} if $p<\mathbf{x_{i}}\le p+q$. Let
$\mathbf{A}$ (resp. $\mathbf{B}$) be the set of all coordinates of type $1$
(resp. type $2$). Then $\mathbf{A}$ is distributed according to $\mu_{p}$
(and in particular, $\Pr\left[\mathbf{A}\in\mathcal{A}\right]=\mu_{p}\left(\mathcal{A}\right)$),
$\mathbf{B}$ is distributed according to $\mu_{q}$, and $\mathbf{A}\cup \mathbf{B}$ is distributed
according to $\mu_{p+q}$. Since $\mathbf{A}$ and $\mathbf{B}$ are disjoint, we have:
\begin{align*}
\mu_{p+q}\left(\mathcal{A}\sqcup\mathcal{B}\right) & =\Pr\left[\mathbf{A}\cup \mathbf{B}\in\mathcal{A}\sqcup\mathcal{B}\right]
 =1-\Pr\left[\mathbf{A}\cup \mathbf{B}\notin\mathcal{A}\sqcup\mathcal{B}\right] \\
 &\ge1-\left(\Pr\left[\mathbf{A}\notin\mathcal{A}\right]+\Pr\left(\mathbf{B}\notin\mathcal{B}\right)\right)
  = 1-(1-\mu_p(\a))-(1-\mu_p(\b)) \\
  &\ge 1-\epsilon_{1}-\epsilon_{2}.
\end{align*}

\end{proof}
We now turn back to the proof of Proposition~\ref{prop:cross-matching}. 
Write $\mu_{p_{i}}\left(\mathcal{F}_{i}\right)=1-\epsilon_{i}$
for each $i$. By Claim~\ref{Claim:coupling fun lemma}, we have
\[
0=\mu_{\sum_{i \in [t]} p_{i}}\left(\emptyset\right)=\mu_{\sum p_{i}}\left(\bigsqcup_{i\in\left[t\right]}\mathcal{F}_{i}\right)\ge1-\sum_{i \in [t]} \epsilon_{i},
\]
where the second equality holds since $\f_1,\ldots,\f_t$ are cross free of a matching. Hence,
\[
\sum_{i=1}^{t}\mu_{p_{i}}\left(\f_{i}\right)\le t-\sum_{i=1}^{t}\epsilon_{i}\le t-1,
\]
as asserted.
\end{proof}

\begin{rem}
We note that if $p_{1}+\cdots+p_{t}>1$, then for \emph{any} $\epsilon>0$, there exist
families $\f_{1},\ldots,\f_{t}$ that are cross free of a matching,
such that $\mu_{p_{i}}\left(\f_{i}\right)\ge1-\epsilon$ for all $i$. Indeed, choosing $q_1,\ldots,q_t$ such that $q_{1}<p_{1},q_{2}<p_{2},\ldots,q_{t}<p_{t}$ and $\sum_{i=1}^{t}q_{i}>1$, for a sufficiently large $n$ we obtain that the families $\f_1,\ldots,\f_t$, where
\[
\f_{i}=\left\{ A\subseteq\pn\,:\,\left|A\right|\ge q_{i}n\right\},
\]
are cross free of a matching and satisfy $\mu_{p_{i}}\left(\f_{i}\right)\ge1-\epsilon$. In this sense, Proposition~\ref{prop:cross-matching} is sharp.
\end{rem}

Now we can prove Proposition \ref{prop:uncapturable}.
\begin{proof}[Proof of Proposition \ref{prop:uncapturable}]
Since the families $\f_{1},\ldots,\f_{t}$ are cross free of a matching,
their monotonizations $\f_{1}^{\uparrow},\ldots,\f_{t}^{\uparrow}$
are cross free of a matching as well. By Proposition \ref{prop:cross-matching},
this implies that $\sum_{i=1}^{t}\mu_{1/t}\left(\f_{i}^{\uparrow}\right)\le t-1$.
Hence there exists a family $\f_{i}^{\uparrow}$, such that
$\mu_{1/t}\left(\f_{i}^{\uparrow}\right)\le1-\frac{1}{t}$.
By Lemma \ref{lem:DF} (applied with $\zeta=\frac{1}{2t}$ and $p=\frac{1}{t}$; note that $k_i<\frac{n}{2t}$ and so the assumption $\frac{k}{n} \leq \frac{p}{2}$ holds), we obtain that there exists a constant $s=s\left(t,r\right)$
such that the family $\f_{i}$ is $\left(s,\left(\frac{k_{i}}{n}\right)^{r}\right)$-capturable.
This completes the proof of the proposition.
\end{proof}

\subsection{The approximating juntas are cross free of a matching}
\label{sec:sub:match:completing-step-1}

We are now ready to show that if $\f_{1},\ldots,\f_{t}$ are
families that are cross free of a matching, then the associated juntas
$\j_{1}^{\uparrow},\ldots,\j_{t}^{\uparrow}$ defined in Section~\ref{sec:sub:match:app-juntas}
are cross free of a matching as well.
\begin{thm}
\label{thm:matching junta approximation} Let $t,r$ be some constants, let $k<\frac{n}{2t}$,
and let $\f_{1},\ldots,\f_{t}\subseteq{{[n]}\choose{k}}$
be families that are cross free of a matching. Then there exist
$O_{r,t}\left(1\right)$-juntas $\g_{1},\ldots,\g_{t}\subseteq{{[n]}\choose{k}}$
that are cross free of a matching, such that $\mu\left(\f_{i}\backslash\g_{i}\right)=O_{t,r}\left(\left(\frac{k}{n}\right)^{r}\right)$
for each $i\in\left[t\right]$.
\end{thm}
We remark that Theorem \ref{thm:matching junta approximation} generalizes
a result of Dinur and Friedgut \cite{dinur2009intersecting} who proved
the same assertion in the case $t=2$ (i.e., for intersecting families); see Section~\ref{sec:proof}.
\begin{proof}
Let $s=s\left(t,r\right)$ be a sufficiently large constant to be defined below.
By Proposition~\ref{lem:associated junta lemma} (applied with $\epsilon=(k/n)^r$), there exist sets $J_{1},\ldots,J_{t}$ of size
$O_{t,r}\left(1\right)$ each, and families $\j_{1}\subseteq\p\left(J_{1}\right),\ldots,\j_{t}\subseteq\p\left(J_{t}\right)$,
such that for each $i\in\left[t\right]$, we have:
\begin{enumerate}
\item For each set $B\in\j_{i}$, the family $\left(\f_{i}\right)_{B}^{B}$
is $\left(s,\left(\frac{k}{n}\right)^{r-\left|B\right|}\right)$-uncapturable.

\item $\mu\left(\f_{i}\backslash\j_{i}^{\uparrow}\right)=O_{t,r}\left(\left(\frac{k}{n}\right)^{r}\right)$.

\end{enumerate}
By the proof of Lemma~\ref{lem:measures of juntas}, we may remove from the families
$\j_{i}$ all the sets of size at least $r$, so we assume without
loss of generality that $\left|B\right|<r$ for each $B\in\j_{i}$.

To complete the proof we shall show that the families $\j_{1}^{\uparrow},\ldots,\j_{t}^{\uparrow}$
are cross free of a matching. Suppose on the contrary that there exist
a matching $A_{1},\ldots,A_{t}$ with $A_{1}\in\j_{1}^{\uparrow},\ldots,A_{t}\in\j_{t}^{\uparrow}.$
This implies that there exists a matching $B_{1},\ldots,B_{t}$
with $B_{1}\in\j_{1},\ldots,B_{t}\in\j_{t}$. Write $E=B_{1}\cup\cdots\cup B_{t}$. To reach a contradiction, we show
that the following contradicting claims hold:
\begin{claim}
\label{Trivial 2}The families $\left(\f_{1}\right)_{E}^{B_{1}},\ldots,\left(\f_{t}\right)_{E}^{B_{t}}$
are $\left(s-\left(r-1\right)\left(t-1\right),\left(\frac{k}{n}\right)^r\right)$-uncapturable.
\end{claim}

\begin{claim}
\label{trivial 1}The families $\left(\f_{1}\right)_{E}^{B_{1}},\ldots,\left(\f_{t}\right)_{E}^{B_{t}}$
are cross free of a matching.
\end{claim}

\begin{proof}[Proof of Claim \ref{Trivial 2}]
By the hypothesis, each family $(\f_i)_{B_{i}}^{B_{i}}$ is $\left(s,\left(\frac{k}{n}\right)^r\right)$-uncapturable.
So by definition, the family $(\f_i)_{E}^{B_{i}}$ is $\left(s-\left|E\backslash B_{i}\right|,\left(\frac{k}{n}\right)^r\right)$-uncapturable.
The claim follows from the fact that $\left|E\backslash B_i\right|\le\left(r-1\right)\left(t-1\right),$ for each $i$.
\end{proof}

\begin{proof}[Proof of Claim \ref{trivial 1}]
Suppose on the contrary that there exists a matching $C_{1}\in\left(\f_{1}\right)_{E}^{B_{1}},\ldots,C_{t}\in\left(\f_{t}\right)_{E}^{B_{t}}$. Then
the sets $C_{1}\cup B_{1},\ldots,C_{t}\cup B_{t}$ constitute a cross matching in the families $\f_{1},\ldots,\f_{t}$, contradicting the hypothesis.
\end{proof}

By Proposition~\ref{prop:uncapturable}, Claims~\ref{Trivial 2} and~\ref{trivial 1} contradict each other (provided that $s=s(r,t)$ is chosen to be large
enough). Hence, the families $\j_{1}^{\uparrow},\ldots,\j_{t}^{\uparrow}$ are cross free of a matching, as asserted.
\end{proof}

\subsection{Characterization of `large' juntas that are cross free of a matching}
\label{sec:sub:match:solving-for-juntas}

In this subsection we present Step~2 of the proof, which shows that if some juntas $\g_{1},\ldots,\g_{t}$ are cross free of a matching,
then $\min\left|\g_{i}\right|\le{{n}\choose{k}}-{{n-t+1}\choose{k}}$, and that near inequality holds if and only if all these juntas are included in the same $\left(t-1,1\right)$-star.

We use the following simple observation.
\begin{obs}
For any constants $j,t$ there exists a constant $C\left(t,j\right)$, such that the following holds. Let $J \in {{[n]}\choose{j}}$, let $k \le n/C$, and let $\left\langle \j_{1}\right\rangle ,\ldots,\left\langle \j_{t}\right\rangle \subseteq{{[n]}\choose{k}}$
be $J$-juntas that are cross free of a matching. Then the families
$\j_{1},\ldots,\j_{t}$ are cross free of a matching as well.
\end{obs}

\begin{proof}
Suppose on the contrary that the sets $A_{1}\in\j_{1},\ldots,A_{s}\in\j_{t}$
constitute a matching. Provided that $C$ is sufficiently large, we can
also find a matching $B_{1}\in {{\left[n\right]\backslash J}\choose{k-\left|A_{1}\right|}},\ldots,B_{t}\in {{\left[n\right]\backslash J}\choose{k-\left|A_{t}\right|}}$.
The sets $A_{1}\cup B_{1}\in\left\langle \j_{1}\right\rangle ,\ldots,A_{t}\cup B_{t}\in\left\langle \j_{t}\right\rangle $
constitute a matching, a contradiction.
\end{proof}

\begin{prop}
\label{lem:finding the largest junta}For any constant $j$, there exists $C(j)$, such that the following holds. Let $J \in {{[n]}\choose{j}}$, and let $k\le n/C$. Suppose that $\g_{1},\ldots,\g_{t}\subseteq{{[n]}\choose{k}}$
are $J$-juntas that are cross free of a matching. Then
\begin{equation}
\min_{\{i \in [t]\}}\left\{ \left|\g_{i}\right|\right\} \le{{n}\choose{k}}-{{n-t+1}\choose{k}}.\label{eq:largest junta matching}
\end{equation}
Moreover, if
\begin{equation}
\min\left\{ \left|\g_{i}\right|\right\} \ge{{n}\choose{k}}-{{n-t+2}\choose{k}}+C\frac{k^{2}}{n^{2}}{{n}\choose{k}},\label{eq:largest junta matching stability}
\end{equation}
then the juntas $\g_{1},\ldots,\g_{t}$ are contained in the same
$\left(t-1,1\right)$-star.\end{prop}

\begin{proof}
The assertion~(\ref{eq:largest junta matching}) clearly follows from the `Moreover'
statement, so we only prove the latter. Write $\g_{i}=\left\langle \j_{i}\right\rangle $
for some family $\j_{i}\subseteq\p\left(J\right)$.
\begin{claim}
\label{Claim each Gi contains many singletons} Each family $\j_{i}$ contains at least $t-1$ singletons.
\end{claim}
\begin{proof}
Otherwise, we would have $\left|\j_{i}\right|\le\left(t-2\right){{n-1}\choose{k-1}}+O\left(\frac{k}{n}\right)^{2}{{n}\choose{k}}$
by Lemma~\ref{lem:measures of juntas}. A straightforward calculation shows that this contradicts~(\ref{eq:largest junta matching stability}),
provided that $C$ is sufficiently large.
\end{proof}

Write $U_{1}=\j_{1}^{\left(1\right)},\ldots,U_{t}=\j_{t}^{\left(1\right)}$ (i.e., the sets of singletons in $\j_1,\ldots,\j_t$, respectively).
Since the families $\j_{1},\ldots,\j_{t}$ are cross free of a matching, there are no distinct elements $i_{1}\in U_{1},\ldots,i_{t}\in U_{t}$.
The following claim shows that this implies $U_{1}=\cdots=U_{t}$.
\begin{claim}
Let $t\ge2$, and let $U_{1},\ldots,U_{t}$ be sets of size at least $t-1$, such that there are no distinct elements $i_{1}\in U_{1},\ldots,i_{t}\in U_{t}$.
Then $U_{1}=\cdots=U_{t}$, and $|U_i|=t-1$.
\end{claim}
\begin{proof}
We prove the claim by induction on $t$. For $t=2$, the claim is trivial. Suppose that $t>2$, and choose some $a\in U_{1}$. By the induction
hypothesis, we have $U_{2}\backslash\left\{ a\right\} = \ldots = U_{t}\backslash\left\{ a\right\} = U' $, for some $|U'|=t-2$.
Since for all $i$, $|U_{i}| \geq t-1$, it follows that $a \in U_i$ for all $i>1$. Hence, the sets $U_{2},\ldots,U_{t}$ are equal. The same argument
shows that the sets $U_{1},U_{3},\ldots,U_{t}$ are equal, and therefore all the sets $U_{1},\ldots,U_{t}$ are equal and are of size
$t-1$.
\end{proof}

Write $U_{1}=\cdots=U_{t-1}=U$, and let $\mathcal{U}$ be the $\left(t-1,1\right)$-star of all the sets whose intersection with $U$ is non-empty (i.e., $\mathcal{U}= \{S \in {{[n]}\choose{k}}: S \cap U \neq \emptyset\}$). The proof
of Proposition~\ref{lem:finding the largest junta} will be finished by the following claim.
\begin{claim}
Each family $\g_{i}$ is contained in $\mathcal{U}$.
\end{claim}
\begin{proof}
Suppose on the contrary that $\g_{t} \not \subseteq \mathcal{U}$,
and let $A_t\in \g_{t}\backslash\mathcal{U}$. Write $U=\left\{ i_{1},\ldots,i_{t-1}\right\} $.
Then the sets $\left\{ i_{1}\right\} \in\g_{1},\ldots,\left\{ i_{t-1}\right\} \in\g_{t-1},A_{t}\in\g_{t}$
constitute a matching. This contradicts the fact the the families $\g_{1},\ldots,\g_{t}$ are cross free of a matching.
\end{proof}
\noindent This completes the proof of the proposition.
\end{proof}

\subsection{The bootstrapping lemma}
\label{sec:sub:match:bootstrap-1}

We now turn to Step~3 of the proof. In this subsection we present Step~3(a) -- the bootstrapping lemma which
asserts that if some families $\b_{1},\ldots,\b_{t}$ are
cross free of a matching and $\b_1,\ldots,\b_{t-1}$ are `very large', then $\b_t$ must be `very small'.

Our proof relies on a coupling argument which is similar to the
argument we used in the proof of Claim~\ref{Claim:coupling fun lemma}.
\begin{prop}
\label{prop: unbalanced matching lemma}For any $t,r\in\mathbb{N}$,
there exists a constant $C=C\left(t,r\right)$, such that the following
holds. Let $\epsilon>0$, let $k_{1},\ldots,k_{t}\le n/C$, and let $\b_{1}\subseteq {{\left[n\right]}\choose{k_{1}}},\ldots,\b_{t}\subseteq {{\left[n\right]}\choose{k_{t}}}$
be families that are cross free of a matching. If $\mu\left(\b_{1}\right),\ldots,\mu\left(\b_{t-1}\right)\ge1-\epsilon$,
then $\mu\left(\b_{t}\right)\le O_{r,t}\left(\epsilon^{r}\right)$.
\end{prop}
\begin{proof}
Let $\b_1,\ldots,\b_t$ be families that satisfy the assumptions of the proposition. Consider the family $\tilde{\b}_{t} = \left(\b_{t}^{\uparrow}\right)^{\left(\lfloor n/t \rfloor \right)}$, and note that the families $\b_{1},\ldots,\b_{t-1},\tilde{\b_{t}}$
are cross free of a matching as well. By Lemma~\ref{lem:Kruskal-Katona going up}, we have
\begin{equation}
\mu\left(\b_{t}\right)\le O_{t,r}\left(\mu\left(\tilde{\b_{t}}\right)^{r}\right),\label{eq:second}
\end{equation}
provided that $C$ is sufficiently large. (Specifically, we apply Lemma~\ref{lem:Kruskal-Katona going up} with $k=\lfloor n/t \rfloor$ and $l=k_t$ to obtain the assertion of the lemma for a constant $C'$, and then~\eqref{eq:second} holds, provided that $C>tC'$.) We now use a simple coupling argument to show that $\mu(\tilde{\b}_{t})=O_t(\epsilon)$,
which will complete the proof. We need the following claim.
\begin{claim}
Let $n,l_{1},\ldots,l_{t} \in \mathbb{N}$ be such that $n>l_{1}+\cdots+l_{t}$.
Then there exists a distribution on tuples $(A_1,\ldots,A_t)$ of subsets of $[n]$, such that the sets $\mathbf{A_1},\ldots,\mathbf{A_t}$ are always pairwise disjoint
and the marginal distributions are $\mathbf{A_{1}}\sim {{\left[n\right]}\choose{l_{1}}},\mathbf{A_{2}}\sim {{\left[n\right]}\choose{l_{2}}},\ldots,\mathbf{A_{t}}\sim {{\left[n\right]}\choose{l_{t}}}$.
\end{claim}
\begin{proof}
Let $\mathbf{\sigma}\sim S_{n}$ be a uniformly chosen random permutation on $[n]$. Then the disjoint sets
\begin{align*}
\mathbf{A_{1}}: & =\left\{ \sigma\left(1\right),\ldots,\sigma\left(l_{1}\right)\right\} ,\mathbf{A_{2}}:=\left\{ \sigma\left(l_{1}+1\right),\ldots,\sigma\left(l_{1}+l_{2}\right)\right\} ,\\
 & \ldots,\mathbf{A_{t}}:=\left\{ \sigma\left(l_{1}+\cdots+l_{t-1}+1\right),\ldots,\sigma\left(l_{1}+\cdots+l_{t}\right)\right\}
\end{align*}
have the desired marginal distributions.
\end{proof}
Let $(\mathbf{A_1},\ldots,\mathbf{A_t})$ be distributed as in the claim, with $(k_1,\ldots,k_{t-1},\lfloor n/t \rfloor)$ in place of $(l_1,\ldots,l_t)$.
Using the claim and a simple union bound, we have
\begin{align*}
0&=\Pr\left[\mathbf{A_{1}}\in\b_{1},\ldots,\mathbf{A_{t}}\in\tilde{\b_{t}}\right] \ge \Pr[\mathbf{A_t} \in \tilde{\b_{t}}] - \sum_{i=1}^{t-1} \Pr[\mathbf{A_i} \not \in \b_i] \\
 &= \mu(\tilde{\b_t}) - \sum_{i=1}^{t-1} (1-\mu(\b_i)) \geq \mu(\tilde{\b_t})-(t-1)\epsilon,
\end{align*}
where the first equality holds since $\b_1,\ldots,\b_{t-1},\tilde{\b}_t$ are cross free of a matching, and the last inequality uses the assumption $\mu(\b_i) \geq 1-\epsilon$ for all $i \in [t-1]$. Therefore, $\mu(\tilde{\b_t}) \leq (t-1)\epsilon$.
By~(\ref{eq:second}), this implies
\[
\mu\left(\mathcal{\b}_{t}\right)\le O_{t,r}\left(\left(\left(t-1\right)\epsilon\right)^{r}\right)=O_{t,r}\left(\epsilon^{r}\right).
\]
This completes the proof of the proposition.
\end{proof}

\subsection{The $\left(t-1,1\right)$-star is locally maximal among the `cross matching-free' families}
\label{sec:sub:match:bootstrap-2}

We now turn to Step~3(b) of the proof, which uses the above bootstrapping to show that if $\f_{1},\ldots,\f_{t}$
are `small alterations' of an $\left(t-1,1\right)$ star, then
\[
\min_{i=1}^{t} \left|\f_{i}\right|\le{{n}\choose{k}}-{{n-t+1}\choose{k}},
\]
 with equality if and only if the families $\f_{1},\ldots,\f_{t-1}$
are all equal to the same $\left(t-1,1\right)$-star.
\begin{prop}
\label{lem:Locally porcupines are the best} For each constant $t$,
there exists $C=C\left(t\right)$ such that the following holds. Let $k\le n/C$, and let $\f_{1},\ldots,\f_{t}\subseteq{{[n]}\choose{k}}$
be families that are cross free of a matching. Suppose additionally, that there exists a set $U$ of size $t-1$, such that $\mu\left(\left(\f_{i}\right)_{U}^{\emptyset}\right)\le O_t \left( \left(\frac{k}{n}\right)^{3} \right)$ for any $i\in\left[t\right]$. Then
\[
\min_{i=1}^{t} \left|\f_{i}\right|\le{{n}\choose{k}}-{{n-t+1}\choose{k}},
\]
 with equality if and only if the families $\f_{1},\ldots,\f_{t}$
are all equal to the same $\left(t-1,1\right)$-star $\mathcal{U}=\{A \in {{[n]}\choose{k}}: A \cap U \neq \emptyset\}$.
\end{prop}
\begin{proof}
Let $\f_{1},\ldots,\f_{t}\subseteq{{[n]}\choose{k}}$ be
families that are cross free of a matching, and suppose that
$\min_{i=1}^{t}\left|\f_{i}\right|\ge{{n}\choose{k}}-{{n-t+1}\choose{k}}$.
We show that all the families are all equal to the same $\left(t-1,1\right)$-star. Write
\[
\max_{i=1}^{t} \mu\left(\left(\f_{i}\right)_{U}^{\emptyset}\right)=\epsilon',
\]
and suppose w.l.o.g. that $\mu\left(\left(\f_{t}\right)_{U}^{\emptyset}\right)=\epsilon'$.
Let $U=\left\{ i_{1},\ldots,i_{t-1}\right\} $. The families
\[
\left(\f_{t}\right)_{U}^{\emptyset},\left(\f_{1}\right)_{U}^{\left\{ i_{1}\right\} },\ldots,\left(\f_{t-1}\right)_{U}^{\left\{ i_{t-1}\right\} }
\]
are cross free of a matching. By Proposition \ref{prop: unbalanced matching lemma} (applied with $r=2$), it follows that there exists $\ell\in\left[t-1\right]$ such that $\mu\left(\left(\f_{\ell}\right)_{U}^{\{i_{\ell}\}}\right)\le1-\Omega_{t}\left(\sqrt{\epsilon'}\right)$, and hence, $\left|\left(\f_{\ell}\right)_{U}^{\{i_{\ell}\}}\right|\le {{n-t+1}\choose{k-1}}(1-ct \sqrt{\epsilon'})$ for some constant $c(t)$.
Thus,
\begin{align}
\left|\f_{\ell}\right| & =\sum_{B\subseteq U} \left| \left(\f_{\ell}\right)_{U}^{B} \right| \le \left|\left(\f_{\ell}\right)_{U}^{\emptyset}\right| + \left|\left(\f_{\ell}\right)_{U}^{\{i_{\ell}\}} \right| +  \sum_{B\in \p(U) \setminus \{\emptyset,\{i_{\ell}\}\}} {{n-t+1}\choose{k-\left|B\right|}} \nonumber \\
 & \leq {{n}\choose{k}}-{{n-t+1}\choose{k}}+\epsilon' {{n-t+1}\choose{k}}-c\sqrt{\epsilon'} {{n-t+1}\choose{k-1}} \nonumber \\
 & \le{{n}\choose{k}}-{{n-t+1}\choose{k}}+{{n-t+1}\choose{k}}\left(\epsilon'-\frac{k}{n}c\sqrt{\epsilon'}\right).\label{eq:ak2}
\end{align}
Since by assumption, $\epsilon' \leq O_t \left(\left(\frac{k}{n} \right)^3 \right)$, this implies
\begin{equation}
\left|\f_{\ell}\right|\le{{n}\choose{k}}-{{n-t+1}\choose{k}}.\label{eq:bound ak}
\end{equation}
Finally, equality holds in (\ref{eq:bound ak}) if and only if equality holds in
(\ref{eq:ak2}), which is possible only when $\epsilon'=0$. Therefore, equality holds only when the families $\f_{1},\ldots,\f_{t}$
are all equal to the $\left(t-1,1\right)$-star $\mathcal{U}=\{A \in {{[n]}\choose{k}}: A \cap U \neq \emptyset\}$. This
completes the proof.
\end{proof}

\subsection{Proof of Theorem~\ref{thm:Proof of the aaroni howrd conjecture}}
\label{sec:sub:match:completing-the-proof}

We are now ready to prove Theorem~\ref{thm:Proof of the aaroni howrd conjecture}. 

\begin{proof}[Proof of Theorem~\ref{thm:Proof of the aaroni howrd conjecture}]
By Theorem~\ref{thm:matching junta approximation} (applied with $r=3$), there exists an
$O_{t}\left(1\right)$-set $J$ and juntas $\j_{1},\ldots,\j_{t}\subseteq\p\left(J\right)$
that are cross free of a matching, such that $\mu\left(\f_{i}\backslash \langle \j_{i}^{\uparrow} \rangle \right)=O_{t}\left(\frac{k}{n}\right)^{3}$ for all $i$.
Assuming that $\min_{i=1}^{t}\left|\f_{i}\right| \geq {{n}\choose{k}}-{{n-t+1}\choose{k}}$, this implies that
\[
\left|\langle \j_{i}^{\uparrow} \rangle \right|\ge{{n}\choose{k}}-{{n-t+1}\choose{k}}-O_{t}\left(\left(\frac{k}{n}\right)^{3}{{n}\choose{k}}\right)
\]
for any $i\in\left[t\right]$. By Proposition \ref{lem:finding the largest junta}, it follows that
the juntas $\langle \j_{1}^{\uparrow} \rangle ,\cdots, \langle \j_{t}^{\uparrow} \rangle$ are all equal
to the same $\left(t-1,1\right)$-star. Provided that $C$ is sufficiently large, the assertion of the theorem
follows now from Proposition \ref{lem:Locally porcupines are the best}.
\end{proof}

\section{Detailed Overview of the Proof Strategy in the General Case}
\label{sec:overview}

After demonstrating our junta technique in the `baby case' where $\h$ is a matching, we are now ready to treat the general case where $\h$ is allowed to be any $d$-expanded hypergraph. In order to facilitate reading, we present in this section the `big picture' of the argument, which spans Sections~\ref{sec:fairness}--\ref{sec:proof}.




Fix $d,h$, and let $\h$ denote a fixed $d$-expanded hypergraph of size $h$. The general structure of the proof of our main results is the same as
in the case where $\h$ is a matching:
\begin{enumerate}
\item We prove the junta approximation theorem (Theorem~\ref{thm:Junta-approx-theorem}) which asserts that for any forbidden hypergraph $\h$, any $\h$-free family can be approximated by an $\h$-free junta.

\item For a specific $\h$ (or for some class of $\h$'s), we find the extremal $\h$-free junta $\j$ and show that any $\h$-free junta that is nearly extremal, is contained in $\j$.

\item The above steps imply that any nearly extremal $\h$-free family is a small alteration of $\j$. The last step is to bootstrap this stability
result and show that the size of any $\h$-free small alteration of $\j$ is smaller than the size of $\j$.
\end{enumerate}

The proof of Step~2 is an easy generalization of the proof of the corresponding step in Section~\ref{sec:Baby Case}. The proof of Step~3 is a rather direct generalization as well, but it requires several technical propositions which generalize Proposition~\ref{prop: unbalanced matching lemma}. These propositions span Section~\ref{sec:bootstrapping}, and a detailed overview of their place in the `large picture' is given in Section~\ref{sec:sub:boot:motivation}. Here we present an informal description of the most complex part of the argument: Step~1 -- the proof of Theorem~\ref{thm:Junta-approx-theorem}.

\subsection{Overview of the proof of Theorem~\ref{thm:Junta-approx-theorem}}
\label{sec:sub:overview:junta}

Recall that the \emph{kernel} $K\left(\h\right)$ of a hypergraph $\h$ is the intersection of all its edges. Set $t=|K\left(\h\right)|+1$. It is easy to see that the $\left(t,t\right)$-star is free of $\h$. As the $(t,t)$-star $\s_T \subseteq {{[n]}\choose{k}}$ has measure $\Theta\left(\left(\frac{k}{n}\right)^{t}\right)$, this allows us to treat any family of measure $o\left(\left(\frac{k}{n}\right)^{t}\right)$
as negligible.

The proof of Theorem~\ref{thm:Junta-approx-theorem} is composed of three steps:
\begin{enumerate}
\item We choose (with foresight) parameters $s$ (a `sufficiently large' integer) and  $\epsilon=\left(\frac{k}{n}\right)^{t} \cdot \max\left(C\frac{k}{n},e^{-k/C}\right)$ for a sufficiently large constant $C$ that depends only on $\h$, and apply Proposition~\ref{lem:associated junta lemma} with these parameters. The proposition asserts that any $\h$-free family $\f\subseteq{{[n]}\choose{k}}$ can be approximated by a junta $\j = \left\langle \j'\right\rangle $ (where $\j' \subseteq \p(J)$), such that for each $B\in\j'$, the slice $\f_{B}^{B}$ is $\left(s,\epsilon\left(\frac{n}{k}\right) ^{\left|B\right|}\right)$-uncapturable.

    Note that $\j'$ contains no set $B$ of size at least $t+1$, since otherwise, the slice $\f_{B}^{B}$ would be $(s,\beta)$-uncapturable for $\beta \geq 1$, which is of course impossible.

\item We show that $\j'$ contains no sets of size at most $t-1$, and thus, is $t$-uniform.

\item We then show that the junta $\j=\left\langle \j'\right\rangle $ is free of $\h$.
\end{enumerate}
It turns out that Steps~2 and~3 can be reduced to pure statements concerning uncapturable families, due to the following observations:

\begin{obs}
For each $B$ of size at most $t-1$, the family $\g:=\f_{B}^{B}$ is free of the hypergraph obtained from $\h$ by removing $\left|B\right|$ vertices out of its kernel.
\end{obs}
\noindent Thus, in order to accomplish Step 2, it is sufficient to show that:
\begin{prop}[Informal]\label{P:1}
Any $\left(s,\max \left(C\frac{k}{n},e^{-k/C}\right) \cdot \left(\frac{k}{n}\right)^{t'}\right)$-uncapturable family contains a copy of any $d$-expanded hypergraph with kernel of size $t'-1$.
\end{prop}

\begin{obs}
For any $h$-tuple of (possibly non-distinct) $t$-sized sets $B_{1},\ldots,B_{h} \in \j'$, if the junta $\langle \j' \rangle$ contains a copy $(B_1 \cup E_1,\ldots,B_h \cup E_h)$ of $\h$ such that $E_i \cap J = \emptyset$ for all $i \in [h]$, then the families
    \[
    \g_{1}:=\f_{B_{1}\cup \ldots \cup B_h}^{B_{1}},\g_{2}:=\f_{B_{1}\cup \ldots \cup B_h}^{B_{2}},\ldots,\g_{h}:=\f_{B_{1}\cup \ldots \cup B_h}^{B_{h}}
    \]
are cross free of the hypergraph $(E_1,E_2,\ldots,E_h)$.
\end{obs}
\noindent Thus, in order to accomplish Step~3, it is sufficient to show that:
\begin{prop}[Informal]\label{P:2}
Any $h$ $\left(s,\max \left(C\frac{k}{n},e^{k/C}\right) \right)$-uncapturable
families $\g_{1},\ldots,\g_{h}$ cross contain a copy of any $d$-expanded hypergraph.
\end{prop}

Propositions~\ref{P:1} and~\ref{P:2} may be of independent interest. The proof of their formal versions (namely, Propositions~\ref{prop:uncap single k small} and~\ref{prop:uncap cross k large}, respectively) span Section~\ref{sec:uncap-contains}. We describe their structure -- along with the tools we develop for proving them -- in the following subsection.

\subsection{Overview of the proof of Propositions~\ref{P:1} and~\ref{P:2}}
\label{sec:sub:overview:uncap}

The proof has two stages.

\subsection*{Showing that sufficiently large families cross contain any $d$-expanded hypergraph}

We prove Proposition~\ref{Prop:turan for cross} which states that any families $\f_{1},\ldots,\f_{h}$, each of measure at least $\epsilon=\max \left(e^{-k/C},C\frac{k}{n} \right)$, cross contain any $d$-expanded hypergraph $\h'=H^+$ of size $h$, provided that $C$ is sufficiently large.

The proof proceeds in three steps:
\begin{enumerate}
\item \textbf{The Fairness Proposition:} Given a family $\f$, recall that any set $S\subseteq[n]$ induces a partition of $\f$ into the $2^{|S|}$ families $\left\{ \f_{S}^{B}\right\} _{B\subseteq S}$. A set $S$ is said to be $\delta$\emph{-fair} for $\f$ if for any $B\subseteq S$, we have $\mu\left(\f_{S}^{B}\right) \geq (1-\delta)\mu\left(\f\right)$. We show that for any constant $s,\delta$ and any `not-very-small' family $\f\subseteq {{[n]}\choose{k}}$ (formally, $\mu(\f) \geq e^{-k/C}$ for a sufficiently large constant $C$), \emph{almost every} constant-sized subset $S \subseteq \left[n\right]$ is $0.1$-fair for $\f$. The proof of this proposition, which uses a rather simple (but somewhat technical) combination of a Chernoff-type argument with double-counting, spans Section~\ref{sec:fairness}.

\item We use the fairness proposition to find a copy of $H$ of the form $\left(H_{1},\ldots,H_{h}\right)$, such that the families
$\left(\f_{1}\right)_{H_{1}\cup\cdots\cup H_{h}}^{H_{1}},\ldots,\left(\f_{h}\right)_{H_{1}\cup\cdots\cup H_{h}}^{H_{h}}$
are also `large' (almost like the initial families $\f_1,\ldots,\f_h$).

\item We use Theorem~\ref{thm:Proof of the aaroni howrd conjecture} to find a matching
$M_{1}\in\left(\f_{1}\right)_{H_{1}\cup\cdots\cup H_{h}}^{H_{1}},\ldots,M_{h}\in\left(\f_{h}\right)_{H_{1}\cup\cdots\cup H_{h}}^{H_{h}}$. This completes the proof, as $M_{1}\cup H_{1}\in\f_{1},\ldots,M_{h}\cup H_{h}\in\f_{h}$ constitute a copy of $\h'$ cross contained in $\f_1,\ldots,\f_h$.
\end{enumerate}

\subsection*{Showing that an $\left(s,\max \left(C\frac{k}{n},e^{-k/C}\right) \cdot \left(\frac{k}{n}\right)^{t'}\right)$-uncapturable family contains a copy of any $d$-expanded hypergraph with kernel of size $t'$}

To prove this statement (whose $t'=0$ case corresponds to Proposition~\ref{P:1} and whose $t'>0$ case corresponds to Proposition~\ref{P:2}), we present two separate arguments -- one that applies for $k> C\log n$, and another that applies for $C<k<n^{1/C}$.

\mn \textbf{The case $k>C\log n$.} In order to deal with this case, we define the notion of \emph{$\left(l,\alpha\right)$-quasiregularity}. For $\ell \in \mathbb{N}$ and $\alpha>1$, we say that a family $\f\subseteq{{[n]}\choose{k}}$ is $\left(l,\alpha\right)$-\emph{quasiregular} if $\mu\left(\f_{B}^{B}\right)\le\alpha\mu\left(\f\right)$
for any $B\subseteq\left[n\right]$ of size at most $l$. Note that while this is clearly a strong notion of regularity when $\alpha$ is very close to 1, we shall apply this notion mostly with $\alpha=n^{\Theta\left(1\right)}.$

Let $\h$ be a fixed $d$-expanded hypergraph. We show that if $\f\subseteq{{[n]}\choose{k}}$ is an $\left(s,C\left(\frac{k}{n}\right)^{r}\right)$-uncapturable family, for any constant $r$ and a sufficiently large constant $C$, then $\f$ contains a copy of $\h$, through the following steps.
\begin{enumerate}
\item We show that any $\left(s',\Theta(n/k)\right)$-quasiregular families $\g_{1},\ldots,\g_{h}$ cross contain a copy of any fixed ordered hypergraph $\h'$.
\item We use the uncapturability of $\f$ to find pairwise disjoint sets $D_{1},\ldots,D_{h}$ such that the families
\[
\f_{1}:=\f_{D_{1}\cup\cdots\cup D_{h}}^{D_{1}},\ldots,\f_{h}:=\f_{D_{1}\cup\cdots\cup D_{h}}^{D_{h}}
\]
are $\left(s',\Theta(n/k)\right)$-quasiregular.
\end{enumerate}
These two steps complete the proof, since denoting $\h=(H_1,\ldots,H_h)^+$, the families $\f_1,\ldots,\f_h$ cross contain a copy $(B_1,\ldots,B_h)$ of $(H_1,\ldots,H_h)^+$, and thus, $(B_1 \cup D_1,\ldots,B_h \cup D_h)$ is a copy of $\h$ contained in $\f$.


\mn \textbf{The case $C<k<n^{1/C}$.} This is the most complex case, and its treatment spans Section~\ref{sec:Random-sampling}. Let $\f\subseteq{{[n]}\choose{k}}$ be an $\left(s,\max \left(C\frac{k}{n},e^{-k/C}\right) \cdot \left(\frac{k}{n}\right)^{t'}\right)$-uncapturable family, and let $\h$ be a hypergraph with kernel of size $t'-1$. Our goal is to show that $\f$ contains a copy of $\h$. To accomplish this, we employ the Kostochka-Mubayi-Verstra\"{e}te~\cite{kostochka2015turan} method of `random sampling from the shadow', which suggests to study a family $\f$ via its shadow.

First, we reduce the claim to the case $K(\h)=\emptyset$ by performing the following steps:
\begin{enumerate}
\item We show that if $\f\subseteq{{[n]}\choose{k}}$ is uncapturable, then the shadow $\partial(\f)$ is uncapturable as well.

\item We show that for any `fixed' $\h$ with $|K(\h)|=t_0$, there exists a `fixed' hypergraph $\h'$ with kernel of size $t_0-1$, such that any family whose shadow contains $\h'$ must contain $\h$.
\item The above steps allow us to use induction on $t'$, thus reducing to the case $t'=1$.
\end{enumerate}
The most interesting of these steps is~(2), whose proof involves a probabilistic argument.

For dealing with the case $t'=0$, we introduce the following notion. For $\f\subseteq{{[n]}\choose{k}}$, we define a coloring $c\colon\partial\left(\f\right)\to\left[n\right]$ of the shadow of $\f$ by letting $c\left(e\right)$ be some $i\in\left[n\right]$,
such that $e\cup\left\{ i\right\} \in \f$. We say that $\partial(\f)$ contains a \emph{rainbow copy} of a hypergraph $\h'$ if it contains a copy $\left(A_{1},\ldots,A_{h}\right)$ of $\h'$, such that the colors of the sets $A_{1},\ldots,A_{h}$ are distinct.
\begin{enumerate}
\item We show that if $\f$ is $\h$-free, then $\partial\left(\f\right)$ is free of a rainbow copy of a `sufficiently large' expanded
hypergraph. Consequently, if $s$ is sufficiently large, and if $S_{1},\ldots,S_{s} \subseteq [n]$ are pairwise disjoint, then the families $\bigcup_{i\in\s_{j}}\f_{\left\{ i\right\} }^{\left\{ i\right\} }\subseteq {{\left[n\right]}\choose{k-1}}$
are cross free of some expanded hypergraph. This allows us to deduce that $\min_j\left\{ \mu\left(\bigcup_{i\in\s_{j}}\f_{\left\{ i\right\} }^{\left\{ i\right\} }\right)\right\}<e^{-k/C}$, using Proposition~\ref{Prop:turan for cross}.
\item Writing (w.l.o.g.) $\mu\left(\f_{\left\{ 1\right\} }^{\left\{ 1\right\} }\right)\ge\cdots\ge\mu\left(\f_{\left\{ n\right\} }^{\left\{ n\right\} }\right)$, and applying Step~1 several times with appropriate choices of the families $S_{1},\ldots,S_{s} \subseteq [n]$, we deduce
    that $\mu\left(\left(\partial\left(\f\right)\right)_{\left[s\right]}^{\emptyset}\right)\le e^{-k/C}$. This shows that $\partial(\f)$ is capturable, and as mentioned above, this implies that $\f$ is capturable, a contradiction.
\end{enumerate}

\section{Most Constant-Sized Subsets of Not-Very-Small Families are Fair}
\label{sec:fairness}

Recall that for any family $\f\subseteq{{[n]}\choose{k}}$, any set $S \subseteq [n]$ induces a partition of $\f$ into
the $2^{\left|S\right|}$ slices $\left\{ \f_{S}^{B}\right\} _{B\subseteq S}$.
We say that $S$ is \emph{$\delta$-fair} for $\f$, if for any $B\subseteq S$ we have $\mu\left(\f_{S}^{B}\right)\ge\mu\left(\f\right)\left(1-\delta\right)$.
In this section we show that for any constant $\delta>0$ and $s \in \mathbb{N}$, `almost all' the $s$-sized subsets of $[n]$ are $\delta$-fair
for $\f$, provided that $|\f|$ is not too small.
\begin{prop}
\label{prop:Fairness}For any constants $\delta>0,s\in\mathbb{N}$, there exist constants $C,k_0$ that depend only on $s$ and $\delta$,
such that the following holds. Let $k_0 \leq k \leq n-k_0$, and let $\f\subseteq{{[n]}\choose{k}}$ satisfy $\mu\left(\f\right)\ge \max(e^{-k/C},e^{-(n-k)/C})$.

Then the probability that a random $s$-subset $\mathbf{S}\sim {{\left[n\right]}\choose{s}}$ is $\delta$-fair for $\f$ is at least $1-\delta$.
\end{prop}
The proof of the proposition is a rather simple combination of a Chernoff-type argument with double counting. We first present the proof in the
case $s=1$, and then we leverage it to a proof for a general $s$ using an inductive argument.

\subsection{The case $s=1$}

We use the following Chernoff-type bound proved in~\cite{janson2000random}, pp.~27--29.

Let $k,m\in\left[n\right]$. The \emph{hypergeometric random variable} $X$ with parameters $\left(n,k,m\right)$
is defined to be the size of the intersection $\left|\mathbf{T}\cap V\right|$, where $\mathbf{T}\sim{{[n]}\choose{k}}$ is a uniformly chosen random set
of size $k$ and $V\in {{\left[n\right]}\choose{m}}$ is a fixed set of size $m$.
\begin{lem}
\label{lem:Chernoff bound} Let $n,k,m\in\mathbb{N}$, let $a\in\left(0,3/2\right)$,
and let $X$ be a hypergeometric random variable with parameters $\left(n,k,m\right)$.
Then $\Pr\left[\left|X-\frac{km}{n}\right|>a\frac{km}{n}\right]\le e^{-\frac{a^{2}km}{3n}}$.
\end{lem}

\begin{lem}
\label{prop:the case s=00003D1}For any $\delta>0$,
there exist constants $C,k_0$ which depend only on $\delta$
such that the following holds. Let $k_{0}<k<n-k_{0}$, let $\f\subseteq{{[n]}\choose{k}}$,
and let $\mathbf{i}\sim\left[n\right]$ be a randomly chosen element. Then:
\begin{enumerate}
\item If $\mu\left(\f\right)\ge e^{-k/C}$, then $\Pr \left[\mu\left(\f_{\left\{ i\right\} }^{\left\{ i\right\} }\right)\ge\left(1-\delta\right) \mu\left(\f\right) \right] \geq 1-\delta/2$.

\item If $\mu\left(\f\right)\ge e^{-\left(n-k\right)/C}$, then $\Pr \left[\mu\left(\f_{\left\{ i\right\} }^{\emptyset}\right)\ge\left(1-\delta\right)\mu\left(\f\right) \right] \geq 1-\delta/2$.
\end{enumerate}
In particular, if $\mu\left(\f\right)\ge\max\left\{ e^{-k/C},e^{-\left(n-k\right)/C}\right\} $,
then the probability that the singleton $\left\{ i\right\} $ is $\delta$-fair for $\f$ is at least $1-\delta$.
\end{lem}

\begin{proof}
We only show~(1), as~(2) follows by replacing the family $\f$ with the family $\left\{ A^{c}\::\, A\in\f\right\} $, and the `in particular'
statement follows from a simple union bound.

Let $V$ be the set of all `bad coordinates', i.e.,
\[
V= \{i\in\left[n\right]: \mu\left(\f_{\left\{ i\right\} }^{\left\{ i\right\} }\right) < \left(1-\delta\right)\mu\left(\f\right) \}.
\]
Suppose on the contrary that $\left|V\right|\ge \frac{\delta}{2}n$.

Let $\f'$ be the sub-family of $\f$ which consists of all sets that have a `large' intersection with $V$. Formally,
\[
\f' = \{A \in \f: |A \cap V| \geq (1-\frac{\delta}{2})\frac{k}{n}|V|\}.
\]
The proof proceeds in three steps.
\begin{enumerate}
\item First, we use a Chernoff-type argument to show that $\f'$ contains most of the sets in $\f$.

\item Then, we consider the average $\frac{1}{\left|V\right|}\sum_{i\in V}\mu \left(\left(\f'\right){}_{\left\{ i\right\} }^{\left\{ i\right\} } \right)$ and use double counting, along with the above Chernoff-type argument, to show that this average is `large'.

\item We show directly that the apparently larger average $\frac{1}{\left|V\right|}\sum_{i\in V}\mu \left(\f{}_{\left\{ i\right\} }^{\left\{ i\right\} } \right)$ is `small', reaching a contradiction.

\end{enumerate}

\mn \textbf{$\f'$ is large.} By Lemma~\ref{lem:Chernoff bound}, we have
\begin{equation}
\Pr_{\mathbf{A}\sim{{[n]}\choose{k}}}\left[\left|\left|\mathbf{A}\cap V\right|-\frac{k}{n}\left|V\right|\right|\ge\frac{\delta}{2} \frac{k}{n}\left|V\right|\right] \le\exp\left(-\Omega_{\delta}\left(\frac{k\left|V\right|}{n}\right)\right)= \exp\left(-\Omega_{\delta}\left(k\right)\right),\label{eq:fairness1}
\end{equation}
where the last equality uses the assumption $|V| \geq \frac{\delta}{2} n$. Hence,
\begin{align}
\mu\left(\f'\right) & \ge\mu\left(\f\right)-\Pr_{\mathbf{A}\sim{{[n]}\choose{k}}} \left[\left|\mathbf{A}\cap V\right|<\left(1-\frac{\delta}{2}\right) \frac{k}{n}\left|V\right|\right] \ge\mu\left(\f\right)-\exp\left(-\Omega_{\delta}\left(k\right)\right).
\label{eq:fairness 2}
\end{align}
Provided that $C,k_0$ are sufficiently large (as functions of $\delta$), we have
$\exp\left(-\Omega_{\delta}\left(k\right)\right)\le \frac{\delta}{2} e^{-k/C}\le \frac{\delta}{2}\mu\left(\f\right)$. Substituting this into~\eqref{eq:fairness 2}, we obtain
\begin{equation}
\mu\left(\f'\right)\ge\mu\left(\f\right)\left(1-\frac{\delta}{2}\right).\label{eq:remove the prime}
\end{equation}

\mn \textbf{The average $\frac{1}{\left|V\right|}\sum_{i\in V}\mu \left(\left(\f'\right){}_{\left\{ i\right\} }^{\left\{ i\right\} } \right)$ is large.}  Since $\left|A\cap V\right|\ge\left(1-\frac{\delta}{2}\right)\frac{k}{n}\left|V\right|$
for any $A\in\f'$, we have:
\begin{align*}
\frac{1}{\left|V\right|}\sum_{i\in V}\mu\left(\left(\f'\right)_{\left\{ i\right\} }^{\left\{ i\right\} }\right) & =\frac{1}{\left|V\right|}\sum_{i\in V}\frac{\left|\left(\f'\right)_{\left\{ i\right\} }^{\left\{ i\right\} }\right|}{{{n-1}\choose{k-1}}} =\frac{1}{\left|V\right|{{n-1}\choose{k-1}}}\sum_{i\in V}\sum_{A\in\f'}1_{i\in A} =\frac{1}{\left|V\right|{{n-1}\choose{k-1}}}\sum_{A\in\f'}\sum_{i\in V}1_{i\in A} \\
& =\frac{1}{\left|V\right|{{n-1}\choose{k-1}}}\sum_{A\in\f'}\left|A\cap V\right|
 \ge\frac{1}{\left|V\right|{{n-1}\choose{k-1}}}\sum_{A\in\f'}\left(1-\frac{\delta}{2}\right)\frac{k}{n}\left|V\right| \\
& = \frac{1}{\left|V\right|{{n-1}\choose{k-1}}}\left(1-\frac{\delta}{2}\right)\left|\f'\right|\frac{k}{n}\left|V\right|
  = \frac{1}{{{n}\choose{k}}}\left(1-\frac{\delta}{2}\right)\left|\f'\right|
  = \mu\left(\f'\right)\left(1-\frac{\delta}{2}\right).
\end{align*}
By~\eqref{eq:remove the prime}, this implies
\begin{align}
\frac{1}{\left|V\right|}\sum_{i\in V}\mu\left(\left(\f'\right)_{\left\{ i\right\} }^{\left\{ i\right\} }\right) \geq \mu\left(\f'\right)\left(1-\frac{\delta}{2}\right)  \ge\left(1-\frac{\delta}{2}\right)^{2}\mu\left(\f\right)>\left(1-\delta\right)\mu\left(\f\right).
\label{eq:fairness 4}
\end{align}

\mn \textbf{The average $\frac{1}{\left|V\right|}\sum_{i\in V}\mu \left(\f{}_{\left\{ i\right\} }^{\left\{ i\right\} } \right)$ is small.} By the definition of $V$, for any $i \in V$ we have $\mu\left(\f_{\left\{ i\right\} }^{\left\{ i\right\} }\right)<\left(1-\delta\right)\mu\left(\f\right)$. This holds also on the average, and thus,
\[
\frac{1}{\left|V\right|}\sum_{i\in V}\mu\left(\f{}_{\left\{ i\right\} }^{\left\{ i\right\} } \right) < (1-\delta)\mu\left(\f\right),
\]
contradicting~\eqref{eq:fairness 4}. This completes the proof of the lemma.
\end{proof}

\subsection{The general case}

We now reduce the case where $s$ is a general constant to the case $s=1$ which we already proved. We use the following simple
claim.
\begin{claim}
\label{fairness}Let $n>0$, let $S\subseteq\left[n\right]$, let
$i\in S$, and let $\f\subseteq\pn$. Suppose that $\left\{ i\right\} $
is $\frac{\delta}{2}$-fair for $\f$, and that $S\backslash\left\{ i\right\} $
is $\frac{\delta}{2}$-fair for both of the families $\f_{\left\{ i\right\} }^{\left\{ i\right\} },\f_{\left\{ i\right\} }^{\emptyset}$.
Then $S$ is $\delta$-fair for $\f$.\end{claim}
\begin{proof}
Let $\f\subseteq\pn$ be as in the hypothesis, and let
$B\subseteq S$. We have to show that $\mu\left(\f_{S}^{B}\right)\ge\left(1-\delta\right)\mu\left(\f\right)$.
Suppose first that $i\in B$. Since the set
$S\backslash\left\{ i\right\} $ is $\frac{\delta}{2}$-fair for
the family $\f_{\left\{ i\right\} }^{\left\{ i\right\} }$, and since
the singleton $\left\{ i\right\} $ is $\frac{\delta}{2}$-fair for $\f$,
we have
\[
\mu\left(\f_{S}^{B}\right)\ge\left(1-\frac{\delta}{2}\right)\mu\left(\f_{\left\{ i\right\} }^{\left\{ i\right\} }\right)\ge\left(1-\frac{\delta}{2}\right)^{2}\mu\left(\f\right)>\left(1-\delta\right)\mu\left(\f\right).
\]
\noindent Similarly, if $i \not \in B$, we obtain
\[
\mu\left(\f_{S}^{B}\right)\ge\left(1-\frac{\delta}{2}\right)\mu\left(\f_{\left\{ i\right\} }^{\emptyset}\right)\ge\left(1-\frac{\delta}{2}\right)^{2}\mu\left(\f\right)>\left(1-\delta\right)\mu\left(\f\right).
\]

\end{proof}

We are now ready to prove Proposition \ref{prop:Fairness}.
\begin{proof}[Proof of Proposition \ref{prop:Fairness}]
The proof goes by induction on $s$. The case $s=1$ is covered by Lemma~\ref{prop:the case s=00003D1}. Hence, we assume that the assertion holds for all $s \leq s_0-1$, namely, that for any $s \leq s_0-1$ and any $\delta'>0$, there exist $C\left(s,\delta'\right),k_0\left(s,\delta'\right)$ such that if  $k>k_{0}$ and if $\f\subseteq{{[n]}\choose{k}}$ is a family that satisfies
$\mu\left(\f\right)\ge\exp\left(-k/C\right)$, then the probability that a set $\mathbf{S}\sim{{[n]}\choose{s}}$ is $\delta'$-fair for $\f$ is at least $1-\delta'$. We have to show that the same statement holds with respect to $s_0$ and any $\delta>0$.

\noindent Let $\mathbf{S}\sim {{\left[n\right]}\choose{s_0}}$, and let $i \in S$. By Claim \ref{fairness}, we have
\begin{align*}
\Pr\left[\mathbf{S}\mbox{ is not }\delta\mbox{-fair for }\f\right] & \le\Pr\left[\left\{ i\right\} \mbox{ is not }\frac{\delta}{2}\mbox{-fair for }\f\right]\\
 & +\Pr\left[\{i\}\mbox{ is }\frac{\delta}{2}\mbox{-fair for }\f \mbox{ and } S\backslash\left\{ i\right\} \mbox{ is not }\frac{\delta}{2}\mbox{-fair for }\f_{\left\{ i\right\} }^{\emptyset}\right]\\
 & +\Pr\left[\{i\}\mbox{ is }\frac{\delta}{2}\mbox{-fair for }\f \mbox{ and } S\backslash\left\{ i\right\} \mbox{ is not }\frac{\delta}{2}\mbox{-fair for }\f_{\left\{ i\right\} }^{\left\{ i\right\} }\right].
\end{align*}
Hence, we will be done if we can choose $C\left(s_0,\delta\right),k_0\left(s_0,\delta\right)$ such that each of the terms in the right hand side is no larger than $\delta/3$. For the first term, this clearly follows from the induction hypothesis (applied with $s=1,\delta'=\delta/3$). For the third term, note that if $\{i\}$ is $\frac{\delta}{2}$-fair for $\f$ then $\mu \left(\f_{\left\{ i\right\} }^{\{i\}}\right) \geq (1-\frac{\delta}{2})\mu(\f)$, and hence, we can apply the induction hypothesis with $s=s_0-1,\delta'=\delta/3$, provided that $C(s_0,\delta)$ is taken to be sufficiently large so that
\[
(1-\frac{\delta}{2}) e^{-\frac{k}{C(s_0,\delta)}} \geq e^{-\frac{k-1}{C(s_0-1,\delta/3)}}.
\]
(Note that the choice of $C(s_0,\delta)$ can be made independently of $k$, since the statement becomes weaker as $k$ grows.)
Finally, the second term is similar to the third one. This completes the proof of the proposition.
\end{proof}

\section{\label{sec:Random-sampling}The Shadows of $\h$-Free Families}

Let $\h$ be a $d$-expanded hypergraph of size $h$. In this section we study the \emph{shadow} of an
$\h$-free family $\f\subseteq{{[n]}\choose{k}}$.

Our `meta'-goal, which originates in the work of Kostochka, Mubayi, and Verstra\"{e}te~\cite{kostochka2015turan,kostochka2015problems,kostochka2014turan}, is to show that $\h$-freeness of $\f$ imposes a restriction of a similar type on the shadow $\partial(\f)$ (and in some cases, also on the $t$-shadow $\partial^t(\f)$). Once we establish such a result, we shall use it in Sections~\ref{sec:uncap-contains} and~\ref{sec:bootstrapping} to bound the size of $\h$-free families in a three-step procedure:
\begin{enumerate}
\item Deduce from the $\h$-freeness of $\f$ that $\partial(\f)$ satisfies a certain restriction.

\item Deduce an upper bound on the size of $\partial(\f)$ from the restriction it satisfies.

\item Deduce an upper bound on $|\f|$ from the upper bound on the size of $\partial(\f)$, using a result of Kostochka et al.~\cite{kostochka2015turan}.
\end{enumerate}

Unfortunately, the assertion that if $\f$ is free of $\h$ then $\partial(\f)$ is free of some (possibly larger) expanded hypergraph $\h'$, is false, as can be seen in the following example:
\begin{example}\label{Ex:6.1}
Let $\h$ be a hypergraph whose kernel is empty. Then the $(1,1)$-star
$\s_{\{1\}}=\left\{ A\in{{[n]}\choose{k}}\,:\,1\in A\right\} $ is free of $\h$, while its shadow consists of all ${{\left[n\right]}\choose{k-1}}$.
\end{example}

\noindent Hence, in order to obtain a restriction on the shadow $\partial(\f)$, we introduce a \emph{coloring} of $\partial\left(\f\right)$.

\medskip

One can think of $\partial(\f)$ as being naturally colored by the elements of $\left[n\right]$, where a subedge $E\in\partial\left(\f\right)$
is colored by all the elements $i\in\left[n\right]$ such that $E\cup\left\{ i\right\} \in \f$. (Note that in this coloring, each subedge can have more than one color.)
\begin{defn} We say that the shadow $\partial\left(\f\right)$
contains a \emph{rainbow copy} of an ordered hypergraph $\left( E_{1},\ldots,E_{h}\right)$, if there exist $F_1,\ldots,F_h \in \partial(\f)$, a bijection $f:E_1 \cup \ldots \cup E_h \to F_1 \cup \ldots \cup F_h$ such that $f(E_i)=F_i$ for each $i$, and distinct $j_{1} \not \in F_1,\ldots,j_{h} \not \in F_h$ such that all the sets $F_{1}\cup\left\{ j_{1}\right\} ,\ldots,F_{h}\cup\left\{ j_{h}\right\} $ belong to $\f$.
\end{defn}
Taking the coloring of the shadow into consideration, we see that the shadow of $\s_{\{1\}}$ in Example~\ref{Ex:6.1} does satisfy some restrictions. Indeed, it
is easily seen that $\partial(\s_{\{1\}})$ does not contain a rainbow copy of a $3$-matching. It turns out that a similar phenomenon occurs for any forbidden hypergraph $\h$.

\medskip

We prove two propositions, which apply to different types of $\h$. The first applies for any $d$-expanded $\h$ of size $h$, and asserts that if $\f$ is free of $\h$ then $\partial(\f)$ is free of a rainbow copy of a larger $d$-expanded hypergraph $\h'=\left(s \cdot {{[v]}\choose{d}} \right)^+$, for sufficiently large $s,v$.
\begin{prop}
\label{Prop:color the shadow} Let $n,k,d,h \in \mathbb{N}$, and let $\h\subseteq{{[n]}\choose{k}}$ be a $d$-expanded hypergraph
of size $h$. Denote $s:=2h^{2}+1$ and $v:=2dh^{2}+d$.

For any $\h$-free family $\f\subseteq{{[n]}\choose{k}}$, the shadow $\partial\left(\f\right)$ does not contain a rainbow copy of
the hypergraph $\h'=\left(s \cdot {{\left[v\right]}\choose{d}}\right)^{+}$.
\end{prop}

The second proposition applies for a hypergraph $\h$ with kernel of size $t$, and asserts that if $\f$ is free of $\h$, then $\partial(\f)$ is free of a larger hypergraph $\h''$ with kernel of size $t-1$. Since any copy of $\h''$ is forbidden (and not only a rainbow one), this proposition can be applied $t$ times in a row, thus allowing us to use inductive arguments.

To formulate the proposition, we need a convenient notation for hypergraphs with a non-empty kernel. For a hypergraph $H \subseteq \p([m])$ with an empty kernel and and for $t>0$, we denote by $H \oplus [t]$ the hypergraph
\[
\{A \subseteq \p([m+t]): A = A' \cup \{m+1,m+2,\ldots,m+t\} \mbox{ for some } A' \in H\}.
\]
Of course, any hypergraph $\h$ with $|K(\h)|=t$ can be represented in this way (up to reordering the coordinates), for some hypergraph $H$ with an empty kernel.
\begin{prop}
\label{prop:t-shadow} For any constants $t,h,d\in\mathbb{N}$, there exists $C=C\left(t,h,d\right)$, such that the
following holds for any $C<\ell<u/C$.

Let $\h$ be a $d$-expanded hypergraph of size $h$ with kernel of size $t$. For any $\h$-free family $\f\subseteq{{[n]}\choose{k}}$, the shadow $\partial\left(\f\right)$ is free of
the hypergraph $\h''=\left({{\left[u\right]}\choose{\ell}}\oplus\left[t-1\right]\right)^{+}$.
\end{prop}

\subsection{Proof overview}
\label{sec:sub:shadow-overview}

The proof of Proposition~\ref{Prop:color the shadow} is a rather simple probabilistic argument.

\medskip

We write $\h=\h_1^+$ for some $d$-uniform multi-hypergraph $\h_1$ of size $h$, and suppose on the contrary that $\partial(\f)$ contains a rainbow copy of $\h'=\left(s \cdot {{[v]}\choose{d}} \right)^+$, for sufficiently large $s,v$ to be determined below. This means that we can assign \emph{distinct} vertices $i_E$ for each $E \in E(\h')$, in such a way that for any $E \in E(\h')$ we have $E \cup i_E \in \f$.

We now define a \emph{random copy} $\mathbf{(H_1,\ldots,H_h)}$ of $\h$ in the following way: First, we choose a random copy $\mathbf{(E'_1,\ldots,E'_h)}$ of $\h_1$ inside ${{[v]}\choose{d}}$. Then we enlarge this copy into the $h$-tuple $\mathbf{(E_1 \in \partial(\f),\ldots,E_h \in \partial(\f))}$, by randomly choosing one of the $s$ disjoint enlargements of each $\mathbf{E'_i}$ contained in $\partial(\f)$. Finally, our candidate copy of $\h$ in $\f$ is $\mathbf{\tilde{H}=(E_1 \cup \{i_{E_1}\},\ldots,E_h \cup \{i_{E_h}\})}$. It is clear that $\mathbf{\tilde{H}}$ is a copy of $\h$ in $\f$ if and only if each added vertex $\mathbf{i_{E_j}}$ is disjoint from the initial copy of $\h_1$ and with each enlargement $\mathbf{E_l \setminus E'_l}$, for all $l \neq j$. We show that by choosing $v,s$ to be sufficiently large, we can ensure that this occurs with a positive probability, and thus, $\f$ contains a copy of $\h$, a contradiction.

\medskip

The proof of Proposition~\ref{prop:t-shadow} is a bit more complex. We denote $\h = \h_0 \oplus [t]$, and suppose on the contrary that $\partial(\f)$ contains a copy of
\[
\h'' = \left({{\left[u\right]}\choose{\ell}}\oplus\left(t-1\right)\right)^{+}:= \left(\c \oplus [t-1] \right)^+,
\]
for sufficiently large (and appropriately chosen) $u,\ell$. We color the shadow $\partial(\f)$ with the color set $[n]$ such that $(\chi(E)=i) \Rightarrow (E \cup \{i\} \in \f)$. (If there are several possible `colors' for some $E$, we choose one of them arbitrarily.) Thus, each color class $\chi^{-1}(i)$ is contained in $\f_{\{i\}}^{\{i\}}$. The coloring induces a partition of the hypergraph $\c \cong {{[u]}\choose{\ell}}$ into $n$ hypergraphs $\c_1,\ldots,\c_n$, where the edge set of $\c_i$ consists of all edges that are colored $i$ in the copy of $\c$ in $\partial(\f)$.

The proof consists of three steps.
\begin{enumerate}
\item We prove Proposition~\ref{Prop:turan for cross} which asserts that any $h$ `sufficiently large' families $\f_1,\ldots,\f_h \subseteq {{[n]}\choose{k}}$ cross contain a copy of any fixed $d$-expanded hypergraph of size $h$. (Formally, the required measure is $\min(\mu(\f_i)) \geq \max \left(e^{-k/C},C\frac{k}{n} \right)$, where $C$ depends on $d,h$.) The proof of the proposition uses Theorem~\ref{thm:Proof of the aaroni howrd conjecture} and Proposition~\ref{prop:Fairness}.

\item We observe that each $\c_i$ is free of $\h_0$. (Indeed, if $\c_i$ had contained a copy of $\h_0$, then $\chi^{-1}(i)$ would contain a copy of $(\h_0 \oplus [t-1])^+$, and so would $\f_{\{i\}}^{\{i\}}$. Hence, $\f$ would contain a copy of $\h=(\h_0 \oplus [t])^+$, a contradiction.) Therefore, Step~1 allows us to deduce that each $\c_i$ is `small'. We show that using the small size of the $\c_i$'s, we can partition $[n]$ into a constant number of sets $V_1,\ldots,V_q$ such that the families $\cup_{i \in V_j} \c_i$ are of `roughly equal' sizes, and in particular, are `large' subsets of ${{[u]}\choose{\ell}}$. We then use Step~1 again to deduce that the families $\cup_{i \in V_1} \c_i, \ldots, \cup_{i \in V_q} \c_i$ cross contain \emph{any} fixed hypergraph $H$, provided that $u,\ell$ are chosen properly (depending on the hypergraph we want the families to cross contain).

\item Step~2 (applied with some fixed hypergraph $H$ to be determined below) implies that the shadow $\partial(\f)$ contains a rainbow copy of $(H \oplus [t-1])^+$. Denoting the set of coordinates that correspond to $[t-1]$ here by $T'$, this implies that $\partial \left(\f_{T'}^{T'} \right)$ contains a rainbow copy of $H^+$. However, we observe that the family $\f_{T'}^{T'}$ is free of the $(d-t+1)$-expanded hypergraph $\h_0 \oplus [1]$, and thus, by Proposition~\ref{Prop:color the shadow}, its shadow does not contain a rainbow copy of the hypergraph $H = s' \cdot {{[v']}\choose{d-t+1}}$. Applying Step~2 with this $H$ and choosing $u,\ell$ properly, this yields a contradiction.
\end{enumerate}

The rest of this Section is organized as follows. The proofs of Propositions~\ref{Prop:color the shadow} and~\ref{prop:t-shadow} are presented in Sections~\ref{sec:sub:shadow:random} and~\ref{sec:sub:shadow:kernel}, respectively. Finally, in Section~\ref{sec:sub:shadow:sizes} we deduce a relation between the size of an $\h$-free family and the size of its shadows, using a result of Kostochka et al.~\cite{kostochka2015turan} and Proposition~\ref{prop:t-shadow}.

\subsection{Families that are free of a general fixed hypergraph -- proof of Proposition~\ref{Prop:color the shadow}}
\label{sec:sub:shadow:random}

In this section we present the proof of Proposition~\ref{Prop:color the shadow}, following the strategy outlined in Section~\ref{sec:sub:shadow-overview}.

\begin{proof}[Proof of Proposition~\ref{Prop:color the shadow}]
Suppose on the contrary that $\partial\left(\f\right)$ contains a rainbow copy of $s \cdot {{\left[v\right]}\choose{d}}^{+}$ and denote this copy by $\c$. We will show that $\f$ contains a copy of $\h$ (and thus obtain a contradiction) by defining a random procedure for choosing $h$ sets in $\f$, and showing that these sets constitute a copy of $\h$ with a positive probability.

Since $\h$ is $d$-expanded, it can be written in the form $\h=\h_1^+$, for some $d$-uniform multi-hypergraph $\h_1$ of size $h$. Write $\h_1=\left\{ H_{1},\ldots,H_{h}\right\}$ and denote $H=H_{1}\cup\cdots\cup H_{h}$.

It will also be helpful for us to make the fact that $\partial\left(\f\right)$ contains a rainbow copy of $\h'=s \cdot {{\left[v\right]}\choose{d}}^{+}$ more
explicit. Write $m=s{{v}\choose{d}}$, and let $\c=\left\{ C_{1},\ldots,C_{m}\right\} $ be a copy of $\h$ in $\partial\left(\f\right)$. Denote by $V$ the \emph{center} of $\c$ (i.e., the set that corresponds to $\left[v\right]$ in the isomorphism between $\c$ and $\h'=s \cdot {{[v]}\choose{d}}^{+}$).

By definition, there exist distinct $i_{1},\ldots,i_{m}\in\left[n\right]$ such that the sets $C_{1}\cup\left\{ i_{1}\right\} ,\ldots,C_{m}\cup\left\{ i_{m}\right\} $ belong to
$\f$. We define a coloring $\chi\colon\c\to\left[n\right]$ that attaches to each set $C_{j}$ the corresponding singleton $\left\{ i_{j}\right\} $.

We now describe a random procedure for choosing $h$ sets out of $C_{1}\cup\left\{ i_{1}\right\} ,\ldots,C_{m}\cup\left\{ i_{m}\right\} $. As we shall prove below, the chosen sets constitute a copy of $\h$ in $\f$ with a positive probability.

Our random procedure proceeds as follows:
\begin{enumerate}
\item We choose a set $\mathbf{G}\sim {{V}\choose{\left|H\right|}}$. Intuitively, $\mathbf{G}$ should be thought of as a {}``copy''
of the set $H$.
\item We choose uniformly at random a bijection $\mathbf{\pi}\colon H\to\mathbf{G}$,
and set $\mathbf{G}_{1}:=\mathbf{\pi}\left(H_{1}\right),\ldots,\mathbf{G_{h}}:=\mathbf{\pi}\left(H_{h}\right)$.
This makes the hypergraph $\left\{ \mathbf{G}_{1},\ldots,\mathbf{G}_{h}\right\} $ a random copy of $\h_1$.

\medskip

At this stage, note that if we will be able to associate to each $\mathbf{G_{i}}$ a set $\mathbf{E_{i}}$ of the form $C_{j}\cup\left\{ i_{j}\right\} $, such that $\mathbf{G_{i}}\subseteq\mathbf{E_{i}},$ and such that $\mathbf{E_{1}\backslash G_{1},\ldots,E_{h}\backslash G_{h}}$
are pairwise disjoint sets that do not intersect $\mathbf{G}$, then our proof would be completed. We associate sets of the form $C_{j}\cup\left\{ i_{j}\right\} $ randomly in the following way:

\item We choose uniformly at random a set $\mathbf{S}_{i}\supseteq\mathbf{G}_{i}$
out of the $s$ sets $C_{i,1},\ldots,C_{i,s} \in \c$ that contain $\mathbf{G_{i}}$, and denote  $\mathbf{E_{i}}:= \mathbf{S}_{i}\cup\chi\left(\mathbf{S}_{i}\right) \in\f$.
\end{enumerate}

The following claim will complete the proof, by contradicting our hypothesis that the family $\f$ is $\h$-free.
\begin{claim}
The hypergraph
\[
\s=\left\{ \mathbf{E_{1},E_{2},\ldots,E_{h}}\right\} \subseteq\f
\]
is a copy of $\h$ with a positive probability.
\end{claim}
\begin{proof}
%
We define a series of `bad events', such that if none of these bad events occur, then the random sets $\mathbf{E_{1}\backslash G_{1},\ldots,E_{h}\backslash G_{h},G}$ are pairwise disjoint. The events are divided into two types.
\begin{itemize}
\item \textbf{Bad events of type 1:} For $i\in\left[h\right]$, we let $B_{i}$ be the event that the vertex $\chi\left(\mathbf{S_{i}}\right)$ is
included in $\mathbf{G\backslash G_{i}}$.
\item \textbf{Bad events of type 2:} For $i, j\in\left[h\right]$ such that $i \neq j$, we let $B_{ij}$ be the event that the vertex $\chi\left(\mathbf{S_{i}}\right)$ is included in $\mathbf{S_{j}\backslash G_{j}}$ and no bad event of type~1 occurs.
\end{itemize}
It is easy to verify that $\s$ is a copy of $\h$ if none of the above bad events occurs. Thus, by a union bound,
\begin{align}
\Pr\left[\s\mbox{ is a copy of }\h\right] & \ge1-\Pr\left[\bigcup_{i=1}^{h}B_{i} \cup \bigcup_{i=1}^{h}\bigcup_{j\ne i}^{h}B_{i,j} \right] \geq
1-\sum_{i=1}^{h}\Pr\left[B_{i}\right]-\sum_{i=1}^{h}\sum_{j\ne i}^{h}\Pr\left[B_{ij}\right].
\label{eq:union bound 1}
\end{align}
The assertion of the claim would follow from an upper bound on the probabilities $\Pr\left[B_{i}\right],\Pr\left[B_{ij}\right]$.

\mn \textbf{An upper bound on the probability of a bad event of type 1.} Let $i\in\left[h\right]$. We claim that
for each possible value $S_{i}$ of the random set $\mathbf{S_{i}}$, we have $\Pr\left[B_{i}\,|\,\mathbf{S_{i}}=S_{i}\right]\le\frac{\left|\mathbf{G\backslash G_{i}}\right|}{\left|V\backslash\mathbf{G_{i}}\right|}$. Indeed, if $\chi\left(S_{i}\right) \not \in V$ then $\Pr\left[B_{i}\,|\,\mathbf{S_{i}}=S_{i}\right]=0$, and if $\chi\left(S_{i}\right) \in V$, then the probability that $\chi\left(S_{i}\right)$ was chosen to be in the random set $\mathbf{G\backslash G_{i}}$ is $\frac{\left|\mathbf{G\backslash G_{i}}\right|}{\left|V\backslash\mathbf{G_{i}}\right|}.$ Hence,
\begin{equation}
\Pr\left[B_{i}\right]\le\frac{\left|\mathbf{G\backslash G_{i}}\right|}{\left|V\backslash\mathbf{G_{i}}\right|}\le\frac{d\left(h-1\right)}{v-d},\label{eq:Bad1}
\end{equation}
for any $i\in\left[h\right]$.

\mn \textbf{An upper bound on the probability of a bad event of type 2.} Let $i,j$ be distinct elements of $\left[h\right]$. We claim that for each possible pair of values $(G,S_i)$ of the random sets $\mathbf{G},\mathbf{S_{i}}$, we have $\Pr\left[B_{i,j}\,|\,\mathbf{S}_{i}=S_{i},\mathbf{G}=G\right]\le\frac{1}{s}$. Indeed, note that for fixed $G$ and $S_i$, assuming that no bad event of type~1 occurs, the vertex $\chi\left(S_{i}\right)$ belongs to at most one of the sets $C_{j,1},\ldots,C_{j,s}$. If $\chi\left(S_{i}\right)$ does not belong to any of these sets, we have
$\Pr\left[B_{i,j}\,|\,\mathbf{S}_{i}=S_{i},\mathbf{G}=G\right]=0<\frac{1}{s}$. Otherwise, let $l\in\left[s\right]$ be such that $\chi\left(S_{i}\right) \in C_{j,l}$. Then
\[
\Pr\left[B_{i,j}\,|\,\mathbf{S}_{i}=S_{i},\mathbf{G}=G\right]=\Pr\left[\mathbf{S_{j}}=C_{j,l}\right]=\frac{1}{s}.
\]
Therefore, we have
\begin{equation}
\Pr\left[B_{i,j}\right]\le\frac{1}{s}.\label{eq:bad2}
\end{equation}
Plugging~(\ref{eq:Bad1}) and~(\ref{eq:bad2}) into Equation~(\ref{eq:union bound 1}), we obtain
\[
\Pr\left[\s\mbox{ is a copy of }\h\right]>1-\frac{dh^{2}}{v-d}-\frac{h^{2}}{s}>0,
\]
where the latter inequality holds due to the choice of $s$ and $r$. This completes the proof of the claim.
\end{proof}
By the claim, the hypergraph $\left\{ \mathbf{E_{1},E_{2},\ldots,E_{h}}\right\} $
is a copy of $\h$ in $\f$ with a positive probability. This contradicts the hypothesis that $\f$ is $\h$-free.
\end{proof}

\subsection{Families that are free of a hypergraph with a non-empty kernel -- proof of Proposition~\ref{prop:t-shadow}}
\label{sec:sub:shadow:kernel}

In this section we present the proof of Proposition~\ref{prop:t-shadow}, following the steps outlined in Section~\ref{sec:sub:shadow-overview}.

We start with a lemma which shows that `large' families cross contain a copy of any fixed $d$-expanded ordered hypergraph.

\begin{prop}
\label{Prop:turan for cross}
For any constants $d,h$, there exists a constant $C$ that depends only on $d,h$,
such that the following holds. Let $\h$ be a $d$-expanded ordered hypergraph of size $h$. Let $C < k_{1},\ldots,k_{h}< n/C$, and let $\f_{1}\subseteq {{\left[n\right]}\choose{k_{1}}},\ldots,\f_{h}\subseteq {{\left[n\right]}\choose{k_{h}}}$ be families such that for each $i \in [h]$,
\begin{equation}\label{Eq:Shadow-1}
\mu\left(\f_{i}\right)\ge \max \left(e^{-k_i/C},C\frac{k_i}{n} \right).
\end{equation}
Then $\f_1,\ldots,\f_h$ cross contain a copy of $\h$.
\end{prop}

\begin{proof}[Proof of Proposition~\ref{Prop:turan for cross}]
Let $\h_1=(H_{1},\ldots,H_{h})$ be a $d$-uniform ordered hypergraph such that $\h=\h_1^+$, and denote $H=H_{1}\cup\cdots\cup H_{h}$. By Proposition \ref{prop:Fairness} (which can be applied since $\mu(\f_i) \geq e^{-k_i/C}$ for all $i$), there exists a set $G \in {{\left[n\right]}\choose{\left|H\right|}}$ that is $\frac{1}{2}$-fair for each of the families $\f_i$ (simultaneously). Choose an arbitrary bijection $\pi\colon H\to G$,
and write $G_{1}=\pi\left(H_{1}\right),\ldots G_{h}=\pi\left(H_{h}\right)$. By the $\frac{1}{2}$-fairness of $G$, we have
\[
\mu\left(\left(\f_{i}\right)_{G}^{G_{i}}\right)\ge\frac{1}{2}\mu\left(\f_{i}\right)\ge \frac{C}{2}\frac{k_i}{n},
\]
for any $i \in [h]$. This allows us to apply Theorem~\ref{thm:Proof of the aaroni howrd conjecture} to the families
$\left(\f_{1}\right)_{G}^{G_{1}},\ldots,\left(\f_{h}\right)_{G}^{G_{h}}$ and deduce that they cross contain a matching $\left(D_{1},\ldots,D_{h}\right)$, provided that $C$ is large enough.
(Note that while Theorem~\ref{thm:Proof of the aaroni howrd conjecture} deals with the case where all families have the same uniformity, Lemma~3.1 of~\cite{huang2012size} allows deducing the same conclusion also for different uniformities, given a stronger assumption on the sizes of the families; this assumption holds in our setting, provided $C$ is large enough.)

Now, the hypergraph $\left\{ G_{1}\cup D_{1},\ldots,G_{h}\cup D_{h}\right\}$ is a copy of $\h$ and satisfies $G_{i}\cup D_{i}\in\f_{i}$ for each
$i\in\left[h\right]$. Hence, $\f_1,\ldots,\f_h$ cross contain a copy of $\h$, as asserted.
\end{proof}

We now present another lemma which will allow us to perform Step~2 of the proof. This simple lemma essentially shows that given a partition of some set to many `small' disjoint sets, we can group the sets together into a constant number of disjoint subsets whose sizes are roughly equal.
\begin{prop}\label{lem:Step 2 shadows}
Let $0<\delta<1$, let $S$ be a set, and let $f:S \rightarrow [n]$ be a partition of $S$ into $n$ disjoint subsets. If for each $i$ we have $|f^{-1}(i)| \leq \delta |S|$, then there exists a partition $V_1,V_2,\ldots,V_q$ of $[n]$ such that for each $j \in [q]$,
\begin{equation}\label{Eq:Shadow-2}
\left|\cup_{i \in V_j} f^{-1}(i) \right| \leq 2 \delta |S|,
\end{equation}
and for each $j \in [q]$ except for at most one,
\begin{equation}\label{Eq:Shadow-3}
\left|\cup_{i \in V_j} f^{-1}(i) \right| \geq \delta |S|.
\end{equation}
\end{prop}

\begin{proof}
The required partition $V_1,V_2,\ldots,V_q$ is obtained by starting with the sets $f^{-1}(1),\ldots,f^{-1}(n)$, and sequentially unifying pairs of sets which are both of size $<\delta |S|$, until there is no such pair. It is clear by construction that the process terminates after at most $n$ steps, that no resulting set $V_i$ exceeds the size $2\delta|S|$, and that at most one of the $V_i$'s is of size $<\delta |S|$.
\end{proof}

Now we are ready to present the proof of Proposition~\ref{prop:t-shadow}.

\begin{proof}[Proof of Proposition~\ref{prop:t-shadow}]
We denote $\h = \h_0 \oplus [t]$, fix $C=C(t,h,d)$ whose value will be given below, and suppose on the contrary that $\partial(\f)$ contains a copy of
\[
\h'' = \left({{[u]}\choose{\ell}}\oplus\left[t-1\right]\right)^{+},
\]
for some $u,\ell$ that satisfy the assumption of the proposition, with respect to $C$. We color the shadow $\partial(\f)$ with the color set $[n]$ such that $(\chi(E)=i) \Rightarrow (E \cup \{i\} \in \f)$. (If there are several possible `colors' for an edge, we choose one of them arbitrarily.) Thus, each color class $\chi^{-1}(i)$ is contained in $\f_{\{i\}}^{\{i\}}$. The coloring induces a partition of the hypergraph $\c \cong {{[u]}\choose{\ell}}$ into $n$ hypergraphs $\c_1,\ldots,\c_n$.

Denote the set of coordinates that correspond in the copy of the hypergraph $\h''$ to $[t-1]$ by $T'$. As explained in Section~\ref{sec:sub:shadow-overview}, the family $\f_{T'}^{T'}$ is free of the $(d-t+1)$-expanded hypergraph $\h_0 \oplus [1]$. Hence, by Proposition~\ref{Prop:color the shadow}, its shadow does not contain a rainbow copy of the hypergraph $H = s' \cdot {{[v']}\choose{d-t+1}}$, for an appropriate choice of $s',v'$. Fix such a choice of $s',v'$ and set $m=s' {{v'}\choose{d-t+1}}$.

\medskip

As explained in Section~\ref{sec:sub:shadow-overview}, each $\c_i$ is free of $\h_0$. Hence, by Proposition~\ref{Prop:turan for cross} there exists $C_1=C_1(d,h)$ such that if $C_1 < \ell < u/C_1$, we have
\begin{equation}\label{Eq:Shadow-4}
\mu(\c_i) \leq \max \left(e^{-\ell/C_1},C_1 \frac{\ell}{u} \right),
\end{equation}
for all $i$, where $\mu$ denotes the uniform measure on $\c \cong {{\left[u\right]}\choose{\ell}}$.
We would like now to apply Proposition~\ref{lem:Step 2 shadows} to the partition $(\c_1,\ldots,\c_n)$ of $\c$, in such a way that we will be able to apply Proposition~\ref{Prop:turan for cross} to the resulting partition. To do so, we first set a constant $\tilde{C}_2$ to be $C(m,d-t+1)$ in the notations of Proposition~\ref{Prop:turan for cross} (i.e., the constant for which the proposition holds with respect to the parameters $(m,d-t+1)$), and set $C_2=\max(C_1,\tilde{C}_2)$. We then apply Proposition~\ref{lem:Step 2 shadows} to the partition $(\c_1,\ldots,\c_n)$ of $\c$, with the parameter
\[
\delta = \max \left(e^{-\ell/C_2},C_2 \frac{\ell}{u} \right).
\]
(Note that by the choice of $C_2$ and~\eqref{Eq:Shadow-4}, we indeed have $\mu(\c_i) \leq \delta$ for all $i \in [n]$). This is the point at which we determine $C$: we choose it in such a way that for any $C<\ell<u/C$, we have $\delta < \frac{1}{2m+1}$. (This can clearly be done without violating the previous conditions, by taking $C$ to be sufficiently large as function of $d,h,t$).

By Proposition~\ref{lem:Step 2 shadows}, there exists a partition $V_1,\ldots,V_q$ of $[n]$ such that for each $j \in [q]$,
\begin{equation}\label{Eq:Shadow-5}
\mu \left(\cup_{i \in V_j} \c_i \right) \leq 2 \delta,
\end{equation}
and for each $j \in [q]$ except for at most one,
\begin{equation}\label{Eq:Shadow-6}
\mu \left(\cup_{i \in V_j} \c_i \right) \geq \delta.
\end{equation}
Note that $\c$ is a disjoint union of the families $\{\cup_{i \in V_j} \c_i\}_{j \in [q]}$, and thus,
\[
1=\mu(\c) = \sum_{i=1}^q \mu \left(\cup_{i \in V_j} \c_i \right) \leq 2\delta q,
\]
where the inequality follows from~\eqref{Eq:Shadow-5}. Since $\delta <\frac{1}{2m+1}$, this implies that $q \geq m+1$. Hence, by~\eqref{Eq:Shadow-6}, there exist $m$ indices $j_1,\ldots,j_m$ such that $\mu \left(\cup_{i \in V_j} \c_i \right) \geq \delta$ for all $j \in \{j_1,\ldots,j_m\}$.

We now apply Proposition~\ref{Prop:turan for cross} to the families $\cup_{i \in V_{j_1}} \c_i, \ldots, \cup_{i \in V_{j_m}} \c_i$. By the choice of $\delta$, Proposition~\ref{Prop:turan for cross} implies that these families cross contain a copy of the hypergraph $H$. Therefore, the shadow $\partial(\f)$ contains a rainbow copy $D$ of the hypergraph $(H \oplus [t-1])^+$ in which the set of coordinates that correspond to $[t-1]$ is $T'$ (defined at the beginning of the proof).

Finally, we consider the family $\f_{T'}^{T'}$ and its shadow. The previous paragraph implies that $\partial(\f_{T'}^{T'})$ contains a rainbow copy of the hypergraph $H$. This is a contradiction, since (as explained at the beginning of the proof) the family $\f_{T'}^{T'}$ is free of the $(d-t+1)$-expanded hypergraph $\h_0 \oplus [1]$, and thus, by Proposition~\ref{Prop:color the shadow}, its shadow does not contain a rainbow copy of $H$. This completes the proof.
\end{proof}

\subsection{Relation between the size of an $\h$-free family and the sizes of its shadows}
\label{sec:sub:shadow:sizes}

The following proposition, essentially proved by Kostochka, Mubayi, and Verstra\"{e}te~\cite[Lemmas~3.1 and~3.2]{kostochka2015turan}, allows us to bound the size of an $\h$-free family $\f$ in terms of the size of its shadow.
\begin{prop}
\label{lem:kostochka-mubayi-verstraete}Let $k>dh+1$, and let $\h \subseteq {{[n]}\choose{k}}$ be a $d$-expanded hypergraph of size $h$. For any $\h$-free family $\f\subseteq{{[n]}\choose{k}}$, we have $\left|\f\right|\le kh\left|\partial\left(\f\right)\right|$.
\end{prop}
As the proposition is proved in~\cite{kostochka2015turan} only in a special case, we present its proof for a general $\h$ for the sake of completeness.
\begin{proof}
For $\ell \in \mathbb{N}$, a family $\f \subseteq {{[n]}\choose{k}}$ is called $\ell$-full, if for any $A \in \partial(\f)$, we have $|\f_{A}^{A}| \geq \ell$. The assertion is an immediate consequence of the following two lemmas:
\begin{itemize}
\item For any $k \geq 2, \ell \geq 1$, any family $\f \subseteq {{[n]}\choose{k}}$ has an $(\ell+1)$-full subfamily $\f'$ with $|\f'| \geq |\f| - \ell|\partial(\f)|$.
\item For any hypergraph $\h$ with $h$ edges of size $\leq d$ and any $k>dh+1$, any non-empty $(kh+1)$-full family $\f \subseteq {{[n]}\choose{k}}$ contains a copy of $\h^+$.
\end{itemize}
The easy first lemma is exactly~\cite[Lemma~3.1]{kostochka2015turan}. The second lemma is proved in~\cite[Lemma~3.2]{kostochka2015turan} for the case where $\h$ is a path or a cycle. To prove it for a general $\h$, observe that the assertion would follow once we show:
\begin{claim}
Let $\f \subseteq {{[n]}\choose{k}}$ be a non-empty $\ell$-full family. For any $E,E_1,\ldots,E_m \in \f$ such that $|E \cup E_1 \cup \ldots \cup E_m|<\ell$   and for any $S \subset E$, there exists $E' \in \f$ such that $E' \cap (E \cup E_1 \cup \ldots \cup E_m) = S$.
\end{claim}
Indeed, given the claim, one can show that $\f$ contains a copy of $\h^+$ constructively, as follows. First, one embeds $V(\h)$ into $V' \subset E$, for an arbitrary $E \in \f$; thus, the edges of $\h$ correspond to $E'_1,E'_2,\ldots,E'_h \subset V'$. Then one constructs the edges of the copy, $E_1,E_2,\ldots,E_h \in \f$, one by one, such that at each step $i$, one has $E_i \cap (E \cup E_1 \cup \ldots \cup E_{i-1}) = E'_i$. It is clear that the resulting sub-hypergraph of $\f$ is a copy of $\h^+$.

\medskip \noindent \emph{Proof of the Claim.} Consider the family $\g=\{A \in \f: S \subseteq A \cap (E \cup E_1 \cup \ldots \cup E_m)\}$. Note that $\g$ is non-empty, as $E \in \g$. Let $E' \in \g$ be such that $|E' \cap (E \cup E_1 \cup \ldots \cup E_m)|$ is minimal. We claim that $E' \cap (E \cup E_1 \cup \ldots \cup E_m)=S$, which will prove the assertion. Assume on the contrary $v \in (E' \cap (E \cup E_1 \cup \ldots \cup E_m)) \setminus S$. Consider $E' \setminus \{v\} \in \partial(\f)$. Since $\f$ is $\ell$-full and $|E \cup E_1 \cup \ldots \cup E_m|<\ell$, there exists $v' \not \in E \cup E_1 \cup \ldots \cup E_m$, such that $E'' := (E' \setminus \{v\}) \cup \{v'\} \in \f$. However, $|E'' \cap (E \cup E_1 \cup \ldots \cup E_m)| < |E' \cap (E \cup E_1 \cup \ldots \cup E_m)|$, a contradiction. This completes the proof of the claim and of Proposition~\ref{lem:kostochka-mubayi-verstraete}.
\end{proof}

Using Proposition \ref{prop:t-shadow}, we can generalize Proposition~\ref{lem:kostochka-mubayi-verstraete} to the $t$-shadow, where $(t-1)$ is the size of the kernel of $\h$.
\begin{prop}
\label{cor:t-shadows}
For any constants $d,h,t$, there exists a constant $C=C(d,h,t)$ such that the following holds. Let $C < k < n/C$ and let $\h\subseteq{{[n]}\choose{k}}$ be a $d$-expanded hypergraph of size $h$ whose kernel is of size $t-1$. For any $\h$-free family  $\f\subseteq{{[n]}\choose{k}}$, we have
\[
\left|\f\right|\le C k^{t}|\partial^{t}\left(\f\right)|.
\]
\end{prop}
\begin{proof}
By Proposition \ref{prop:t-shadow}, for any $l \leq t-1$, the $l$'th shadow $\partial^l(\f)$ is free of a $d_l$-expanded hypergraph $\h_l$ of size $h_l$, where $d_l,h_l$ depend only on $d,h,t$. Hence, Proposition \ref{lem:kostochka-mubayi-verstraete} can be applied $t$ times in a row, to the families $\f,\partial(\f),\partial^2(\f),\ldots$, to yield the assertion. (Note that the condition $C < k < n/C$ is required in the application of Proposition \ref{prop:t-shadow}).
\end{proof}

\section{Uncapturable Families Contain Any Fixed $d$-Expanded Hypergraph}
\label{sec:uncap-contains}

Let $d,h$ be constants, and let $\h$ be a $d$-expanded ordered hypergraph of size $h$. In this section we prove two propositions. The first essentially asserts that for $C\log n< k <n/C$, any $h$ uncapturable families of $k$-sets cross contain a copy of $\h$. Formally:
\begin{prop}
\label{prop:uncap cross k large} For any constants $d,h,r$, there exist constants $C,s$ that depend only on $d,h,r$,
such that the following holds. Let $\h$ be a $d$-expanded ordered hypergraph of size $h$. Let $C\log n < k_{1},\ldots,k_{h} < n/C$, and let $\f_{1}\subseteq{{[n]}\choose{k_1}},\ldots,\f_{h}\subseteq{{[n]}\choose{k_h}}$
be $\left(s,\left(\frac{k_{i}}{n}\right)^{r}\right)$-uncapturable families. Then $\f_1,\ldots,\f_h$ cross contain a copy of $\h$.
\end{prop}
Note that the hypothesis $k_{i}> C\log n$ is necessary, due to the following example.
\begin{example}
Let $\{S_{1},\ldots,S_{h}\}$ be a balanced partition of $\left[n\right]$. Then the families $\f_{1}={{S_{1}}\choose{k}},\ldots,\f_{h}={{S_{h}}\choose{k}}$
are cross free of any hypergraph of size $h$ except for the matching hypergraph $\m_{h}$ that consists of $h$ pairwise disjoint edges. However, for any constant $s$,
the families $\f_{i}$ are $\left(s,\frac{k}{n}\right)$-uncapturable if $k\le c\log n$ for a sufficiently small constant $c$.
\end{example}

The second proposition holds for any $C< k < n/C$, but applies only in the `single family' setting. Here we essentially show that any `not too small' uncapturable family contains a copy of $\h$. Formally:
\begin{prop}
\label{prop:uncap single k small} For any constants $d,h,r,t$, there exist constants $C,s$ that depend only on $d,h,r,t$,
such that the following holds. Let $\h$ be a $d$-expanded hypergraph of size $h$ with a kernel of size $t-1$. Let $C \le k\le n/C$, and let $\f\subseteq{{[n]}\choose{k}}$ be an $\left(s,\epsilon \left(\frac{k}{n}\right)^{t}\right)$-uncapturable family, where
$\epsilon=\max \left(e^{-k/C},C \left(\frac{k}{n}\right)^r \right)$. Then $\f$ contains a copy of $\h$.
\end{prop}

\subsection{Proof overview}

A central ingredient in the proofs of both propositions is the following new notion of regularity, that is somewhat reminiscent of $\epsilon$-fairness.
\begin{defn}
A family $\f\subseteq{{[n]}\choose{k}}$ is said to be $\left(h,\alpha\right)$-quasiregular
if $\mu\left(\f_{B}^{B}\right)\le\alpha\mu\left(\f\right)$ for any
$B\subseteq\left[n\right]$ with $|B| \leq h$.
\end{defn}
It is clear that when $\alpha$ is very close to $1$, quasiregularity imposes a strong restriction on the structure of $\f$. However, we shall consider this notion mainly with much larger values of $\alpha$, at the vicinity of $n/k$. We will show that quasiregularity implies uncapturability (with appropriate parameters), and that on the other hand, uncapturable families can be `upgraded' to quasiregular ones by looking at appropriate slices.

\medskip

The proof of Proposition \ref{prop:uncap cross k large} consists of three steps. Let $\h_1 \subseteq {{\left[hd\right]}\choose{d}}$ be the `base' multi-hypergraph, such that $\h_1^{+}=\h$.
\begin{enumerate}
\item We first show that there exist `small' disjoint sets $T_{1},\ldots,T_{h}$, such that the families $\g_{1}:=\left(\f_{1}\right)_{T_{1}\cup\cdots\cup T_{h}}^{T_{1}},\ldots,\g_{h}:=\left(\f_{h}\right)_{T_{1}\cup\cdots\cup T_{h}}^{T_{h}}$ are quasiregular. (This is the `upgrade' from uncapturability to quasiregularity mentioned above.)

\item We use the quasiregularity to show that there exists a copy of $\h_1$ of the form $\left( H_{1},\ldots,H_{h}\right) $,
such that the families $\left(\g_{1}\right)_{H_{1}\cup\cdots\cup H_{h}}^{H_{1}},\ldots,\left(\g_{h}\right)_{H_{1}\cup\cdots\cup H_{h}}^{H_{h}}$
are uncapturable. This part relies on Proposition~\ref{prop:Fairness}.

\item We apply Proposition \ref{prop:uncapturable} to find a matching $(M_1,M_2,\ldots,M_h)$ such that
\[
M_{1}\in\left(\g_{1}\right)_{H_{1}\cup\cdots\cup H_{h}}^{H_{1}},\ldots,M_{h}\in\left(\g_{h}\right)_{H_{1}\cup\cdots\cup H_{h}}^{H_{h}}.
\]
Now, the sets
\[
M_{1}\cup H_{1}\cup T_{1}\in\f_{1},\ldots,M_{h}\cup H_{h}\cup T_{h}\in\f_{h}
\]
constitute a copy of $\h$, and thus, the families $\f_{1},\ldots,\f_{h}$ cross contain a copy of $\h$.
\end{enumerate}

Proposition \ref{prop:uncap single k small} is proved by induction on $|K(\h)|$ and is mainly based on
the relation between $\f$ and its shadow explored in Section \ref{sec:Random-sampling}. We let $\f$ be an $\h$-free family and
want to show that $\f$ is capturable. Denote $\a_i = \f_{\{i\}}^{\{i\}}$ for all $i \in [n]$ and assume w.l.o.g. $\mu(\a_1) \ge \ldots \ge \mu(\a_n)$. The proof of the induction basis proceeds in three steps.
\begin{enumerate}
\item We use Proposition~\ref{Prop:color the shadow} to assert that there exist $m,l$ such that any $s=l {{m}\choose{d}}$ of the $\a_i$'s are cross free of a `multi-colored' copy of the hypergraph $l\cdot {{[m]}\choose{d}}$, and deduce that one of them must be `small'.

\item By a slightly more involved argument, we deduce that not only the families $\a_i$ are `small' for all $i \geq s$, but also that $\mu(\cup_{i>s} \a_i)$ is `small'.

\item We deduce that the shadow of the family $\f_{[s]}^{\emptyset}$ is `small', and hence, by Proposition \ref{lem:kostochka-mubayi-verstraete},
$\f_{\left[s\right]}^{\emptyset}$ is `small' as well, meaning that $\f$ is capturable.
\end{enumerate}
The induction step is easy, using Proposition~\ref{prop:t-shadow} which asserts that if $\f$ is free of a hypergraph with kernel of size $t$, then its shadow
is free of a `not much larger' hypergraph with kernel of size $t-1$.

\mn The rest of this section is organized as follows. In Section~\ref{sec:sub:quasiregularity} we study the relation between quasiregularity and uncapturability. The proof of Proposition~\ref{prop:uncap cross k large} is presented in Section~\ref{sec:sub:uncap-proof-cross}, and the proof of Proposition~\ref{prop:uncap single k small} is presented in Section~\ref{sec:sub:uncap-proof-small}.

\subsection{Uncapturability and quasiregularity}
\label{sec:sub:quasiregularity}

In this subsection we study the relation between the new notion of quasiregularity and the notion of uncapturability discussed in the previous sections. Our first proposition asserts that quasiregularity, even with $\alpha$ as large as $cn/k$ (for a sufficiently small constant $c$) is sufficient to imply uncapturability. Moreover, we show that if $\f$ is quasiregular and $t$ is a constant, then for any set $T$ of size $t$ that is $\epsilon$-fair for $\f$, each slice $\f_{T}^{B}$ (for $B \subseteq T$) is uncapturable.
\begin{prop}
\label{lem:random and fair imply ultra uncapturable} Let $s \geq 1$ and $t \geq 0$ be integers, let $0<\alpha<1$ and $0 \leq \epsilon<1$, and write $c=\frac{\alpha\left(1-\epsilon\right)}{\left(s+t\right)}$. Let
$k<cn$, let $\f\subseteq{{[n]}\choose{k}}$ be a $\left(t+1,c\frac{n}{k}\right)$-quasiregular family, and let $T$
be a $t$-element set that is $\epsilon$-fair for $\f$. Then for any
$B\subseteq T$, the slice $\f_{T}^{B}$ is $\left(s,\left(1-\alpha\right)\mu\left(\f\right)\right)$-uncapturable.
\end{prop}

\begin{proof}
First we prove the assertion for $t=0$. That is, we assume that $\f$ is $\left(1,c\frac{n}{k}\right)$-quasiregular and have to show that $\f$ itself is
$\left(s,\left(1-\alpha\right)\mu\left(\f\right)\right)$-uncapturable. Suppose, on the contrary, that
there is a set $S$ of size $s$ such that $\mu\left(\f_{S}^{\emptyset}\right)\le\left(1-\alpha\right)\mu\left(\f\right)$.
We have
\begin{align*}
\mu\left(\f\right) = \frac{\left|\f\right|}{{{n}\choose{k}}} & \leq \frac{\sum_{i\in S}\left|\f_{\left\{ i\right\} }^{\left\{ i\right\} }\right|+\left|\f_{S}^{\emptyset}\right|}{{{n}\choose{k}}}=\sum_{i\in S}\frac{{{n-1}\choose{k-1}}}{{{n}\choose{k}}}\mu\left(\f_{\left\{ i\right\} }^{\left\{ i\right\} }\right)+\mu\left(\f_{S}^{\emptyset}\right)\frac{{{n-s}\choose{k}}}{{{n}\choose{k}}}.
\end{align*}
By the quasiregularity of $\f$, we have $\mu\left(\f_{\left\{ i\right\} }^{\left\{ i\right\} }\right)<\frac{\alpha}{s} \frac{n}{k}\mu\left(\f\right)$,
and by the assumption on $S$, $\mu\left(\f_{S}^{\emptyset}\right)\frac{{{n-s}\choose{k}}}{{{n}\choose{k}}}\le (1-\alpha)\mu\left(\f\right)$.
Thus,
\[
\mu\left(\f\right)\leq \sum_{i\in S}\frac{{{n-1}\choose{k-1}}}{{{n}\choose{k}}}\mu\left(\f_{\left\{ i\right\} }^{\left\{ i\right\} }\right)+\mu\left(\f_{S}^{\emptyset}\right)\frac{{{n-s}\choose{k}}}{{{n}\choose{k}}} < \mu\left(\f\right)\left(\alpha+1-\alpha\right)=\mu\left(\f\right),
\]
a contradiction.

\mn Now we consider a general $t \geq 0$. Let $|T|=t$ and $B \subseteq T$. By the quasiregularity of $\f$,
the family $\f_{B}^{B}$ is $\left(t+1-\left|B\right|,\frac{\mu\left(\f\right)}{\mu\left(\f_{B}^{B}\right)}c\frac{n}{k}\right)$-quasiregular.
Note that by the fairness of $T$, we have $\mu(\f_B^B) \geq (1-\epsilon)\mu(\f)$, and thus, $\frac{\mu\left(\f\right)}{\mu\left(\f_{B}^{B}\right)}c \leq \frac{1}{1-\epsilon} \frac{\alpha(1-\epsilon)}{s+t} \leq \frac{\alpha}{s+t}$. Hence, $\f_B^B$ is $\left(1,\frac{\alpha}{s+t}\frac{n}{k}\right)$-quasiregular.
Therefore, the above proof of the case $t=0$ implies that $\f_{B}^{B}$ is $\left(s+t,\left(1-\alpha\right)\mu\left(\f\right)\right)$-uncapturable.
Hence, the family $\f_{T}^{B}$ is $\left(s,\left(1-\alpha\right)\mu\left(\f\right)\right)$-uncapturable.
This completes the proof.
\end{proof}

\medskip

Our second proposition asserts that uncapturable families can be `upgraded' to quasiregular families, namely, that if $\f_{1},\ldots,\f_{h}$ are uncapturable
families, then we may find `small' pairwise disjoint sets $T_{1},\ldots,T_{h}$, such that the families $\left(\f_{1}\right)_{T_{1}\cup\cdots\cup T_{h}}^{T_{1}},\ldots,\left(\f_h\right)_{T_{1}\cup\cdots\cup T_{h}}^{T_{h}}$ are quasiregular.
\begin{prop}
\label{lem:Step 2 uncap} For any constants $h,h',r,c$, there exist constants $C,s$ that depend only on $h,h',r,c$,
such that the following holds. Let $C < k_{1},\ldots,k_{h} < n/C$
and let $\f_{1}\subseteq{{[n]}\choose{k_1}},\ldots,\f_{h}\subseteq{{[n]}\choose{k_h}}$
be families. Suppose that each family $\f_{i}$ is $\left(s,\left(\frac{k_{i}}{n}\right)^{r}\right)$-uncapturable.
Then there exist pairwise disjoint sets $T_{1},\ldots,T_{h}$ of size at most $2h'r$, such that for any $i\in\left[h\right],$
the family $\left(\f_{i}\right)_{T_{1}\cup\ldots\cup T_{h}}^{T_{i}}$
is $\left(h',c\frac{n}{k_{i}}\right)$-quasiregular, and $\mu \left(\left(\f_{i}\right)_{T_{1}\cup\cdots\cup T_{h}}^{T_{i}} \right) \ge\frac{1}{2}\left(\frac{k_{i}}{n}\right)^{r}$.
\end{prop}
We shall need the following simple claim.
\begin{claim}
\label{lem: Every family contains a big pseudorandom slice}Let $r,h$ be constants, and let $\f\subseteq{{[n]}\choose{k}}$
be a $k$-uniform family. If $\mu\left(\f\right)\ge (k/n)^{r}$, then there exists a set $T$ of size $\le2h'r$ such that the slice
$\f_{T}^{T}$ is an $\left(h',\sqrt{n/k}\right)$-quasiregular family, and $\mu\left(\f_{T}^{T}\right)\ge\mu\left(\f\right)$.
\end{claim}

\begin{proof}
If $\f$ itself is $\left(h',\sqrt{n/k}\right)$-quasiregular, we are done. Otherwise,
we may find some set $S_{0}$ of size at most $h'$, such that $\mu\left(\f_{S_{0}}^{S_{0}}\right)\ge\sqrt{n/k} \cdot \mu\left(\f\right)$.
We now repeat this process with $\f_{S_{0}}^{S_{0}}$ instead of $\f$. The process will terminate after at most $2r$ iterations.
\end{proof}
\begin{proof}[Proof of Proposition \ref{lem:Step 2 uncap}]
Let $s,C$ be sufficiently large constants to be determined below. Since $\f_{1}$ is $\left(s,\left(\frac{k_{1}}{n}\right)^{r}\right)$-uncapturable,
we have $\mu\left(\f_1\right)\ge\left(\frac{k_{1}}{n}\right)^{r}$. By Claim \ref{lem: Every family contains a big pseudorandom slice},
there exists a set $T_{1}$ of size at most $2h'r$ such that the slice $\left(\f_{1}\right)_{T_{1}}^{T_{1}}$ is $\left(h',\sqrt{\frac{n}{k_{1}}}\right)$-quasiregular and $\mu(\left(\f_{1}\right)_{T_{1}}^{T_{1}})\geq \mu(\f_1)$.
Since the family $\left(\f_{2}\right)_{T_{1}}^{\emptyset}$ is $\left(s-\left|T_1\right|,\left(\frac{k_{2}}{n}\right)^{r}\right)$-uncapturable,
there exists a set $T_{2}$ of size at most $2h'r$
such that the slice $\left(\f_{2}\right)_{T_{1}\cup T_{2}}^{T_{2}}$
is $\left(h',\sqrt{\frac{n}{k_{2}}}\right)$-quasiregular. Continuing in this fashion, we obtain that there exist pairwise disjoint sets $T_{1},T_{2},\ldots,T_{h}$, each of size at most $2h'r$, such that each family $\left(\f_{i}\right)_{T_{1}\cup T_{2}\cup\cdots\cup T_{i}}^{T_{i}}$
is $\left(h',\sqrt{\frac{n}{k_{i}}}\right)$-quasiregular, and such
that $\mu\left(\left(\f_{i}\right)_{T_{1}\cup T_{2}\cup\cdots\cup T_{i}}^{T_{i}}\right)\ge\left(\frac{k_{i}}{n}\right)^{r}$.

Let $i\in\left[h\right]$. Write $\a$ for the $\left(h',\sqrt{\frac{n}{k_{i}}}\right)$-quasiregular
family $\left(\f_{i}\right)_{T_{1}\cup T_{2}\cup\cdots\cup T_{i}}^{T_{i}}$, and denote
$S=T_{i+1}\cup\cdots\cup T_{h}$. It is clear that in order to complete the proof of the proposition, it is sufficient to show that
for any $i$, the corresponding family $\a_{S}^{\emptyset}$ is $\left(h',c\frac{n}{k_{i}}\right)$-quasiregular,
and satisfies $\mu\left(\a_{S}^{\emptyset}\right)\ge\frac{1}{2}\left(\frac{k_i}{n}\right)^{r}$. (Note that while we suppress the index $i$ in the
definition of $\a$ and $S$ for sake of clarity, $\a$ and $S$ do depend on $i$.)

\medskip

The family $\a \subseteq {{[n] \setminus (T_1 \cup \ldots \cup T_i)}\choose{k-|T_i|}}$ is $\left(h',\sqrt{\frac{n}{k_{i}}}\right)$-quasiregular, and thus, for a sufficiently large $C$ it is $\left(1,\frac{1}{2|S|}\frac{n-|T_1 \cup \ldots \cup T_i|}{k-|T_i|}\right)$-quasiregular. Hence, by Proposition~\ref{lem:random and fair imply ultra uncapturable} (applied with $t=0,s=|S|,\alpha=1/2$), the family $\a$ is $(|S|,\mu(\a)/2)$-uncapturable. In particular,
\begin{align}\label{Aux-1}
\mu\left(\a_{S}^{\emptyset}\right)\ge \frac{\mu\left(\a\right)}{2}\ge\frac{1}{2}\left(\frac{k_i}{n}\right)^{r}.
\end{align}
Suppose on the contrary that $\a_{S}^{\emptyset}$ is not $\left(h',c\frac{n}{k_{i}}\right)$-quasiregular, i.e., that there exists a set $B$ with $|B| \leq h'$ such that $\mu\left(\a_{S\cup B}^{B}\right)\ge c\frac{n}{k_{i}}\mu\left(\a_{S}^{\emptyset}\right)$. In such a case, we have
\[
\mu\left(\a_{B}^{B}\right)\ge\Pr_{\mathbf{A}\sim {{\left[n\right]\backslash B}\choose{k_i-\left|B\right|}}}\left[\mathbf{A}\cap S=\emptyset\right]\mu\left(\a_{S \cup B}^{B}\right) \ge c'\mu \left(\a_{S \cup B}^B \right) \geq cc' \frac{n}{k_i} \mu(\a_S^{\emptyset}) \ge \frac{cc'}{2} \frac{n}{k_i} \mu(\a),
\]
for some $c'$ that depends on $h,h',r,c$ (where the second to last equality follows from the assumption on $B$ and the last inequality follows from~\eqref{Aux-1}).

However, provided that $C$ is sufficiently large, this contradicts the fact that
\[
\mu\left(\a_{B}^{B}\right) \le \sqrt{\frac{n}{k_{i}}}\mu\left(\a\right),
\]
which follows from the $\left(h',\sqrt{\frac{n}{k_{i}}}\right)$-quasiregularity of $\a$. This completes the proof.
\end{proof}

\subsection{The case of a large $k$ -- proof of Proposition~\ref{prop:uncap cross k large}}
\label{sec:sub:uncap-proof-cross}

In this subsection we prove Proposition~\ref{prop:uncap cross k large}, namely, that for $C\log n < k < n/C$,
uncapturable families cross contain any fixed ordered expanded hypergraph. First we prove the same assertion under the stronger assumption
of quasiregularity, and then we prove Proposition~\ref{prop:uncap cross k large} using the `upgrade' from uncapturability to quasiregularity presented
above.

\begin{prop}
\label{prop: pseudorandomness} For any constants $d,h,r \in \mathbb{N}$,
there exists a constant $c=c\left(d,h,r\right)$ such that the
following holds. Let $k_{1},\ldots,k_{h} \leq cn$, and let $\f_{1}\subseteq{{[n]}\choose{k_1}},\ldots,\f_{h}\subseteq {{[n]}\choose{k_{h}}}$.
Suppose that each family $\f_{i}$ is $\left(dh+1,c\frac{n}{k_{i}}\right)$-quasiregular.
If $\f_{1},\ldots,\f_{h}$ are cross free of some $d$-expanded hypergraph $\h$ of size $h$, then there exists
$i\in\left[h\right]$, such that
\[
\mu\left(\f_{i}\right)\le \min \left(e^{-ck_{i}},\frac{3}{2}\left(\frac{k_{i}}{n}\right)^{r} \right).
\]
\end{prop}
\begin{proof}
Let $c$ be a small constant to be determined below and let $\f_{1}\subseteq{{[n]}\choose{k_1}},\ldots,\f_{h}\subseteq {{[n]}\choose{k_{h}}}$
be families that satisfy the assumptions of the proposition. Suppose that for all $i$, $\mu\left(\f_{i}\right) \geq e^{-ck_{i}}$. We shall show
that there exists $i\in\left[h\right]$, such that $\mu\left(\f_{i}\right) \leq \frac{3}{2}\left(\frac{k_{i}}{n}\right)^{r}$, which will complete the proof.

Write $\h=\left(A_{1},\ldots,A_{h}\right)^{+}$ for sets $A_{1},\ldots,A_{h}$ of size $d$, and let $A=A_{1}\cup\cdots\cup A_{h}$. Since
$\mu\left(\f_{i}\right) \geq e^{-ck_{i}}$ for all $i$, Proposition~\ref{prop:Fairness} (that can be applied if $c$ is sufficiently small,
as function of $h,d$) implies that there exists an $\left|A\right|$-set $A'$ that is $\frac{1}{3}$-fair for each of the families $\f_{1},\ldots,\f_{h}$.

Let $\left( A_{1}',\ldots,A'_{h}\right) \subseteq\p\left(A'\right)$
be a copy of the ordered hypergraph $\left( A_{1},\ldots,A_{h}\right) $.
Since the families $\f_{1},\ldots,\f_{h}$ are cross free of $\h=(A_1,\ldots,A_h)^+$,
the families $\left(\f_{1}\right)_{A'}^{A'_{1}},\ldots,\left(\f_{h}\right)_{A'}^{A_{h}'}$
are \emph{cross free of a matching}. Indeed, if there existed pairwise disjoint sets
$D_{1}\in\left(\f_{1}\right)_{A'}^{A_{1}'},\ldots,D_{h}\in\left(\f_{h}\right)_{A'}^{A_{h}'}$,
then the sets $D_{1}\cup A_{1}'\in\f_{1},\ldots,D_{h}\cup A_{h}'\in\f_{h}$
would constitute a copy of $\h$.

By Proposition \ref{prop:uncapturable} (applied to the families $\left(\f_{1}\right)_{A'}^{A'_{1}},\ldots,\left(\f_{h}\right)_{A'}^{A_{h}'}$),
there exists $i\in\left[h\right]$ and a constant $s=s\left(r,d,h\right)$, such that the family $\left(\f_{i}\right)_{A'}^{A_{i}^{'}}$
is $\left(s,\left(\frac{k}{n}\right)^{r}\right)$-capturable.

On the other hand, as $A'$ is $(1/3)$-fair for $\f_i$,
by Proposition~\ref{lem:random and fair imply ultra uncapturable} (which can be applied provided $c$ is sufficiently small; note that this is the place
where we use the $(dh+1,\cdot)$-quasiregularity of $\f_i$), the family $\left(\f_{i}\right)_{A'}^{A'_{i}}$ is $\left(s,\frac{2}{3}\mu(\f_i)\right)$-uncapturable.

Therefore, we have $\frac{2}{3}\mu\left(\f_i\right)<\left(\frac{k_i}{n}\right)^{r}$.
This completes the proof of the proposition.
\end{proof}

The following corollary is immediate, using the fact that for $k \gg \log n$, we have $e^{-\Omega(k)} \ll (k/n)^{O(1)}$.
\begin{cor}\label{Cor:aux-1}
For any constants $d,h,r>0$, there exist constants $C,c$ that depend only on $d,h,r$ such that the
following holds. Let $C \log n < k_{1},\ldots,k_{h} < n/C$, and let $\f_{1}\subseteq{{[n]}\choose{k_1}},\ldots,\f_{h}\subseteq {{[n]}\choose{k_{h}}}$.
Suppose that each family $\f_{i}$ is $\left(dh+1,c\frac{n}{k_{i}}\right)$-quasiregular.
If for all $i \in [h]$ we have $\mu(\f_i) \ge 2\left(\frac{k_i}{n}\right)^r$, then $\f_1,\ldots,\f_h$ cross contain any
$d$-expanded hypergraph $\h$ of size $h$.
\end{cor}

Now we are ready to prove Proposition~\ref{prop:uncap cross k large}.

\begin{proof}[Proof of Proposition \ref{prop:uncap cross k large}]
Let $\h$ be a $d$-expanded ordered hypergraph of size $h$. Let $C,s$ be large constants to be determined below,
let $C\log n< k_{1},\ldots,k_{h}< n/C$, and let
$\f_{1}\subseteq{{[n]}\choose{k_1}},\ldots,\f_{h}\subseteq{{[n]}\choose{k_h}}$
be families such that for any $i \in [h]$, $\f_{i}$
is $\left(s,\left(\frac{k_{i}}{n}\right)^{r}\right)$-uncapturable.

Let $c>0$ be a small constant to be determined below. Proposition~\ref{lem:Step 2 uncap}, that can be applied if $s,C$ are sufficiently large (as function of
$d,h,r,c$), assures that there exist pairwise disjoint sets $T_{1},\ldots,T_{h}$,
such that each family $\left(\f_{i}\right)_{T_{1}\cup\cdots\cup T_{h}}^{T_{i}}$ is
$\left(dh+1,c\frac{n}{k_{i}}\right)$-quasiregular, and such that
\begin{align}\label{Eq:Aux-0}
\mu\left(\left(\f_{i}\right)_{T_{1}\cup\cdots\cup T_{h}}^{T_{i}}\right)\ge\frac{1}{2}\left(\frac{k_{i}}{n}\right)^{r}> 4 \left(\frac{k_{i}}{n}\right)^{r+1} > 2\left(\frac{k_{i}-|T_i|}{n-|T_1 \cup \ldots \cup T_h|}\right)^{r+1},
\end{align}
(where the two last inequalities hold, assuming $C$ is sufficiently large as function of $r,d,h$).

Thus, we may apply Corollary~\ref{Cor:aux-1} to the families
$\left(\f_{1}\right)_{T_{1}\cup\cdots\cup T_{h}}^{T_{1}},\ldots,\left(\f_{h}\right)_{T_{1}\cup\cdots\cup T_{h}}^{T_{h}}$
(with the parameters $(d,h,r+1)$; note that this requires $c$ to be sufficiently small as function of $d,h,r$, and in turn, by
the previous paragraph requires $s,C$ to be sufficiently large as function of $d,h,r$) to deduce that these families cross
contain $\h$.

Let $B_1 \in \left(\f_{1}\right)_{T_{1}\cup\cdots\cup T_{h}}^{T_{1}}, \ldots, B_h \in \left(\f_{h}\right)_{T_{1}\cup\cdots\cup T_{h}}^{T_{h}}$ be
a copy of $\h$ cross contained in the families $\left(\f_{1}\right)_{T_{1}\cup\cdots\cup T_{h}}^{T_{1}}, \ldots, \left(\f_{h}\right)_{T_{1}\cup\cdots\cup T_{h}}^{T_{h}}$.
Then the sets $B_{1}\cup T_{1}\in\f_{1},\ldots,B_{h}\cup T_h\in\f_{h}$ form a copy of $\h$ cross contained in $\f_1,\ldots,\f_h$. This completes the proof.
\end{proof}

\subsection{The case of a small $k$ -- proof of Proposition~\ref{prop:uncap single k small}}
\label{sec:sub:uncap-proof-small}

In this subsection we present the proof of Proposition~\ref{prop:uncap single k small}. The proof is by induction on $|K(\h)|$. We first prove the assertion for $t=1$ (i.e., the case of empty kernel), which requires most
of the work, and then present the easier
induction step. Throughout this section, $\f\subseteq{{[n]}\choose{k}}$ denotes a family that is free of a $d$-expanded hypergraph $\h$
of size $h$. Note that since the case $k \gg \log n$ was covered by Proposition~\ref{prop:uncap cross k large}, we may assume that $k < C'\log n$
for a sufficiently large constant $C'$.
\begin{prop}
\label{prop:uncap k small t=00003D1} For any constants $d,h,r$, there exist constants $C,s$ that depend only on $d,h,r$,
such that the following holds. Let $\h$ be a $d$-expanded hypergraph of size $h$. Let $C < k< n/C$, and let $\f\subseteq{{[n]}\choose{k}}$ be an $\left(s,\epsilon \frac{k}{n}\right)$-uncapturable family, where
$\epsilon=\max \left(e^{-k/C},C (k/n)^r \right)$. Then $\f$ contains a copy of $\h$.
\end{prop}

The proof of the proposition uses the results of Section~\ref{sec:Random-sampling} so let us briefly recall the relevant notions.

The shadow $\partial(\f)$ is said to contain a rainbow copy of a hypergraph $\h^1=\left\{ A_{1},\ldots,A_{s}\right\} $
if there is a copy $\{B_{1},\ldots,B_{s}\}$ of $\h^1$ and distinct elements $v_{1},\ldots,v_{s}$, such that $B_1,\ldots,B_s \in \partial(\f)$ and $B_{1}\cup\left\{ v_{1}\right\} ,\ldots,B_{s}\cup\left\{ v_{s}\right\} \in \f$. Proposition~\ref{Prop:color the shadow} asserts
that if $\f$ is $\h$-free, then the shadow $\partial(\f)$ is free of a rainbow copy of the hypergraph $\h'=l \cdot {{\left[m\right]}\choose{d}}^{+}$,
provided that $l$ and $m$ are sufficiently large constants. This implies the following:
\begin{claim}\label{Cl:Aux-2}
Let $s=l{{m}\choose{d}}$. For any pairwise disjoint sets $U_{1},U_{2},\ldots,U_{s}\subseteq\left[n\right]$, the families
\begin{equation}\label{Eq:Aux-2}
\g_{1}:=\bigcup_{j\in U_{1}}\f_{\left\{ j\right\} }^{\left\{ j\right\} },\ldots,\g_{s}:=\bigcup_{j\in U_{s}}\f_{\left\{ j\right\} }^{\left\{ j\right\} }
\end{equation}
(which are subsets of $\partial(\f)$) are cross free of the hypergraph $\h'=l \cdot {{\left[m\right]}\choose{d}}^{+}$ as well.
\end{claim}

\begin{proof}[Proof of Proposition \ref{prop:uncap k small t=00003D1}]
Let $r>0$ and let $\h$ be a $d$-expanded hypergraph of size $h$. Let $C,s$ be sufficiently large constants to be determined below.
Let $C < k< n/C$, and let $\f\subseteq{{[n]}\choose{k}}$ be an $\h$-free family. We want to show that $\f$ is
 $\left(s,\epsilon \frac{k}{n}\right)$-capturable, where
$\epsilon=\max \left(e^{-k/C},C (\frac{k}{n})^r \right)$.
By Proposition~\ref{prop:uncap cross k large}, we may assume that $k< C'\log n$, for some
$C'=C'\left(d,h,r\right)$. By Proposition \ref{Prop:color the shadow}, there exist constants
$m=m\left(d,h\right),l=l\left(d,h\right)$, such that $\partial(\f)$ is free of a rainbow copy of the hypergraph $\h' = l \cdot {{[m]}\choose{d}}^{+}$.
Set $s:=l{{m}\choose{d}}$. Denote $\a_i=\f_{\{i\}}^{\{i\}}$ for any $i \in [n]$ and assume
without loss of generality that $\mu(\a_1) \geq \mu(\a_2) \geq \ldots \geq \mu(\a_n)$.

The proof consists of four steps:

\mn \textbf{Step 1: An upper bound on the size of families that are cross free of the hypergraph $\h'=l \cdot {{[m]}\choose{d}}^{+}$.}
We show that if $\b_1,\ldots,\b_{s} \subseteq {{\left[n\right]}\choose{k-1}}$ are families that are cross free of $\h'$ then there exists $i \in [s]$ such that $\mu(\b_i) \le e^{-k/C''}$, for a sufficiently large constant $C''$. This is established by the following claim.
\begin{claim}
\label{Claim: Step 1 Proposition uncap t=00003D1 k small}There exists a constant $C''=C''\left(d,h,r\right)$ such that the following
holds. Let $\b_{1},\ldots,\b_{s}\subseteq {{\left[n\right]}\choose{k-1}}$ be families that are cross free of the hypergraph $\h' = l \cdot {{[m]}\choose{d}}^{+}$.
Then $\min_{i=1}^{s}\left\{ \mu\left(\b_{i}\right)\right\} \le e^{-k/C''}$.
\end{claim}
\begin{proof}
The claim follows immediately from Proposition~\ref{Prop:turan for cross}, using the assumption $k< C'\log n$.
\end{proof}

\mn \textbf{Step 2: An upper bound on $\mu(\a_{s})$.}
Let $C''$ be the constant of Claim \ref{Claim: Step 1 Proposition uncap t=00003D1 k small}.
Applying Claim \ref{Claim: Step 1 Proposition uncap t=00003D1 k small} to the families $\a_1,\ldots,\a_{s}$ (which satisfy its assumptions by Claim~\ref{Cl:Aux-2}), we get
\begin{equation}\label{Eq:Aux-3.5}
\mu(\a_{s})= \min_{i \in [s]} \mu(\a_i) \le e^{-k/C''}.
\end{equation}

\mn \textbf{Step 3: An upper bound on $\mu \left( \cup_{j > s} \a_j \right)$.}
Now, we partition the indices $\{s+1,s+2,\ldots,n\}$ into $s$ disjoint sets $U_1,\ldots,U_{s}$  such that
$\max_j \mu \left(\cup_{i \in U_j} \a_i \right) - \min_j \mu \left(\cup_{i \in U_j} \a_i \right)$ is minimal. As for all $i > s$ we have $\mu(\a_i) \leq \mu(\a_{s}) \leq e^{-k/C''}$, the `most even' partition satisfies
\begin{equation}\label{Eq:Aux-4}
\max_j \mu \left(\bigcup_{i \in U_j} \a_i \right) - \min_j \mu \left(\bigcup_{i \in U_j} \a_i \right) \leq 2e^{-k/C''}.
\end{equation}
Applying Claim \ref{Claim: Step 1 Proposition uncap t=00003D1 k small} to the families $\g_1= \cup_{i \in U_1} \a_i,\ldots,\g_{s}=\cup_{i \in U_{s}} \a_i$ (which satisfy its assumptions by Claim~\ref{Cl:Aux-2}), we get $\min_{j \in [s]} \mu(\g_j) \leq e^{-k/C''}$. Hence, by~\eqref{Eq:Aux-4},
$\max_{j \in [s]} \mu(\g_j) \leq 3e^{-k/C''}$. Therefore,
\begin{equation}\label{Eq:Aux-5}
\mu \left(\bigcup_{i=s+1}^n \a_i \right) \leq \sum_{j \in [s]} \mu(\g_j) \leq 3se^{-k/C''}.
\end{equation}

\mn \textbf{Step 4: Deducing that $\f$ is capturable.}
We observe that
\[
\partial\left(\f_{\left[s\right]}^{\emptyset}\right)\subseteq\bigcup_{i=s+1}^{n}\left(\f_{\left\{ i\right\} }^{\left\{ i\right\} }\right) = \bigcup_{i>s} \a_i.
\]
Hence, $\mu \left(\partial \left(\f_S^{\emptyset} \right) \right) \leq 3s e^{-k/C''}$.
Since $\f_{[s]}^{\emptyset}$ is free of $\h$, Proposition~\ref{lem:kostochka-mubayi-verstraete} implies that
\[
\mu\left(\f_{[s]}^{\emptyset}\right)\le\frac{k^{2}h}{n-s-k+1}\mu\left(\partial\left(\f_{[s]}^{\emptyset}\right)\right) \le\frac{k^{2}h}{n} \cdot 4se^{-k/C''} \le e^{-k/C} \cdot \frac{k}{n},
\]
provided that $C$ is sufficiently large. Therefore, $\f$ is $(s,\epsilon \frac{k}{n})$-capturable, as asserted.
\end{proof}

Now we are ready to present the proof of Proposition \ref{prop:uncap single k small} for all $t$.
\begin{proof}[Proof of Proposition \ref{prop:uncap single k small}]
Let $r>0$ and let $\h$ be a $d$-expanded hypergraph of size $h$ with $|K(\h)|=t-1$. Let $C,s$ be sufficiently large
constants to be determined below. Let $C<k<n/C$ and let $\f \subseteq {{[n]}\choose{k}}$ be an $\h$-free family. We have to show that
$\f$ is $\left(s,\epsilon\left(\frac{k}{n}\right)^{t}\right)$-capturable, where $\epsilon=\max \left(e^{-k/C},C\left(\frac{k}{n}\right)^{r}\right)$.

The proof is by induction on $t$. The case $t=1$ was proved in Proposition~\ref{prop:uncap k small t=00003D1}. Thus, let $t_0>1$, assume that the statement holds for all $t<t_0$, and assume that $|K(\h)|=t_0-1$.

By Proposition~\ref{prop:uncap cross k large}, we may assume that $k< C'\log n$ for some constant $C'=C'\left(d,h,r\right)$.
Thus,
\[
\epsilon = \max\left\{ e^{-k/C},C\left(\frac{k}{n}\right)^{r}\right\} = e^{-k/C},
\]
provided that $C$ is sufficiently large.

By Proposition~\ref{prop:t-shadow}, there exist constants $u,\ell$ depending only on $d,h,t_0$, such that $\partial(\f)$ is free
of the hypergraph ${{[u]}\choose{\ell}}^{+}\oplus\left[t_0-2\right]$. Applying the induction hypothesis to $\partial(\f)$ we
obtain that there exist constants $s,c$ which depend only on $d,h,t_0$, such that the family $\partial\left(\f\right)$ is $\left(s,e^{-ck} \left(\frac{k}{n}\right)^{t_0-1}\right)$-capturable.

Let $S$ be a set of size $s$, such that
$\mu\left(\left(\partial\left(\f\right)\right)_{S}^{\emptyset}\right)\le e^{-ck} \left(\frac{k}{n}\right)^{t_0-1}$.
Since $\f_{S}^{\emptyset}$ is free of $\h$, Proposition~\ref{lem:kostochka-mubayi-verstraete} implies that
\begin{align*}
\mu\left(\left(\f\right)_{S}^{\emptyset}\right) & \le\frac{k^{2}h}{n}\mu\left(\partial\left(\f_{S}^{\emptyset}\right)\right)
\le\frac{k^{2}h}{n}\mu\left(\left(\partial\left(\f\right)\right)_{S}^{\emptyset}\right)\\
 & \le khe^{-ck}\left(\frac{k}{n}\right)^{t_0}\le e^{-k/C}\left(\frac{k}{n}\right)^{t_0},
\end{align*}
provided that $C$ is sufficiently large. Therefore, $\f$ is $(s,\epsilon \left(\frac{k}{n} \right)^{t_0})$-capturable, as asserted.
\end{proof}

\section{Cross Containment when Almost All Families are Large}
\label{sec:bootstrapping}

In this section we present the `bootstrapping' step which essentially asserts that if some families $\f_{1} \subseteq {{[n]}\choose{k_1}},\ldots,\f_{h} \subseteq {{[n]}\choose{k_h}}$ are cross free of a fixed $d$-expanded ordered hypergraph $\h$ of size $h$ and the families $\f_{1},\ldots,\f_{h-1}$ are `very large', then the family $\f_h$ must be `very small'.

In order to make the results we prove more intuitive, we begin in Section~\ref{sec:sub:boot:motivation} with a short motivation that explains the place of this bootstrapping step in the `large picture' of the proof. We then present the results and outline the proofs in Section~\ref{sec:sub:boot:results}. The detailed proofs are presented in Section~\ref{sec:sub:boot:proofs}.

\subsection{Motivation}
\label{sec:sub:boot:motivation}

To demonstrate how the problem of `cross containment when almost all families are large' fits into our proof strategy, we consider the special case of determining how large can a family $\f\subseteq{{[n]}\choose{k}}$ be, given that it does not contain a special $d-$simplex (i.e., expansion of the hypergraph $\{\{2,3,\ldots,d+1\},\{1,3,\ldots,d+1\},\ldots,\{1,2,\ldots,d\}\}$). As mentioned in the introduction, our goal here is to show that under some restrictions on $n,k$, we have $\left|\f\right|\le{{n-1}\choose{k-1}}$, with equality if and only if $\f$ is a $(1,1)$-star.

In this special case, our proof strategy translates into the following:
\begin{enumerate}
\item We first show that any family $\f$ that is free of a special $d$-simplex can be approximated by a junta that is free of a special $d$-simplex.

\item We then show that any `sufficiently large' junta that is free of a special $d$-simplex is actually a $(1,1)$-star.

\mn Steps 1,2 together imply that if $\f$ is a `sufficiently large' family that is free of a special $d$-simplex then it is essentially contained in a $(1,1)$-star. In other words, there exists $i\in\left[n\right]$ such that $\mu\left(\f_{\left\{ i\right\} }^{\left\{ i\right\} }\right)=1-\epsilon$, for some `small' $\epsilon$.

\item The third step is to bootstrap the above `stability' result and show that $\left|\f\right|\le{{n-1}\choose{k-1}}$, with equality if and only if $\epsilon=0$.
\end{enumerate}
To accomplish Step~(3), we consider the families $(\f_1,\ldots,\f_{d+1})$, where  $\f_{1}=\f_{2}=\cdots=\f_{d}:=\f_{\left\{ i\right\} }^{\left\{ i\right\} }$ and $\f_{d+1}:=\f_{\left\{ i\right\} }^{\emptyset}$. We observe that since $\f$ is free of a special $d$-simplex, then these families are cross free of the ordered hypergraph
\[
\left(\left\{2,3,\ldots,d\right\} , \{1,3,4,\ldots,d\},\ldots,\left\{ 1,2,\ldots,d-1\right\} ,\left\{ 1,2,\ldots,d\right\} \right).
\]
As we have
\begin{equation}\label{Eq:Aux8-1}
\mu\left(\f\right)=\frac{k}{n}\mu\left(\f_{\left\{ i\right\} }^{\left\{ i\right\} }\right)+\left(1-\frac{k}{n}\right)\mu\left(\f_{\left\{ i\right\} }^{\emptyset}\right),
\end{equation}
it is sufficient to show that if $\mu(\f_{\{i\}}^{\{i\}})=1-\epsilon$ then $\mu(\f_{\{i\}}^{\emptyset})< \frac{k}{n} \epsilon$. This will follow once we show that if some families $\a_1,\ldots,\a_h$ are cross free of a fixed hypergraph $\h$, and $\a_1,\ldots,\a_{h-1}$ are `very large' (i.e., have measure $\geq 1-\epsilon$), then $\a_h$ must be very small (i.e., satisfy $\mu(\a_h) < \frac{k}{n} \epsilon$). We prove results of this kind in this section.

The meaning of the result in the case described above is that if we start with $\epsilon=0$ (in which the family is equal to the $(1,1)$-star $\{S: i \in S\}$) and allow to add elements to $\f_{\left\{ i\right\} }^{\emptyset}$, then the `gain' from adding these elements is smaller than the {}`cost' that we have to pay by removing elements from $\f_{\left\{ i\right\} }^{\left\{ i\right\} }$ (so that the family will remain free of a special $d$-simplex). We note that when the special $d$-simplex is replaced with a general forbidden hypergraph $\h$ with kernel of size $t-1$, we will have to show that $\mu\left(\a_{h}\right)$ is even smaller -- specifically, is small relatively to $\epsilon \left(\frac{k}{n}\right)^{t}.$

\subsection{Results and proof overview}
\label{sec:sub:boot:results}

We prove three propositions, applicable to different ranges of $k_1,\ldots,k_h$ and different assumptions on the sizes of the families. For sake of clarity, in the informal statements we write `$k$' instead of $k_1,\ldots,k_h$.

\subsubsection{The families $\f_1,\ldots,\f_{h-1}$ are extremely large} The first proposition applies for any $C < k < n/C$ but requires that the measures of $\f_1,\ldots,\f_{h-1}$ are very close to 1. It asserts that $\f_h$ must be `very small'.
\begin{prop}
\label{prop:unbalanced cross free of an hypergraph} For any constants $d,h,r$, there exists a constant $C=C(d,h,r)$ such that the following holds. Let $C < k_{1},\ldots,k_{h} < n/C$, write $k_{\min}=\min\left\{ k_{1},\ldots,k_{h}\right\} $, and let $\h$ be a $d$-expanded ordered hypergraph of size $h$.

Let $\f_{1}\subseteq{{[n]}\choose{k_1}},\ldots,\f_{h}\subseteq{{[n]}\choose{k_h}}$ be families that are cross free of $\h$, and suppose that
\[
\mu\left(\f_{1}\right),\ldots,\mu\left(\f_{h-1}\right)\ge 1-\epsilon,
\]
for some $\epsilon \leq \left(\frac{k_{\min}}{n}\right)^{2d}$. Then $\mu\left(\f_{h}\right)\le C\epsilon^{r}$.
\end{prop}

The proof of Proposition~\ref{prop:unbalanced cross free of an hypergraph} is a rather simple reduction to the case where the `forbidden' hypergraph $\h$ is a matching that was already dealt with in Section~\ref{sec:Baby Case}, as follows:

Let $\h=\h_1^+$ for some $d$-uniform ordered hypergraph $\h_1$. Using the extremely large size of $\f_1,\ldots,\f_{h-1}$, we show that one can find a copy $(B_1,\ldots,B_h)$ of $\h_1$ such that (denoting $B=\cup_{i=1}^h B_i$) the measure of each of the families $(\f_i)_B^{B_i}$ is not much smaller than $\mu(\f_i)$. Since the families $(\f_1)_B^{B_1},\ldots,(\f_h)_B^{B_{h}}$ are cross free of a matching, we can use Proposition~\ref{prop: unbalanced matching lemma} to deduce that $(\f_h)_B^{B_{h}}$ is `very small', and hence, $\f_h$ is `very small' as well.

\subsubsection{$k$ is large} The second proposition applies when $k>C\log n$, for a sufficiently large constant $C$. In this case, it is sufficient to assume that $\f_1,\ldots,\f_{h-1}$ are `moderately large' to deduce that $\f_h$ is `very small'.
\begin{prop}
\label{prop:unbalanced cross free of hypergraph all negligible}
For any constants $d,h,r$, there exists a constant $C=C(d,h,r)$ such that the following holds. Let $C\log n < k_{1},\ldots,k_{h} < n/C$ and let $\h$ be a $d$-expanded ordered hypergraph of size $h$.

Let $\f_{1}\subseteq{{[n]}\choose{k_1}},\ldots,\f_{h}\subseteq{{[n]}\choose{k_h}}$ be families that are cross free of $\h$, and suppose that for each $i \in [h-1]$, we have $\mu(\f_i) \geq C \frac{k_i}{n}$. Then $\mu\left(\f_{h}\right)\le C\left(\frac{k_h}{n}\right)^{r}$.
\end{prop}

Note that the hypothesis $k> C\log n$ is necessary. Indeed, if $S_{1}\sqcup\cdots\sqcup S_{h}=n$ is an even partition of $n$ then the families ${{S_{1}}\choose{k}},\ldots,{{S_{h}}\choose{k}}$ are cross free of any hypergraph of size $h$ except for the $h$-matching, while $\mu\left({{S_{h}}\choose{k}}\right)$ is `not very small' when $k=o\left(\log n\right)$.

\medskip The proof of Proposition~\ref{prop:unbalanced cross free of hypergraph all negligible} uses the results on uncapturable families obtained in Section~\ref{sec:uncap-contains}. The families $\f_1,\ldots,\f_{h-1}$ are trivially uncapturable due to their large size, and for $\f_h$, we use Proposition~\ref{lem:associated junta lemma} to approximate it by a junta $\j$ such that for each $B \in \j$, the family $(\f_h)_B^B$ is uncapturable. Since for each such $B$, the families
\[
(\f_1)_B^{\emptyset},\ldots,(\f_{h-1})_B^{\emptyset},(\f_h)_B^B
\]
are uncapturable and cross free of $\h$, this contradicts Proposition~\ref{prop:uncap cross k large}, unless the junta $\j$ is empty (and so, there are no such $B$'s). Since $\j$ approximates $\f_h$, this implies that $\f_h$ is very small, as asserted.

The place in the proof where the assumption $k_{min} > C\log n$ is used is the application of Proposition~\ref{prop:uncap cross k large}; indeed, as noted in Section~\ref{sec:uncap-contains}, this proposition does not hold when $k=o(\log n)$.

\subsubsection{$k$ is small and $\f_h$ is free of a (possibly another) $d'$-expanded hypergraph} The most complex case is where $k=O(\log n)$ and the families $\f_1,\ldots,\f_{h-1}$ are not extremely large. In this case, one cannot guarantee that the inequality $\mu\left(\f_{h}\right) \leq O(k/n)$ (which is the assertion we will need in view of Equation~\eqref{Eq:Aux8-1}) holds without additional assumptions, even if the measures of $\f_1,\ldots,\f_{h-1}$ are close to 1. For example, if $\h$ consists of two edges that intersect in a single element then a counterexample of the form $\f_{1}={{S}\choose{k}},\f_{2}={{\left[n\right]\backslash S}\choose{k}}$ can be easily found. What we show is that we can deduce $\mu\left(\f_{h}\right) \leq O(k/n)$ (and even stronger bounds) under the additional assumption that $\f_h$ is free of some $d'$-expanded hypergraph $\h'$ (which possibly differs from $\h$).

To see why the additional assumption on $\f_{h}$ may make sense, let's return to the case discussed in Section~\ref{sec:sub:boot:motivation} where $\f$ is a family that is free of a special $d$-simplex, such that $\mu\left(\f_{\left\{ i\right\} }^{\left\{ i\right\} }\right)\ge1-\epsilon$. In that case, we consider the families $\f_1,\ldots,\f_{d+1}$, where $\f_{1}=\cdots=\f_{d}=\f_{\left\{ i\right\} }^{\left\{ i\right\} }$ and $\f_{d+1}=\f_{\left\{ i\right\} }^{\emptyset}$, and want to apply to them our technique since they are cross free of the hypergraph $\left(\left\{2,3,\ldots,d\right\} , \{1,3,4,\ldots,d\},\ldots,\left\{ 1,2,\ldots,d-1\right\} ,\left\{ 1,2,\ldots,d\right\} \right)$. It is clear that in this case, the family $\f_{d+1}$ is in itself free of a special $d$-simplex. As we shall see in Section~\ref{sec:proof}, such situation occurs for general forbidden hypergraphs as well.
\begin{prop}
\label{prop: unbalnced cross free of hypergraph k small} For any constants $d,d',h,h',r$, there exists $C=C(d,d',h,h',r)$ such that
the following holds. Let $C < k_{1},\ldots,k_{h} < n^{1/3}/C$, let $\h$ be a $d$-expanded ordered hypergraph of size $h$, and let $\h'$ be a $d'$-expanded hypergraph of size $h'$ whose kernel is of size $t-1$.

Let $\f_{1}\subseteq{{[n]}\choose{k_1}},\ldots,\f_{h}\subseteq{{[n]}\choose{k_h}}$ be families that are cross free of $\h$, and suppose in addition that $\f_h$ is free of $\h'$. If for some $\epsilon>0$,
\[
\mu\left(\f_{1}\right),\ldots,\mu\left(\f_{h-1}\right)\ge1-\epsilon,
\]
then
\[
\mu\left(\f_{h}\right)\le C\frac{k_{h}^{2t}}{n^{t}}\epsilon^{r}.
\]
\end{prop}

The proof of Proposition~\ref{prop: unbalnced cross free of hypergraph k small} uses the results on the relation between the size of an $\h'$-free family and the size of its shadows obtained in Section~\ref{sec:Random-sampling}, as follows:

First, we use Proposition~\ref{prop:unbalanced cross free of an hypergraph} to show that we can assume w.l.o.g. that $\epsilon \geq n^{-1/3}$ (as otherwise, the measures of $\f_1,\ldots,\f_{h-1}$ are `sufficiently close to 1' for applying Proposition~\ref{prop:unbalanced cross free of an hypergraph}). We then use a probabilistic coupling argument to deduce that $\mu(\partial^{k_h-d}(\f_h)) = O(\epsilon)$. By Lemma~\ref{lem:Kruskal Katona going down}, this implies $\mu(\partial^{t}(\f_h)) = O(\epsilon^r)$. Finally, we use Proposition~\ref{cor:t-shadows} (which exploits the assumption that $\f_h$ is free of $\h'$) to deduce that $\mu(\f_h) \leq C \frac{k_h^{2t}}{n^t} \epsilon^r$, as asserted.

\subsection{Proofs}
\label{sec:sub:boot:proofs}


In the proof of Proposition \ref{prop:unbalanced cross free of an hypergraph} we use the following simple observation.
\begin{claim}\label{Cl:Simple}
Let $n,k \in \mathbb{N}$ be such that $k<n$. For any family $\f \subseteq {{[n]}\choose{k}}$, and for any $d_1>d_2$, we have
\[
\mu(\f) = \mathbb{E}_{\mathbf{B_1} \sim {{\left[n\right]}\choose{d_1}}, \mathbf{B_2}\sim {{\mathbf{B_1}}\choose{d_2}}} \mu \left(\f_{\mathbf{B_1}}^{\mathbf{B_2}} \right).
\]
\end{claim}

\noindent The proof of the observation, using a simple coupling argument or direct counting, is omitted.

\begin{proof}[Proof of Proposition \ref{prop:unbalanced cross free of an hypergraph}]
 Let $\h,\f_{1},\ldots,\f_{h},$ and $\epsilon$ be as in the hypothesis
of the proposition and let $C$ be a large constant to be defined below.
Write $\h=\left(A_{1},\ldots,A_{h}\right)^{+}$ for some $d$-sets $A_{1},\ldots,A_{h}$, and denote $A=A_{1}\cup\cdots\cup A_{h}$.
By Claim~\ref{Cl:Simple}, there exist sets $B,B_h$ with $|B|=|A|$ and $|B_h|=|A_h|$ such that
\begin{equation}\label{Eq:Aux8-2}
\mu \left((\f_h)_{B}^{B_h} \right) \geq \mu(\f_h).
\end{equation}
Choose $d$-sets $B_{1},\ldots,B_{h-1}\in {{B}\choose{d}}$ in such a way that the ordered hypergraph $\left( B_{1},\ldots,B_{h}\right)$ is isomorphic to  $\left( A_{1},\ldots,A_{h}\right)$.

Consider the families $\left(\f_{1}\right)_{B}^{B_{1}},\ldots\left(\f_{h}\right)_{B}^{B_{h}}$. These families are cross free of a matching, since if $(D_1,\ldots,D_h)$ was a matching cross contained in $\left(\f_{1}\right)_{B}^{B_{1}},\ldots\left(\f_{h}\right)_{B}^{B_{h}}$, then the hypergraph $\left( B_{1}\cup D_{1},\ldots,B_{h}\cup D_{h}\right)$ would be a copy of $\h$ cross contained in $\f_1,\ldots,\f_h$. The following claim asserts that the measures $\mu\left(\f_{i}\right)_{B}^{B_{i}}$ are `large', which will allow us to apply Proposition~\ref{prop: unbalanced matching lemma} to these families.
\begin{claim}
\label{Claim: silly 2}For each $i\in\left[h-1\right]$, we have
\[
\mu\left(\f_{i}\right)_{B}^{B_{i}}\ge1-\epsilon^{\frac{1}{3}}.
\]
\end{claim}
\begin{proof}
The proof is a simple calculation. We have
\begin{align*}
\epsilon\ge1-\mu\left(\f_{i}\right) & =\Pr_{\mathbf{A}\sim {{\left[n\right]}\choose{k_i}}}\left[\mathbf{A}\notin\f_{i}\right]\ge\Pr\left[\mathbf{A}\cap B=B_{i}\right]\Pr_{\mathbf{A} \setminus B \sim {{\left[n\right]\backslash B}\choose{k_i-d}}}\left[\mathbf{A} \setminus B \notin (\f_{i})_B^{B_i}\right]\\
 & =\Omega_{h,d}\left(\left(\frac{k_i}{n}\right)^{d}\right)\left(1-\mu\left(\left(\f_{i}\right)_{B}^{B_{i}}\right)\right).
\end{align*}
 Rearranging, we obtain
\[
\mu\left(\left(\f_{i}\right)_{B}^{B_{i}}\right)=1-O_{d,h}\left(\epsilon\left(\frac{n}{k_i}\right)^{d}\right).
\]
Since $\epsilon\le\left(\frac{k_i}{n}\right)^{2d}$, this implies $\mu\left(\left(\f_{i}\right)_{B}^{B_{i}}\right) \ge1-\epsilon^{\frac{1}{3}}$,
provided that $C$ is sufficiently large.
\end{proof}
Since $\left(\f_{1}\right)_{B}^{B_{1}},\ldots\left(\f_{h}\right)_{B}^{B_{h}}$ are cross free of a matching and satisfy $\mu\left(\left(\f_{i}\right)_{B}^{B_{i}}\right) \ge1-\epsilon^{\frac{1}{3}}$ for all $i \in [h-1]$,
Proposition \ref{prop: unbalanced matching lemma} (applied with $3r$ instead of $r$) implies that
\[
\mu\left(\f_{h}\right)_{B}^{B_{h}} \leq O_{h,r}\left((\epsilon^{1/3})^{3r}\right) = O_{h,r}\left(\epsilon^{r}\right).
\]
Finally, plugging into~\eqref{Eq:Aux8-2} and taking $C$ sufficiently large, this yields
\[
\mu(\f_h) \leq \mu \left((\f_h)_{B}^{B_h} \right) \leq O_{h,r}\left(\epsilon^{r}\right) \leq C\epsilon^r,
\]
as asserted.
\end{proof}


\begin{proof}[Proof of Proposition \ref{prop:unbalanced cross free of hypergraph all negligible}]
Let $\f_{1},\ldots,\f_{h}$ be families as in the hypothesis of the proposition and let $C,s$ be sufficiently large constants (depending on $d,h,r$) to be defined below. Write $\h=\left(A_{1},\ldots,A_{h}\right)^{+}$ for some $d$-sets $A_{1},\ldots,A_{h}$.

For each $i\in\left[h-1\right]$, the family $\f_{i}$ is $\left(s,\frac{k_{i}}{n}\right)$-uncapturable, for otherwise we would have
\[
\mu\left(\f_{i}\right)\le (s+1) \frac{k_{i}}{n},
\]
contradicting the hypothesis $\mu\left(\f_{i}\right)\ge C \frac{k_{i}}{n}$ (provided $C>s+1$).

By Proposition~\ref{lem:associated junta lemma}, there exists a family $\j\subseteq\p\left(J\right)$, where $|J|=(2s)^r$, such that:
\begin{itemize}
\item $\mu\left(\f_{h}\backslash\j^{\uparrow}\right)=O_{s,r}\left(\frac{k_{h}}{n}\right)^{r}$;
\item All the sets in $\j$ are of size at most $r-1$;
\item For each set $B\in\j$, the family $\left(\f_{h}\right)_{B}^{B}$
is $\left(s,\left(\frac{k_{h}}{n}\right)^{r-\left|B\right|}\right)$-uncapturable.
\end{itemize}
(Note that the second condition does not appear in Proposition~\ref{lem:associated junta lemma}. However, it follows from the third condition, as no family is $(s,\beta)$-uncapturable for $\beta \geq 1$.) We claim that the family $\j$
is empty. This will imply that
\[
\mu\left(\f_{h}\right)=\mu\left(\f_{h}\backslash\j^{\uparrow}\right)\le C\left(\frac{k_{h}}{n}\right)^{r}
\]
for a sufficiently large $C$, completing the proof.

Suppose on the contrary that $\j \neq \emptyset$, and let $B \in \j$. The families
\[
\left(\f_{1}\right)_{B}^{\emptyset},\ldots,\left(\f_{h-1}\right)_{B}^{\emptyset},\left(\f_{h}\right)_{B}^{B}
\]
are cross free of the hypergraph $\left(A_{1},\ldots,A_{h}\right)^{+}$. In addition, each family $\left(\f_{i}\right)_{B}^{\emptyset}$ is $\left(s-\left|B\right|,\frac{k_{i}}{n}\right)$-uncapturable, and the family $\left(\f_{h}\right)_{B}^{B}$ is $\left(s,\left(\frac{k_{h}}{n}\right)^{r-|B|}\right)$-uncapturable.
Provided that $s$ is sufficiently large, this contradicts Proposition \ref{prop:uncap cross k large}. This completes the proof.
\end{proof}

\begin{proof}[Proof of Proposition \ref{prop: unbalnced cross free of hypergraph k small}]

Let $\h,\h',\f_{1},\ldots,\f_{h}$ be as in
the hypothesis of the proposition, and let $C$ be a sufficiently large constant to be determined below.
Recall that for any family $\f \subseteq {{[n]}\choose{k}}$, the $(k-C)$-shadow $\partial^{k-C}\left(\f\right)$ is the family of all $C$-sets that are contained in some element of $\f$. The proof of the proposition consists of four steps.
\begin{enumerate}
\item We first present a simple reduction to the case $\epsilon=\Omega\left(n^{-1/3}\right)$.
\item We then present a probabilistic argument that yields the upper bound $\mu\left(\partial^{k_h-d}\left(\f_{h}\right)\right)=O\left(\epsilon\right)$.
\item Lemma \ref{lem:Kruskal Katona going down} then tells us that $\mu\left(\partial^{t}\left(\f_{h}\right)\right)=O\left(\epsilon^{r}\right)$.
\item Finally, we apply Proposition~\ref{cor:t-shadows} to deduce that $\mu\left(\f_{h}\right)\le O\left(\frac{k_{h}^{2t}}{n^{t}}\epsilon^{r}\right)$.
\end{enumerate}

We begin with the reduction step.
\begin{claim}
Suppose that the proposition holds under the additional hypothesis $\epsilon\ge n^{-1/3}$. Then it holds for all $\epsilon>0$.
\end{claim}
\begin{proof}
If $\epsilon\le\left(\frac{\min_i (k_i)}{n}\right)^{2d}$, then Proposition \ref{prop:unbalanced cross free of an hypergraph} (applied with $r+\lceil t/d \rceil$ in place of $r$) yields
\[
\mu\left(\f_{h}\right)\le O\left(\epsilon^{r+t/d}\right)=O\left(\left(\frac{\min_i (k_{i})}{n}\right)^{2d \cdot (t/d)}\epsilon^{r}\right)=O\left(\frac{k_{h}^{2t}}{n^{t}}\epsilon^{r}\right),
\]
as asserted. To prove the assertion in the case $\epsilon\ge\left(\frac{\min_i (k_{i})}{n}\right)^{2d}$, we note that in this case, we have $\epsilon^{1/6d} \geq n^{-1/3}$, and so the assertion follows by applying the proposition with $\epsilon^{\frac{1}{6d}}$ in place of $\epsilon$ and with $6dr$ in place of $r$.
\end{proof}
We now establish Step~2, under the additional assumption $\epsilon \geq n^{-1/3}$. We use a probabilistic coupling argument.
\begin{claim}
$\mu\left(\partial^{k_h-d}\left(\f_{h}\right)\right)=O\left(\epsilon\right)$.
\end{claim}
\begin{proof}
Write $\h=\left(A_{1},\ldots,A_{h}\right)^{+}$ for $d$-sets $A_{1},\ldots,A_{h}$, and denote $A=A_{1}\cup\cdots\cup A_{h}$. Let $f\colon {{\left[n\right]}\choose{d}}\to{{[n]}\choose{k_h}}$ be a function that assigns to each set $A\in\partial^{k_h-d}\left(\f_{h}\right)$
a set $B\in\f_{h}$ that contains it (the values of $f$ on sets not in $\partial^{k_h-d}\left(\f_{h}\right)$ can be arbitrary). We define random sets $\mathbf{B_{1}},\ldots,\mathbf{B_{h}}$ in the following way.
\begin{itemize}
\item We choose a random set $\mathbf{A'}\sim {{\left[n\right]}\choose{\left|A\right|}}$ and a random bijection $\mathbf{\pi}\colon A\to \mathbf{A'}$ (i.e., $\mathbf{\pi}$ is chosen uniformly at random among the bijections $g:A \rightarrow \mathbf{A'}$), and set $\mathbf{A_{1}'}=\mathbf{\pi}\left(A_{1}\right),\ldots,\mathbf{A_{h}'}=\mathbf{\pi}\left(A_{h}\right)$.

Note that since the families $\f_{1},\ldots,\f_{h}$ are cross free of $\h$, then there are no pairwise disjoint sets $D_{1},\ldots,D_{h}$
such that $D_{i}\cup \mathbf{A_{i}'}\in\f_{i}$ for all $i \in [h]$.

\item We choose random sets $\mathbf{E_{1}}\sim {{\left[n\right]\backslash \mathbf{A'}}\choose{k_{1}-d}},\ldots,\mathbf{E_{h-1}}\in {{\left[n\right]\backslash \mathbf{A'}}\choose{k_{h-1}-d}}$.
\item We set
\[
\mathbf{B_{1}}:=\mathbf{A_{1}'}\cup \mathbf{E_{1}},\ldots,\mathbf{B_{h-1}}:=\mathbf{A_{h-1}'}\cup \mathbf{E_{h-1}},\mathbf{B_{h}}=f\left(\mathbf{A'_{h}}\right).
\]
\end{itemize}
It is clear that the families $(\mathbf{B_1},\ldots,\mathbf{B_h})$ satisfy
\[
\mathbf{B_{1}}\sim{{[n]}\choose{k_1}},\ldots,\mathbf{B_{h-1}}\sim {{\left[n\right]}\choose{k_{h-1}}},
\]
and
\[
\Pr\left[\mathbf{B_{h}}\in\f_{h}\right]\ge\Pr\left[\mathbf{A_{h}'}\in\partial^{k_{h}-d}\left(\f_{h}\right)\right]=\mu\left(\partial^{k_h-d}
\left(\f_{h}\right)\right).
\]
(This is the `coupling' element of our argument.) Since $\f_1,\ldots,\f_h$ are cross free of $\h$, a union bound implies that
\begin{align}\label{Eq:Aux8-3}
\begin{split}
\Pr &\left[\mbox{The ordered hypergraph }\left( \mathbf{B_{1}},\ldots,\mathbf{B_{h}}\right) \mbox{ is isomorphic to }\h\right] \le\\
&\le \sum_{i=1}^{h}\Pr\left[\mathbf{B_{i}}\notin\f_{i}\right]\le\left(h-1\right)\epsilon+\left(1-\mu\left(\partial^{k_{h}-d}\left(\f_{h}\right)\right)\right).
\end{split}
\end{align}
Note that the hypergraph $(\mathbf{B_{1}},\ldots,\mathbf{B_{h}})$ is a copy of $\h$ if and only if the sets $\mathbf{E_{1}},\ldots,\mathbf{E_{h-1}},\mathbf{B_{h}}\backslash \left(\mathbf{A_{h}'}\right)$ are pairwise disjoint. As the total number of elements in these sets (including possible multiplicities) is $k_1+\ldots+k_h-dh$ and all but one of the sets are chosen at random from the elements of $[n]\setminus \mathbf{A'}$, a union bound implies
\begin{align*}
\Pr&\left[\mbox{The sets } \mathbf{E_{1}},\ldots,\mathbf{E_{h-1}},\mathbf{B_{h}}\backslash \left(\mathbf{A_{h}'}\right) \mbox{ are pairwise disjoint}\right] \\
&\geq 1- {{k_1+\ldots+k_h-dh}\choose{2}} \cdot \frac{1}{n-|A|} \geq 1- O(n^{-1/3}) \geq 1-O(\epsilon),
\end{align*}
where the second inequality holds since $k_1,\ldots,k_h \leq n^{-1/3}/C$. Substituting into~\eqref{Eq:Aux8-3}, we obtain
\begin{align*}
1+(h-1)\epsilon-\mu\left(\partial^{k_{h}-d}\left(\f_{h}\right)\right) \ge\Pr\left[\mbox{The hypergraph }\left( \mathbf{B_{1}},\ldots,\mathbf{B_{h}}\right) \mbox{ is a copy of }\h\right]\ge 1-O(\epsilon).
\end{align*}
Rearranging yields
\[
\mu\left(\partial^{k_{h}-d}\left(\f_{h}\right)\right)=O\left(\epsilon\right),
\]
as asserted.
\end{proof}

Now we are ready to complete the proof of Proposition \ref{prop: unbalnced cross free of hypergraph k small}. Write $\b=\partial^{t}\left(\f_{h}\right)$. By Lemma~\ref{lem:Kruskal Katona going down}, we obtain
\[
\mu\left(\b\right)\le O\left(\mu\left(\partial^{k_{h}-t-d}\left(\b\right)\right)^{r}\right) \leq O \left(\mu\left(\partial^{k_{h}-d}\left(\f_{h}\right)\right) \right)^r \leq O\left(\epsilon^{r}\right),
\]
provided that $C$ is sufficiently large. Finally, since the family $\f_{h}$ is $\h'$-free, Proposition~\ref{cor:t-shadows} yields
\[
\mu\left(\f_h\right)\le O\left(\frac{k_h^{2t}}{n^{t}}\right)\mu\left(\b\right)=O\left(\frac{k_h^{2t}}{n^{t}}\epsilon^{r}\right).
\]
This completes the proof of the proposition.
\end{proof}

\section{Proof of the Main Theorems}
\label{sec:proof}

We are finally ready to present the proof of our main theorems. In Section~\ref{sec:sub:proof:junta} we prove the `junta approximation theorem' (i.e., Theorem~\ref{thm:Junta-approx-theorem}) which asserts that any family $\f \subseteq {{[n]}\choose{k}}$ that is free of a $d$-expanded hypergraph $\h$, can be approximated by an $\h$-free junta. We then compare our theorem with previously known results of Frankl and F\"{u}redi~\cite{frankl1987exact} and of Dinur and Friedgut~\cite{dinur2009intersecting}. In Section~\ref{sec:sub:proof:star} we prove Theorem~\ref{thm:exact solution for t-stars} which characterizes all forbidden hypergraphs $\h$, for which the extremal $\h$-free families are the $(t,t)$-stars. We conclude in Section~\ref{sec:sub:proof:porcupine} with proving Theorem~\ref{thm:OR} which gives sufficient conditions (on $\h$) for the $(t,1)$-stars to be the extremal $\h$-free families.

\subsection{Proof of the Junta approximation theorem}
\label{sec:sub:proof:junta}

We prove the following precise version of Theorem~\ref{thm:Junta-approx-theorem}.
\begin{thm}
\label{thm:main junta approximation theorem}For any constants $d,h\in\mathbb{N}$, there exist constants $C,j$ which depend only on $d,h$, such that the following holds. Set $\epsilon=\max \left(Ce^{-k/C},C\frac{k}{n}\right),$ and let $\h\subseteq{{[n]}\choose{k}}$
be a $d$-expanded hypergraph of size $h$. Let $\f\subseteq{{[n]}\choose{k}}$ be an $\h$-free family. Then there exists an $\h$-free $j$-junta
$\j$, such that
\begin{equation}
\mu\left(\f\backslash\j\right)\le\epsilon\mu\left(\j\right).
\label{eq:f essentially contained  in J}
\end{equation}
\end{thm}

\begin{proof}
Denote $|K(\h)|:=t-1$, let $C,C',s$ be sufficiently large constants to be determined below, and set $\epsilon'=\epsilon/C'$. Applying Proposition~\ref{lem:associated junta lemma} to the family $\f$, with the parameters $(t+1,s,\epsilon' \left(\frac{k}{n}\right)^t)$ in place of $(r,s,\epsilon)$, respectively, we obtain that there exists a set $J$ with $|J|=(2s)^t$, and a family $\j'\subseteq\p\left(J\right),$
such that:
\begin{enumerate}
\item For each set $B\in\j'$, the family $\f_{B}^{B}$ is $\left(s,\epsilon'\left(\frac{k}{n}\right)^{t-\left|B\right|}\right)$-uncapturable;

\item We have
\[
\mu\left(\f\backslash(\mathcal{\j'})^{\uparrow}\right)\le O_{s,t}\left(\epsilon'\right)\left(\frac{k}{n}\right)^{t}.
\]
\end{enumerate}
We set
\[
\j=\begin{cases}
\left\langle \j'\right\rangle :=\left\{ A\in{{[n]}\choose{k}}\,:\, A\cap J\in\j'\right\} , & \j'\ne\emptyset;\\
\s_{[t]}:= \left\{A\,:\,\left\{ 1,\ldots,t\right\} \subseteq A\right\} , & \j'=\emptyset.
\end{cases}
\]
We claim that $\j$ is the desired approximating $\h$-free junta.

This clearly holds in the case $\j'=\emptyset$. Indeed, in this case we have $\mu(\f) = \mu\left(\f\backslash(\mathcal{\j'})^{\uparrow}\right)\le O_{s,t}\left(\epsilon'\right)\left(\frac{k}{n}\right)^{t}$. Since $\mu(\s_{[t]})=\Theta \left(\left(\frac{k}{n}\right)^t \right)$, we have $\mu\left(\f\backslash\j\right)\le\epsilon\mu\left(\j\right)$, provided that $C'$ is sufficiently large (as function of $s,t$). As $|K(\h)|=t-1$, $\s_{[t]}$ is free of $\h$. Hence, $\j = \s_{[t]}$ satisfies the assertion of the theorem. (Of course, there is nothing specific about $\s_{[t]}$ here; any other $(t,t)$-star would be an equally good `approximation').

Suppose now $\j' \neq \emptyset$, and so $\j=\langle \j' \rangle$. We first show that $\j'$ is $t$-uniform. Let $B \in \j'$. By Condition~(1) above, the family $\f_{B}^{B}$ is $\left(s,\epsilon'\left(\frac{k}{n}\right)^{t-\left|B\right|}\right)$-uncapturable. Provided that $C>C'$, we have $\epsilon' \cdot \frac{n}{k}>1$, and thus, no family can be $(s,\epsilon' \frac{n}{k})$-uncapturable. Hence, we must have $|B| \leq t$. On the other hand, if $|B| \leq t-1$, then $\f_B^B$ is free of the hypergraph $\h'$ obtained from $\h$ by removing $|B|$ elements out of its kernel. However, this contradicts Proposition~\ref{prop:uncap single k small} which says that an $\left(s,\epsilon'\left(\frac{k}{n}\right)^{t-\left|B\right|}\right)$-uncapturable family contains a copy of any fixed-size hypergraph with kernel of size $t-1-|B|$, provided $s,C$ are sufficiently large. (Note that in order to apply Proposition~\ref{prop:uncap single k small}, we need the additional assumption $C_0 \leq k \leq n/C_0$, for the constant $C_0$ mentioned in the proposition, which in our case depends on $d,h$. We take $C$ sufficiently large, as function of $d,h$, so that this assumption is satisfied.)  Hence, the only remaining possibility is $|B|=t$, and so $\j'$ is $t$-uniform.

By Lemma~\ref{lem:measures of juntas}, we have $\mu(\j)=\mu(\langle \j' \rangle)=\Theta \left(\left(\frac{k}{n} \right)^t \right)$. As by Condition~(2) above, \[
\mu\left(\f\backslash(\mathcal{\j'})^{\uparrow}\right)\le O_{s,t}\left(\epsilon'\right)\left(\frac{k}{n}\right)^{t},
\]
it follows that~\eqref{eq:f essentially contained  in J} holds, provided that $C'$ is sufficiently large.

We now complete the proof by showing that $\j$ is $\h$-free.
Suppose on the contrary that $\j$ contains a copy $(A_1,\ldots,A_h)$ of $\h$. For each $i \in [h]$, denote $B_i=A_i \cap J$ and $E_i = A_i \setminus B_i$. By the definition of $\j$, we have $B_i \in \j'$, and in particular, $|B_i|=t$. Thus, by Condition~(1) above, each family $\f_{B_i}^{B_i}$ is $\left(s,\epsilon'\right)$-uncapturable. Therefore, the families
\[
\a_{1}:=\f_{B_{1}\cup\cdots\cup B_{h}}^{B_{1}},\ldots,\a_{h}:=\f_{B_{1}\cup\cdots\cup B_{h}}^{B_{h}}
\]
are $\left(s-(h-1)t,\epsilon'\right)$-uncapturable, and in particular, satisfy $\mu(\a_i) > \epsilon'$ (provided that $s$ is sufficiently large). Therefore, by Proposition~\ref{Prop:turan for cross}, these families cross contain a copy of any $d$-expanded ordered hypergraph of size $h$ (provided that $C$ is large enough).

However, since $\f$ is free of $\h$, the families $\a_1,\ldots,\a_h$ are cross free of the ordered hypergraph $\left(E_{1},\ldots,E_{h}\right)$, a contradiction. This completes the proof.
\end{proof}

\begin{rem}\label{Rem:Uniform}
Note that in Theorem~\ref{thm:main junta approximation theorem}, one can further require the approximating junta to be $t$-uniform. Indeed, if $\mu(\f) \leq \epsilon (\frac{k}{n})^t$, then the assertion holds trivially for $\j$ being any $(t,t)$-star. Otherwise, $\j$ must contain an element of size $\leq t$, provided $C$ is sufficiently large. In such a case, the assertion of the theorem remains true if we remove from $\j$ all elements of size $>t$, as they contribute to the measure of $\langle \j \rangle$ at most $O((\frac{k}{n})^{t+1})$, which is negligible for a sufficiently large $C$. On the other hand, $\j$ cannot contain any element of size $<t$, since otherwise, the approximating junta would contain a copy of any fixed hypergraph with kernel of size $\leq t-1$, and in particular, would contain a copy of $\h^+$, contradicting the assertion of the theorem.
\end{rem}

\medskip As mentioned in the introduction, Theorem~\ref{thm:main junta approximation theorem} can be viewed as a generalization of the following fundamental theorem of Frankl and F\"{u}redi~\cite[Theorem~5.3]{frankl1987exact}.
\begin{thm}[Frankl-F\"{u}redi, 1987]
\label{thm:Frankl-Furedi}
For any constants $t,s,d\in\mathbb{N}$, $\epsilon>0$, and any fixed $d$-expanded hypergraph $\h$ with kernel of size $t-1$ and center of size $s$, the following holds.

For any $k \geq s+2t$ and any sufficiently large $n$ (as function of $k$ and $\h$), there exists an $\h$-free $t$-expanded $O_{\h}\left(1\right)$-junta $\j\subseteq{{[n]}\choose{k}}$ such that any $\h$-free family $\f\subseteq{{[n]}\choose{k}}$ satisfies $\left|\f\right|\le\left|\j\right|\left(1+\epsilon\right)$.
\end{thm}
In words, the Frankl-F\"{u}redi theorem asserts that asymptotically, the largest $\h$-free families are juntas. Theorem~\ref{thm:main junta approximation theorem} extends Theorem \ref{thm:Frankl-Furedi} in two directions. Firstly, we remove the hypothesis that $k$ is a constant and instead, we allow $k$ to be up to linear in $n$. Secondly, we strengthen the numerical statement that the $t$-expanded $O\left(1\right)$-juntas are the largest extremal families into the stronger structural statement that any $\h$-free family is essentially contained in a $t$-expanded $O\left(1\right)$-junta.

\medskip Another related result is a theorem of Dinur and Friedgut~\cite{dinur2009intersecting}, who established (a stronger version of) Theorem~\ref{thm:main junta approximation theorem} in the special case where $\h$ consists of two disjoint edges, and thus, an $\h$-free family is simply an intersecting family.
\begin{thm}[Dinur and Friedgut, 2009]
\label{thm:Dinur Friedgut}
For any $r>0$, there exist constants $j\left(r\right),C\left(r\right)$ such that the following holds.
	
Let $n,k$ be such that $k<n/C$, and let $\f\subseteq{{[n]}\choose{k}}$ be an intersecting family. Then there exists an intersecting $j$-junta
$\j$, such that $\left|\f\backslash\j\right|\le O_{r}\left(\left(\frac{k}{n}\right)^{r}\right)\left|\j\right|$.
\end{thm}
In the special case $r=1$, the assertion of Theorem~\ref{thm:Dinur Friedgut} was proved already in 1987 by Frankl~\cite{frankl1987erdos}, who also showed that in that case, the junta $\j$ may be taken to be a $(1,1)$-star.

In the case of intersecting families to which it applies, Theorem \ref{thm:Dinur Friedgut} is stronger than our Theorem~\ref{thm:main junta approximation theorem} in two senses. Firstly, in Theorem~\ref{thm:Dinur Friedgut} the hypothesis that $k$ is larger than some constant is removed, and secondly,
Theorem~\ref{thm:Dinur Friedgut} allows to deduce that $\left|\f\backslash\j\right|\le O_{r}\left(\frac{k}{n}\right)^{r}\left|\j\right|$,
while our Theorem \ref{thm:main junta approximation theorem} only gives us the weaker approximation $\left|\f\backslash\j\right|\le\epsilon\left|\j\right|$ for $\epsilon=\max \left(Ce^{\left(-k/C\right)},C\frac{k}{n}\right)$. However, in the special case where the edges of $\h$ are pairwise disjoint, a more general result is given in our Theorem~\ref{thm:matching junta approximation} which shows that the assertion of the Dinur-Friedgut theorem holds for forbidden matchings of an arbitrary fixed size, and not only in the `single-family' setting, but also in the `cross' setting.

\subsection{Forbidden hypergraphs for which the extremal families are the $(t,t)$-stars}
\label{sec:sub:proof:star}

In this section we prove Theorem~\ref{thm:exact solution for t-stars} which characterizes all forbidden expanded hypergraphs $\h$ for which the extremal $\h$-free families are the $(t,t)$-stars, along with a stability version (Theorem~\ref{thm:Stab-for-Stars}).

Let $\h$ be a $d$-expanded hypergraph of size $h$. In order for the $(t,t)$-stars to be the extremal $\h$-free families, it is necessary that the $(t,t)$-star is $\h$-free, and that no hypergraph that properly contains a $(t,t)$-star is $\h$-free. We first show that these two necessary conditions are equivalent to the following intrinsic property of $\h$:

\medskip

\begin{description}
\item [{Condition}] ({*}) $\h$ is a $d$-expanded hypergraph of size $h$, $|K(\h)|=t-1$, and there exists a set of size $2t-1$ that is contained in $h-1$ of the edges of $\h$.
\end{description}

\medskip

We then show that this trivially necessary condition is also sufficient; namely, that for any forbidden hypergraph $\h$ that satisfies (*), the extremal $\h$-free families are the $(t,t)$-stars. The proof consists of three steps:
\begin{enumerate}
\item We show that if $\h$ satisfies condition (*) and $\f$ is an $\h$-free family, then the approximating junta of $\f$ given by Theorem~\ref{thm:main junta approximation theorem} is a $(t,t)$-star. Hence, any `large' $\h$-free family is a small perturbation of some $(t,t)$-star $\s_T$.

\item We prove a bootstrapping lemma which asserts that if $\h$ satisfies condition (*), and $\f$ is an $\h$-free family that satisfies $\mu(\f_{T}^{T}) \geq 1-\epsilon$ for some set $T$ of size $t$, then $\mu(\f \setminus \s_T)$ is much smaller than $\epsilon$.

\item We combine Steps 1 and 2 to deduce that if $\h$ satisfies condition (*) then an $\h$-free family $\f$ cannot be larger than the $(t,t)$-star, hence proving Theorem~\ref{thm:exact solution for t-stars}. Furthermore, we deduce a `stability version' which asserts that if $\f$ is `sufficiently large' then it is essentially contained in a $(t,t)$-star, hence proving Theorem~\ref{thm:Stab-for-Stars}.
\end{enumerate}

We begin with proving the equivalence between the conditions on $\h$.
\begin{lem}
\label{lem:inclusion maximal-> intrinsic}
For any constants $d,h,t$, there exists a constant $C,$ such that the following holds. For any $C< k< n/C$, a $d$-expanded hypergraph $\h \subseteq {{[n]}\choose{k}}$ of size $h$ satisfies Condition~({*}) if and only if the following two conditions hold:
\begin{enumerate}
\item The $\left(t,t\right)$-star is $\h$-free;
\item No family that properly contains a $\left(t,t\right)$-star is $\h$-free.
\end{enumerate}
\end{lem}
\begin{proof}
Suppose that $\h$ satisfies ({*}). Then (1) holds, since the intersection of all the edges of $\h$ is of size $t-1$, while the intersection of any set of elements of a $(t,t)$-star is of size $\geq t$. To see that~(2) holds, let $\f\subseteq{{[n]}\choose{k}}$ be a family that properly contains a $\left(t,t\right)$-star $\s_T$, and assume w.l.o.g. that $T=[t]$. Let $E\in\f\backslash\s_{[t]}$. Since $\f_{[t]}^{[t]}$ contains the `entire universe' ${{[n-t]}\choose{k-t}}$ and $\f_{[t]}^{[t] \cap E}$ is non-empty, it is clear that the $h$ hypergraphs
\[
\f_{[t]}^{[t]},\ldots,\f_{[t]}^{[t]},\f_{[t]}^{[t] \cap E}
\]
cross contain \emph{any} ordered hypergraph of size $h$ (with edges of appropriate sizes). We shall use this right away.

\mn The hypergraph $\h$ can be written in the form
\[
\h=\left\{ K\sqcup K'\sqcup E_{1},\ldots,K\sqcup K'\sqcup E_{h-1},K\sqcup E_{h}\right\} ,
\]
where $K=K(\h)$, $K'$ is of size $t$, and $E_1,\ldots,E_h$ are disjoint from $K\sqcup K'$. Denoting $i:=|E \cap [t]|$, letting $K_{i}\subseteq K$ be a set of size $i$, and letting $K'_{t-i} \subseteq K'$ be a set of size $t-i$, we can write $\h$ in the form
\[
\h=\left\{ K_{i}\sqcup K'_{t-i}\sqcup F_{1},\ldots,K_{i}\sqcup K'_{t-i}\sqcup F_{h-1},K_{i}\sqcup F_{h}\right\},
\]
where $F_1,\ldots,F_h$ are disjoint from $K_{i}\sqcup K'_{t-i}$. By the above argument, the families
$\f_{\left[t\right]}^{\left[t\right]},\ldots,\f_{\left[t\right]}^{\left[t\right]},\f_{\left[t\right]}^{\left[t\right]\cap E}$
cross contain a copy $(B_1,\ldots,B_h)$ of the ordered hypergraph $\left(F_{1},\ldots,F_{h}\right)$. The sets $(B_1 \cup [t],\ldots, B_{h-1} \cup [t], B_h \cup ([t] \cap E))$ constitute a copy of $\h$ in $\f$. This shows that~(2) holds.

\medskip

In the converse direction, let $\h$ be a hypergraph that satisfies~(1) and (2). Since the $\left(t,t\right)$-star is free of $\h$, we must have $|K(\h)|\leq t-1$ (as the $(t,t)$-star contains a copy of \emph{any} hypergraph with kernel of size $\geq t$). We want to show that $|K(\h)|=t-1$ and that there exists a set of size $2t-1$ that is contained in all edges of $\h$ except for one.

Let $E_{0}$ be a set that is disjoint from $\left[t\right]$, and let $E_{t-1}$ be a set whose intersection with $[t]$ is of size $t-1$. Denote $\f_0=\s_{[t]} \cup E_0$ and $\f_{t-1}=\s_{[t]} \cup E_{t-1}$. By (2), $\f_{t-1}$ contains a copy of $\h$. Hence, we cannot have $|K(\h)|\leq t-2$ (as the intersection of any set of elements of $\f_{t-1}$ is of size $\geq t-1$), which means that $|K(\h)|=t-1$.

Let $H=\{K\sqcup A_{1},\ldots,K\sqcup A_{h-1},K\sqcup A_{h}\}$ be a copy of $\h$ in $\f_{0}$, where $K$ corresponds to $K(\h)$. As $\s_{[t]}$ is free of $\h$, one of the edges of $H$ must be $E_0$, and hence, $K \cap [t] = \emptyset$. Furthermore, as $E_0$ is the only element of $\f_0$ that is not contained in $\s_{[t]}$, all the remaining edges of $H$ are contained in $\s_{[t]}$, and thus, exactly $h-1$ of the sets $A_{1},\ldots,A_{h}$ contain $[t]$. Therefore, the set $K\cup [t]$ in the `copy' $H$ corresponds to a $(2t-1)$-element set contained in $h-1$ of the edges of $\h$. This completes the proof.
\end{proof}

We now show that if the forbidden hypergraph $\h$ satisfies (*) and if $\f$ is $\h$-free then the junta which approximates $\f$ according to Theorem~\ref{thm:main junta approximation theorem} can be taken to be a $(t,t)$-star.
\begin{lem}
\label{lem:rough stability t stars} For any constants $d,h$, there exists a constant $C$ such that the following holds. Let $C<k<n/C$,
and set $\epsilon=\max \left(C\frac{k}{n},e^{\left(-k/C\right)} \right)$.

For any hypergraph $\h$ that satisfies ({*}) and any $\h$-free family $\f\subseteq{{[n]}\choose{k}}$, there exists a $(t,t)$-star $\s_T$
such that $\mu\left(\f\backslash\s_T\right) \leq \epsilon\left(\frac{k}{n}\right)^{t}$.
\end{lem}
\begin{proof}
By Theorem~\ref{thm:main junta approximation theorem} and Remark~\ref{Rem:Uniform}, there exists an $O_{d,h}\left(1\right)$-sized set $J$ and a $t$-uniform set $\mathcal{\j}\subseteq {{J}\choose{t}}$, such that the junta $\left\langle \j\right\rangle $ is $\h$-free and
\[
\mu\left(\f\backslash\left\langle \j\right\rangle \right)\le\epsilon \mu(\langle \j \rangle).
\]
Our proof will be accomplished once we show that $\j$ contains at most one element. Suppose on the contrary that there exist sets $B_{1}\ne B_{2}\in\j$,
and write $i:= |B_1 \cap B_2|$. As in the proof of Lemma~\ref{lem:inclusion maximal-> intrinsic}, we can write
\[
\h=\left\{ K_{i}\sqcup K'_{t-i}\sqcup F_{1},\ldots,K_{i}\sqcup K'_{t-i}\sqcup F_{h-1},K_{i}\sqcup F_{h}\right\},
\]
where $K_{i}\subseteq K(\h)$ is a set of size $i$, $K'_{t-i} \subseteq K'$ is a set of size $t-i$, and $F_1,\ldots,F_h$ are disjoint from $K_{i}\sqcup K'_{t-i}$. Since $\h$ is $d$-expanded and $k$ is larger than $d+(t-i)$, we can write $F_h := F'_h \sqcup E_h$, where $|E_h|=t-i$ and $E_h$ is disjoint from all other edges of $\h$.

As $B_1,B_2 \in \j$, both hypergraphs $\left\langle \j\right\rangle _{J}^{B_{1}}$ and $\left\langle \j\right\rangle _{J}^{B_{2}}$ consist of the `entire universe' ${{\left[n\right]\backslash J}\choose{k-t}}$. Hence, the $h$ families
\[
\left\langle \j\right\rangle _{J}^{B_{1}},\ldots,\left\langle \j\right\rangle _{J}^{B_{1}},\left\langle \j\right\rangle _{J}^{B_{2}}
\]
cross contain \emph{any} ordered hypergraph of size $h$ with edges of appropriate sizes. In particular, they cross contain a copy $(A_1,\ldots,A_h)$ of the hypergraph $(F_1,\ldots,F_{h-1},F'_h)$. Thus, the hypergraph $(A_1 \cup B_1,\ldots,A_{h-1} \cup B_1,A_h \cup B_2)$ is a copy of $\h$ in $\langle \j \rangle$, contradicting the assumption that $\langle \j \rangle$ is $\h$-free. This completes the proof.
\end{proof}

We now present the bootstrapping step which asserts that the theorem `holds locally', i.e., if $\h$ satisfies condition (*), and $\f$ is an $\h$-free family which is a small perturbation of a $(t,t)$-star $\s_T$, then $\mu(\f \setminus \s_T)$ is much smaller than $1-\mu(\f_{T}^{T})$. The proof uses the results of Section~\ref{sec:bootstrapping}, and so different arguments (and even slightly different statements) are needed for different ranges of $k$.
\begin{lem}
\label{lem:small-k  Bootstrapping} For any constants $d,h,r$, there exists a constant $C=C(d,h,r)$ such that the following holds. Let $C<k<n/C$, and denote
\[
\epsilon_0(k):=\begin{cases}
1/C , & k>C\log n;\\
1/(kC), & k \leq C \log n.
\end{cases}
\]
Let $\h$ be a hypergraph that satisfies ({*}) and let $\f \subseteq {{[n]}\choose{k}}$ be an $\h$-free family that satisfies $\mu\left(\f_{\left[t\right]}^{\left[t\right]}\right)\ge1-\epsilon$, for some $\epsilon \leq \epsilon_0(k)$.
Then
\[
\mu\left(\f\backslash\s_{[t]}\right)\le\mu\left(\s_{[t]}\right)\epsilon^{r}.
\]
 \end{lem}
\begin{proof}
The main observation we use is that as shown in the proof of Lemma~\ref{lem:inclusion maximal-> intrinsic}, if $\f$ satisfies the hypothesis then for any $A \subsetneq [t]$, the $h$ families
\[
\f_{[t]}^{[t]},\ldots,\f_{[t]}^{[t]},\f_{[t]}^{A}
\]
are cross free of some $d$-expanded ordered hypergraph with $h$ edges. (That hypergraph was denoted by $(F_1,\ldots,F_h)$ in the proof of Lemma~\ref{lem:inclusion maximal-> intrinsic}.) Since the family $\f_{[t]}^{[t]}$ is `very large', this allows us to apply the results of Section~\ref{sec:bootstrapping} to deduce that $\f_{[t]}^{A}$ is `very small'. The assertion will then follow, using the relation
\[
\mu\left(\f\backslash\s_{[t]}\right)=\sum_{A\subsetneq\left[t\right]}\Pr_{\mathbf{B}\sim{{[n]}\choose{k}}}\left[\mathbf{B}\cap\left[t\right]=A\right]
\mu\left(\f_{\left[t\right]}^{A}\right).
\]
We consider three cases.

\mn \textbf{Case 1: $k > C\log n$ and $\epsilon\le\left(\frac{k}{n}\right)^{2d}$.} For any $A\subsetneq [t]$, by Proposition \ref{prop:unbalanced cross free of an hypergraph} (applied with $3r$ in place of $r$), there exists $C'=C'(d,h,r)$ such that we have
\[
\mu\left(\f_{\left[t\right]}^{A}\right)\le C'\epsilon^{3r}\le2^{-t}\mu\left(\s_{[t]}\right)\epsilon^{r},
\]
where the inequalities hold provided that $C$ is sufficiently large. Thus,
\begin{align*}
\mu\left(\f\backslash\s_{[t]}\right) & =\sum_{A\subsetneq\left[t\right]}\Pr_{\mathbf{B}\sim{{[n]}\choose{k}}}\left[\mathbf{B}\cap\left[t\right]=A\right]
\mu\left(\f_{\left[t\right]}^{A}\right)\le\sum_{A\subsetneq\left[t\right]}2^{-t}\mu\left(\s_{[t]}\right)\epsilon^{r}\le\mu\left(\s_{[t]}\right)\epsilon^{r},
\end{align*}
as asserted.

\mn \textbf{Case 2: $k > C\log n$ and $\epsilon\ge\left(\frac{k}{n}\right)^{2d}$.} For any $A\subsetneq [t]$, by Proposition~\ref{prop:unbalanced cross free of hypergraph all negligible} (applied with $2dr+t+1$ in place of $r$), we have
\[
\mu\left(\f_{\left[t\right]}^{A}\right)=O_{d,h,r}\left(\frac{k}{n}\right)^{2dr+t+1}\le 2^{-t}\mu\left(\s_{[t]}\right)\epsilon^{r},
\]
provided that $C$ is sufficiently large. The assertion follows like in Case~1.

\mn \textbf{Case 3: $k \leq C\log n$.} Let $A\subsetneq\left[t\right]$. We observe that the family $\f_{\left[t\right]}^{A}$ is free of the hypergraph $\h'$ obtained from $\h$ by removing $\left|A\right|$ vertices out of its kernel (which is a $(d-|A|)$-expanded hypergraph with kernel of size $t-1-\left|A\right|$). Hence,
Proposition~\ref{prop: unbalnced cross free of hypergraph k small} (applied to the families $\left(\f_{[t]}^{[t]},\ldots,\f_{[t]}^{[t]},\f_{[t]}^{A} \right)$ with $r+2t$ in place of $r$) implies
\[
\mu\left(\f_{\left[t\right]}^{A}\right)\le O\left(\frac{k^{2\left(t-\left|A\right|\right)}}{n^{t-\left|A\right|}}\right)\epsilon^{r+2t},
\]
provided that $C$ is sufficiently large. Using again the appropriate choice of $C$, along with the assumption $\epsilon \leq \epsilon_0(k)=1/(Ck)$, this implies
\begin{align*}
\mu\left(\f\backslash\s_{[t]}\right) & =\sum_{A\subsetneq\left[t\right]}\Pr_{\mathbf{B}\sim{{[n]}\choose{k}}}\left[\mathbf{B}\cap [t]=A\right]\mu\left(\f_{[t]}^{A}\right)=\sum_{A\subsetneq [t]} O\left(\frac{k}{n}\right)^{\left|A\right|} \left(\frac{k^{2\left(t-\left|A\right|\right)}}{n^{t-\left|A\right|}}\right) \epsilon^{r+2t}\\
 & \le 2^t \cdot O\left(\frac{k^{2t}}{n^{t}}\right)\left(\frac{1}{Ck}\right)^{2t}\epsilon^{r}\le\mu\left(\s_{[t]}\right)\epsilon^{r}.
\end{align*}
This completes the proof.
\end{proof}

Now we are ready to prove Theorems~\ref{thm:exact solution for t-stars} and~\ref{thm:Stab-for-Stars}.
\begin{thm}
\label{thm:Solution for t-stars}For any constants $d,h,r$, there exists a constant $C$ such that the following holds.

Let $\h$ be a hypergraph that satisfies (*). For each $C<k<n/C$, any $\h$-free family $\f\subseteq{{[n]}\choose{k}}$ has
at most ${{n-t}\choose{k-t}}$ elements, and equality holds if and only if $\f$ is a $(t,t)$-star.

Moreover, let $\epsilon\in\left(0,1/C\right)$ and suppose that $\left|\f\right|\ge \left(1-\epsilon\right) {{n-t}\choose{k-t}}$.
Then there exists a $(t,t)$-star $\s$ such that $\mu\left(\f\backslash\s\right)\le\mu\left(\s\right)\epsilon^{r}$.
\end{thm}
It is clear that Theorem~\ref{thm:Solution for t-stars} (together with Lemmas~\ref{lem:inclusion maximal-> intrinsic} and~\ref{lem:measures of juntas}) implies Theorems~\ref{thm:exact solution for t-stars} and~\ref{thm:Stab-for-Stars}.

\begin{proof}[Proof of Theorem~\ref{thm:Solution for t-stars}]
Let $\f$ be $\h$-free for a hypergraph $\h$ that satisfies (*), and let $\epsilon$ be such that $\left|\f\right|\ge \left(1-\epsilon\right){{n-t}\choose{k-t}}$. (Note that `theoretically', $|\f|$ may be larger than ${{n-t}\choose{k-t}}$.) We first show that we may assume that $\epsilon$ is small. By Lemma~\ref{lem:rough stability t stars}, there exists a $(t,t)$-star $\s_T$ and a constant $C_{1}$, such that
\[
\mu\left(\f\backslash\s_{T}\right)\le \max\left(e^{-k/C_{1}},C_1\frac{k}{n}\right)\mu\left(\s_{T}\right).
\]
This proves the `Moreover...' assertion if $\mu\left(\s_{T}\right)\epsilon^{r}\ge \max\left(e^{-k/C_{1}},C_1\frac{k}{n}\right)\mu\left(\s_{T}\right)$. Hence,
it remains to prove the theorem in the case where
\[
\epsilon\le\left(e^{-k/C_{1}}+\frac{C_{1}k}{n}\right)^{1/r}.
\]
Suppose without loss of generality that $T=[t]$ and let $\epsilon_1 \geq 0$ be such that  $\mu\left(\f_{\left[t\right]}^{\left[t\right]}\right)=1-\epsilon_{1}$. We now show that we may assume that $\epsilon_1$ satisfies the hypothesis of Lemma
\ref{lem:small-k  Bootstrapping}.

If $\epsilon_1=0$, then $\f_{[t]}^{[t]}$ is the `entire universe' ${{[n] \setminus [t]}\choose{k-t}}$. Thus, for any $A \subsetneq [t]$, either the family $\f_{[t]}^{A}$ is empty, or the $h$ families $\f_{[t]}^{[t]},\ldots,\f_{[t]}^{[t]},\f_{[t]}^{A}$ cross contain a copy of \emph{any} fixed ordered  hypergraph of size $h$ (with edges of appropriate sizes). The latter cannot hold since, as shown in the proof of Lemma~\ref{lem:inclusion maximal-> intrinsic}, these families are cross free of some $d$-expanded hypergraph with $h$ edges (that was denoted by $(F_1,\ldots,F_h)$ in the proof of Lemma~\ref{lem:inclusion maximal-> intrinsic}.) Therefore, $\f_{[t]}^{A} = \emptyset$ for any $A \subsetneq [t]$, which means that $\f=\s_{[t]}$, as asserted.

On the other hand, we have
\begin{align*}
1-\epsilon_{1}=\mu\left(\f_{\left[t\right]}^{\left[t\right]}\right) & =\frac{\mu\left(\f\right)-\mu\left(\f\backslash\s_{[t]}\right)}{\mu\left(\s_{[t]}\right)}\ge1-\epsilon- \max\left(e^{-k/C_{1}},C_1\frac{k}{n}\right),
\end{align*}
and hence,
\[
\epsilon_{1}\le \max \left(e^{-k/C_{1}},C_1\frac{k}{n} \right)+\epsilon\le 2 \max \left(e^{-k/C_{1}},C_1\frac{k}{n}\right)^{1/r},
\]
where the last inequality uses the assumption on $\epsilon$. Provided that $C$ is sufficiently large, this implies that $\epsilon_1$ satisfies the hypothesis of Lemma~\ref{lem:small-k  Bootstrapping}. Applying Lemma~\ref{lem:small-k  Bootstrapping} (with $2r$ in place of $r$), we obtain
\[
\mu\left(\f\backslash\s_{[t]}\right)\le\epsilon_{1}^{2r}\mu\left(\s_{[t]}\right).
\]
Thus,
\begin{align*}
\mu\left(\s_{[t]}\right)\left(1-\epsilon\right) & \le\mu\left(\f\right)=\mu\left(\f_{\left[t\right]}^{\left[t\right]}\right)\mu\left(\s_{[t]}\right)+\mu\left(\f\backslash\s_{[t]}\right)\le
\left(1-\epsilon_{1}+\epsilon_{1}^{2r}\right)\mu\left(\s_{[t]}\right).
\end{align*}
Rearranging, we obtain $\epsilon>0$ (and in particular, $|\f| < {{n-t}\choose{k-t}}$ which proves the `uniqueness' in the first part of the theorem), and $\epsilon_{1}=O\left(\epsilon\right)$. Therefore,
\[
\mu\left(\f\backslash\s_{[t]}\right)\le\epsilon_{1}^{2r}\mu\left(\s_{[t]}\right)=O\left(\epsilon^{2r}\right)\mu\left(\s_{[t]}\right)\le
\epsilon^{r}\mu\left(\s_{[t]}\right).
\]
This completes the proof of the theorem.
\end{proof}

\subsection{Forbidden hypergraphs for which the extremal families are the $(t,1)$-stars}
\label{sec:sub:proof:porcupine}

In this section we present the proof of Theorem~\ref{thm:OR} which gives sufficient conditions (on $\h$) for the $(t,1)$-stars to be the extremal $\h$-free families.

As in the case of $(t,t)$-stars considered in Section~\ref{sec:sub:proof:star}, obviously necessary conditions for the $(t,1)$-stars to be the extremal $\h$-free families are that the $(t,1)$-star is free of $\h$ and that no hypergraph which properly contains a $(t,1)$-star is $\h$-free. However, it turns our that these conditions are not sufficient, as demonstrated by the following example.
\begin{example}
\label{or s extremal} Let $C$ be a sufficiently large constant, let $C<k<n/C$, and let $\h$ be the $k$-expansion of the
hypergraph $H=\left\{ \left\{ 1,2\right\} , \left\{ 1,4\right\}, \left\{ 1,5\right\} ,\left\{ 2,6\right\} , \left\{ 2,7\right\}, \left\{ 3\right\} \right\}$.
Then it is easy to see that the $\left(2,1\right)$ star $\s'_{[2]} = \left\{ A\in{{[n]}\choose{k}}\,:\,A\cap\left\{ 1,2\right\}
\ne\emptyset\right\} $ is $\h$-free and is maximal under inclusion among the $\h$-free families.

However, we claim that the family $\f=\left\{ A\in{{[n]}\choose{k}}\,:\,\left|A\cap\left\{ a,b,c\right\} \right|=1\right\}$ (for arbitrary distinct $a,b,c \in [n]$), which is larger than $\s'_{[2]}$ by Lemma~\ref{lem:measures of juntas}, is $\h$-free. Indeed, suppose on the contrary that $H'$ is a copy of $\h$ in $\f$. Then, without loss of generality, the edges of $H'$ that correspond to the expansions of $\{1,4\},\{1,5\},\{2,6\},\{2,7\},\{3\}$ are of the form $\{a\} \cup A_1, \{a\} \cup A_2, \{b\} \cup A_3,\{b\} \cup A_4,\{c\} \cup A_5$, respectively, where $A_1,\ldots,A_5$ are pairwise disjoint. As the edge of $H'$ that corresponds to the expansion of $\{1,2\}$ must intersect the first four of these edges, it must contain both $a$ and $b$, a contradiction.
\end{example}
To avoid Example~\ref{or s extremal} and its relatives, we bound our discussion to hypergraphs $\h$ that satisfy the following stronger conditions:
\begin{itemize}
\item The $\left(t,1\right)$-star $\s'_{[t]} = \{A: A \cap [t] \neq \emptyset\}$ is free of $\h$;
\item No family that contains the family $\left\{ A\,:\,\left|A\cap\left[t\right]\right|=1\right\} $ and is not contained in $\s'_{[t]}$, is $\h$-free.
\end{itemize}
We show below that these conditions are equivalent to the following intrinsic property of $\h$.
\begin{description}
\item [{Condition}] ({*}{*}) $\h$ is a $d$-expanded hypergraph with $h$ edges, there exists a set $T$ of size $t$ such that $\left|\mathrm{Span}_{\h}^{1}\left(T\right)\right|=h-1$, and there is no set $T'$ of size $t$ such that $\left|\mathrm{Span}_{\h}\left(T'\right)\right|=h.$
\end{description}
We then show that this condition is sufficient; namely, that for any forbidden hypergraph $\h$ that satisfies (**), the extremal $\h$-free families are the $(t,1)$-stars. The proof is very similar to the proof in the case of $(t,t)$-stars presented in Section~\ref{sec:sub:proof:star}. Hence, we only state the corresponding lemmas and the required changes with respect to the proof in the `$(t,t)$-stars' case.

\medskip

We begin with proving the equivalence between the conditions on $\h$.
\begin{lem}
\label{lem:inclusion maximal-> intrinsic-1}For any constants $d,h,t$, there exists a constant $C$ such that the following holds.

For any $C< k <n/C$, a $d$-expanded hypergraph $\h \subseteq {{[n]}\choose{k}}$ of size $h$ satisfies ({*}{*})
if and only if the following two conditions hold:
\begin{enumerate}
\item The $\left(t,1\right)$-star $\s'_{[t]}:=\left\{ A\,:\, A\cap\left[t\right] \neq \emptyset \right\}$ is $\h$-free.
\item No family that contains the family $\left\{ A\,:\,\left|A\cap\left[t\right]\right|=1\right\} $ and is not contained in $\s'_{[t]}$, is $\h$-free.
\end{enumerate}
\end{lem}
\begin{proof}
Suppose that $\h=\left\{ A_{1},\ldots,A_{h}\right\} $ satisfies ({*}{*}). Condition~(1) clearly follows from the assumption that there is no set $T$
of size $t$, such that $|\mathrm{Span}_{\h}\left(T\right)|=h.$ To see that~(2) holds, let $\f$ be a family that contains the family $\left\{ A\,:\,\left|A\cap\left[t\right]\right|=1\right\} $ and is not contained in $\s'_{[t]}$. Note that without loss of generality, $\h$ can be written in the form
\[
\h=\left\{ \left\{ 1\right\} \cup B_{1},\ldots,\left\{ 1\right\} \cup B_{i_{1}},\left\{ 2\right\} \cup B_{i_{1}+1},\ldots,\left\{ 2\right\} \cup B_{i_{2}},\ldots,\left\{ t\right\} \cup B_{i_{t-1}+1}\cdots,\left\{ t\right\} \cup B_{i_{t}},B_{h}\right\},
\]
where $B_{j}\cap\left[t\right]=\emptyset$ for all $[h]$. Consider the families
\[
\underset{i_{1}}{\underbrace{\f_{\left[t\right]}^{\left\{ 1\right\} },\ldots\f_{\left[t\right]}^{\left\{ 1\right\} }}},\ldots,\underset{i_{t}-i_{t-1}}{\underbrace{\f_{\left[t\right]}^{\left\{ t\right\} },\ldots,\f_{\left[t\right]}^{\left\{ t\right\} }}},\f_{\left[t\right]}^{\emptyset}.
\]
Since $\f \supseteq \left\{ A\,:\,\left|A\cap\left[t\right]\right|=1\right\} $, the first $h-1$ of these families consist of the `entire universe' ${{[n] \setminus [t]}\choose{k-1}}$, and since $\f \nsubseteq \s'_{[t]}$, the last family $\f_{\left[t\right]}^{\emptyset}$ is non-empty. Hence, these $h$ families cross contain \emph{any} ordered hypergraph (with edges of appropriate sizes), and in particular, cross contain a copy $(C_1,\ldots,C_h)$ of the ordered hypergraph $\left( B_{1},\ldots,B_{h}\right).$ Therefore, the sets
\[
\left\{ \left\{ 1\right\} \cup C_{1},\ldots,\left\{ 1\right\} \cup C_{i_{1}},\left\{ 2\right\} \cup C_{i_{1}+1},\ldots,\left\{ 2\right\} \cup C_{i_{2}},\ldots,\left\{ t\right\} \cup C_{i_{t-1}+1}\cdots,\left\{ t\right\} \cup C_{i_{t}},C_{h}\right\}
\]
constitute a copy of $\h$ in $\f$. This proves that~(2) holds.

\medskip

In the converse direction, suppose that (1) and (2) hold. It clearly follows from (1) that there is no set $T$ of size $t$ such that $|\mathrm{Span}_{\h}\left(T\right)|=h$. By~(2), the family $\left\langle \left\{ 1\right\} ,\ldots,\left\{ t\right\} \right\rangle \cup\left\{ \left\{ t+1,\ldots,t+k\right\} \right\} $ contains a copy of $\h$. At least $h-1$ edges in this copy belong to $\left\langle \left\{ 1\right\} ,\ldots,\left\{ t\right\} \right\rangle$, and thus, there exists a set of size $t$ that intersects at least $h-1$ edges of $\h$ in exactly one element. This completes the proof.
\end{proof}

The next step is to show that if the forbidden hypergraph satisfies (**), then the junta which approximates $\f$ according to Theorem~\ref{thm:main junta approximation theorem} can be taken to be a $(t,1)$-star.
\begin{lem}
\label{lem:rough stability t porcupines} For any constants $d,h$, there exists a constant $C$ such that the following holds. Let $C<k<n/C$,
and set $\epsilon= \max \left(e^{-k/C},C\frac{k}{n} \right)$.

For any hypergraph $\h$ that satisfies ({**}) and any $\h$-free family $\f\subseteq{{[n]}\choose{k}}$, there exists a $(t,1)$-star $\s'_T$
such that $\mu\left(\f\backslash\s'_T\right) \leq \epsilon \cdot \frac{k}{n}$.
\end{lem}

\begin{proof}
By Theorem~\ref{thm:main junta approximation theorem} and Remark~\ref{Rem:Uniform}, there exists an $O_{d,h}\left(1\right)$-sized set $J$ and a $1$-uniform set $\mathcal{\j} \subseteq {{[n]}\choose{k}}$, such that the junta $\left\langle \j\right\rangle $ is $\h$-free and
\[
\mu\left(\f\backslash\left\langle \j\right\rangle \right)\le\epsilon \mu(\langle \j \rangle).
\]
Our proof will be accomplished once we show that $\j$ contains at most $t$ elements. (Formally, we should apply Theorem~\ref{thm:main junta approximation theorem} with $\epsilon'=\epsilon/t$, to obtain the upper bound $\epsilon' \mu(\langle \j \rangle) = \epsilon' \cdot t\frac{k}{n} = \epsilon \frac{k}{n}$. This can be done, assuming $C$ is sufficiently large.) Hence, it is sufficient to show that the family $\g=\langle \{\{1\},\{2\},\ldots,\{t+1\} \}\rangle$ contains a copy of $\h^+$.

Note that $\g = \{A \in {{[n]}\choose{k}}: \exists 1 \leq i \leq t+1, A \cap J = \{i\}\}$. Let $\h^+ = \{E_1,E_2,\ldots,E_h\}$ and denote $V=E_1 \cup \ldots \cup E_h$. Since $\h$ satisfies (**), we can choose a set $T''$ of size $t+1$, such that each edge of $\h^+$ intersects $T''$ in a single element. Choose an embedding $\phi:V \to [n]$ that sends $T''$ to $\{1,2,\ldots,t+1\}$ and $V \setminus T''$ into $[n]\setminus J$. Then $(\phi(E_1),\ldots,\phi(E_h))$ is a copy of $\h^+$ in $\g$, since each of its edges intersects $J$ in a singleton from $\{1,2,\ldots,t+1\}$. This completes the proof.
\end{proof}

The next step is to show that the theorem `holds locally', i.e., if $\h$ satisfies condition (**), and $\f$ is an $\h$-free family which is a small perturbation of a $(t,1)$-star $\s'_T$, then $\mu(\f \setminus \s'_T)$ is much smaller than $\min_{i \in T}(1 - \mu(\f_{T}^{\{i\}}))$.
\begin{lem}
\label{lem:small-k  Bootstrapping-porcupine} For any constants $d,h$, there exists a constant $C=C(d,h)$ such that the following holds. Let $C<k<n/C$, and denote
\[
\epsilon_0(k):=\begin{cases}
1/C , & k>C\log n;\\
1/(kC), & k \leq C \log n.
\end{cases}
\]
Let $\h$ be a hypergraph that satisfies ({**}) and let $\f \subseteq {{[n]}\choose{k}}$ be an $\h$-free family that satisfies $ \min_{i \in [t]}(1 - \mu(\f_{T}^{\{i\}})) \leq \epsilon$, for some $\epsilon \leq \epsilon_0(k)$.
Then
\[
\mu\left(\f\backslash\s'_{[t]}\right)\le\mu\left(\s'_{[t]}\right)\epsilon^{r}.
\]
\end{lem}
\begin{proof}
The proof is almost exactly the same as that of Lemma~\ref{lem:small-k  Bootstrapping}. The only non-negligible difference is the following. In  Lemma~\ref{lem:small-k  Bootstrapping} we use the fact that if $\f$ is $\h$-free for some hypergraph $\h$ that satisfies (*), then for any $A \subsetneq [t]$, the $h$ families
\[
\f_{[t]}^{[t]},\ldots,\f_{[t]}^{[t]},\f_{[t]}^{A}
\]
are cross free of some $d$-expanded ordered hypergraph with $h$ edges. Here, we use instead the fact (proved in Lemma~\ref{lem:inclusion maximal-> intrinsic-1}) that if $\f$ is $\h$-free for some hypergraph $\h$ that satisfies (**), then the $h$ families
\[
\underset{i_{1}}{\underbrace{\f_{\left[s\right]}^{\left\{ 1\right\} },\ldots\f_{\left[s\right]}^{\left\{ 1\right\} }}},\ldots,\underset{i_{s}-i_{s-1}}{\underbrace{\f_{\left[s\right]}^{\left\{ s\right\} },\ldots,\f_{\left[s\right]}^{\left\{ s\right\} }}},\f_{\left[s\right]}^{\emptyset}
\]
are cross free of some $d$-expanded hypergraph with $h$ edges.
\end{proof}

Now we are ready to prove Theorem~\ref{thm:OR}, along with its stability version.
\begin{thm}
\label{thm:Solution for t-porcupines}For any constants $d,h,r$, there exists a constant $C$ such that the following holds.

Let $\h$ be a hypergraph that satisfies (**). For any $C<k<n/C$, any $\h$-free family $\f\subseteq{{[n]}\choose{k}}$ has
at most ${{n}\choose{k}}-{{n-t}\choose{k}}$ elements, and equality holds if and only if $\f$ is a $(t,1)$-star.

Moreover, let $\epsilon\in\left(0,1/C\right)$ and suppose that $\left|\f\right|\ge \left(1-\epsilon\right) \left({{n}\choose{k}}-{{n-t}\choose{k}}\right)$.
Then there exists a $(t,1)$-star $\s'$ such that $\mu\left(\f\backslash\s'\right)\le\mu\left(\s'\right)\epsilon^{r}$.
\end{thm}

\begin{proof}
The proof follows the proof of Theorem~\ref{thm:Solution for t-stars} with straightforward alterations, replacing Lemmas~\ref{lem:rough stability t stars} and~\ref{lem:small-k  Bootstrapping} with Lemmas~\ref{lem:rough stability t porcupines} and~\ref{lem:small-k  Bootstrapping-porcupine}, respectively.
\end{proof}

\section{Proof of the Erd\H{o}s-Chv\'{a}tal Simplex Conjecture for $\frac{n}{C} \leq k\leq \frac{d-1}{d}n$}
\label{sec:Chvatal}

Theorem~\ref{thm:Solution for t-stars} implies that for any $d$, there exists a constant $C=C\left(d\right)$, such that for all $C<k<n/C$, any family $\f\subseteq{{[n]}\choose{k}}$ that does not contain a special simplex satisfies $\left|\f\right|\le{{n-1}\choose{k-1}}.$ This implies that the Erd\H{o}s-Chv\'{a}tal simplex conjecture holds for all $C<k<n/C.$ As mentioned in the introduction, several previous works proved the conjecture for other ranges of $k$: Frankl~\cite{frankl1987erdos} proved it for any $k>\frac{d-1}{d}n$, Frankl and F\"{u}redi~\cite{frankl1987exact} proved it for all $k\le C$ (provided that $n\ge n_{0}\left(C\right)$), and Keevash and Mubayi~\cite{keevash2010set} proved it for $n/C<k<n/2-O_d(1)$. Hence, the only remaining range is $\frac{n}{2}-O_d(1) \leq k \leq \frac{d-1}{d}n$.

In this section, we prove the following theorem:
\begin{thm}
\label{thm:main-simplex}For any constants $d,\zeta$, there exist $\epsilon_0,n_0$ which depend only on $d,\zeta$ such that the following holds.

Suppose that $n>n_{0}$, that $\zeta n\le k\le\frac{d-1}{d}n$, and that $\f\subseteq{{[n]}\choose{k}}$ is a family that is free of a $d$-simplex and satisfies $|\f| \geq (1-\epsilon_0) {{n-1}\choose{k-1}}$. Then $\f$ is included in a $(1,1)$-star.
\end{thm}
Theorem~\ref{thm:main-simplex} proves the Erd\H{o}s-Chv\'{a}tal conjecture in the range $\frac{n}{C}\le k\le\frac{d-1}{d}n$, provided that $n\ge n_{0}\left(d\right).$ (Note that this range includes the range considered in~\cite{keevash2010set}). Combining with the above results, this proves the conjecture for all $k$, provided $n\ge n_{0}\left(d\right).$

\subsection{Proof overview}
\label{sec:sub:Chvatal:overview}

Recall that a family $\f\subseteq\p\left(\left[n\right]\right)$ is said to be \emph{$s$-wise intersecting} if for any $A_{1},\ldots,A_{s}\in\f$, we have $A_{1}\cap\cdots\cap A_{s}\ne\emptyset$. Families $\f_{1},\ldots,\f_{s}$ are called \emph{$s$-wise cross intersecting} if $A_{1}\cap\cdots\cap A_{s}\ne\emptyset$ for any $A_{1}\in\f_{1},\ldots,A_{s}\in\f_{s}.$

It is clear that any $\left(d+1\right)$-wise intersecting family does not contain a $d$-simplex. On the other hand, if $k>\frac{d-1}{d}n$ then
any $d+1$ sets whose intersection is empty constitute a $d$-simplex. Hence, for such $k$, any family $\f\subseteq{{[n]}\choose{k}}$
that does not contain a $d$-simplex is $\left(d+1\right)$-wise intersecting.

A main idea behind our proof is to reduce the problem of understanding families that do not contain a $d$-simplex to the problem of understanding
$\left(d+1\right)$-wise cross-intersecting families. The reduction is based on the following simple observation.
\begin{obs}\label{obs:simplex}
Let $\f\subseteq{{[n]}\choose{k}}$ be a family that does not contain a $d$-simplex. Then:
\begin{enumerate}
\item If $\left\{ B_{1},\ldots,B_{d+1}\right\} \subseteq B$ is a $d$-simplex, then the families $\f_{B}^{B_{1}},\ldots,\f_{B}^{B_{d+1}}$ are $\left(d+1\right)$-wise cross intersecting.
\item If $\left\{ B_{1},\ldots,B_{d+1}\right\} \subseteq B$ are sets whose intersection is empty, then the families $\f_{B}^{B_{1}},\ldots,\f_{B}^{B_{d+1}}$
are cross free of a $d$-simplex.
\end{enumerate}
\end{obs}

Suppose that $\f\subseteq{{[n]}\choose{k}}$ is free of a $d$-simplex and satisfies $\left|\f\right|\ge(1-\epsilon){{n-1}\choose{k-1}}$. We want to prove that $\f$ is included in a $(1,1)$-star. Our proof consists of four steps:
\begin{enumerate}
\item \textbf{Fairness step.} We apply Proposition~\ref{prop:Fairness} to find a $\left(d+1\right)$-sized set $D=\left\{ i_{1},\ldots,i_{d+1}\right\} $ that is $\epsilon$-fair for the family $\f$. Using Observation~\ref{obs:simplex}(1), we deduce that
    \[
    \f_{D}^{D\backslash\left\{ i_{1}\right\} },\ldots,\f_{D}^{D\backslash\left\{ i_{d+1}\right\} } \subseteq {{[n] \setminus D}\choose{k-d+1}}
    \]
    are $\left(d+1\right)$-wise cross intersecting families whose measures cannot be much smaller than $\frac{k-d+1}{n-d-1}$.
\item \textbf{Stability step.} We show that if $\b_{1},\ldots,\b_{s} \subseteq {{[n]}\choose{k}}$ are $s$-wise cross intersecting families whose measure is not significantly smaller than $\frac{k}{n}$, then all these families are essentially contained in the same $\left(1,1\right)$-star.

    Applying this step to the families $\f_{D}^{D\backslash\left\{ i_{1}\right\} },\ldots,\f_{D}^{D\backslash\left\{ i_{d+1}\right\} }$,  we obtain that there exists some $i_{d+2}$, such that the measures of the families $\f_{D\cup\left\{ i_{d+2}\right\} }^{D\cup\left\{ i_{d+2}\right\} \backslash\left\{ i_{1}\right\} },\ldots,\f_{D\cup\left\{ i_{d+2}\right\} }^{D\cup\left\{ i_{d+2}\right\} \backslash\left\{ i_{d+1}\right\} }$ are close to $1$.
\item \textbf{Bootstrapping step.} We prove that if $\b_{1},\ldots,\b_{d+1}$ are families that are cross free of a $d$-simplex, such that the measures of $\b_{1},\ldots,\b_{d}$ are very close to 1, then the family $\b_{d+1}$ must be \emph{empty}.
\item \textbf{`Sudoku' step.} We perform a sequence of applications of the bootstrapping proposition, exploiting also Observation~\ref{obs:simplex}(2),
to deduce that various slices of the form $\f_{D\cup\left\{ i_{d+2}\right\} }^{B}$ are empty. This will eventually imply that the family $\f$ is contained
in the $\left(1,1\right)$-star $\s_{\{i_{d+2}\}}$.

The motivation behind the name `Sudoku' is as follows. We start with a number of subsets $B_i$ such that each slice $\f_{D\cup\left\{ i_{d+2}\right\} }^{B_i}$ is almost full, which we can label by `1's.  We can deduce from this that for other subsets $B'_j$, the slices $\f_{D\cup\left\{ i_{d+2}\right\} }^{B'_j}$ are empty, which can be labelled by `0's. The procedure of deducing the places of the `0's from the places of the `1's reminds of solving `sudoku' puzzles.
\end{enumerate}

This section is organized as follows. We begin in Section~\ref{sec:sub:Chvatal:preliminaries} with several results that will be used in the proof of Theorem~\ref{thm:main-simplex}. The stability proposition, which is the main step of the proof and which also may be of independent interest, is presented in Section~\ref{sec:sub:Chvatal:stability}. The bootstrapping proposition is presented in Section~\ref{sec:sub:Chvatal:boot}, and in Section~\ref{sec:sub:Chvatal:Sudoku} we present the `Sudoku' step and combine all components into a proof of Theorem~\ref{thm:main-simplex}.

\subsection{Preliminaries}
\label{sec:sub:Chvatal:preliminaries}

We begin with citing two previous results that will be used in the proof of Theorem~\ref{thm:main-simplex}.

\medskip The first result, which is an immediate corollary of a theorem of Frankl and Tokushige~\cite{frankl2011r}, asserts that $s$-wise cross-intersecting families cannot be `too large'.
\begin{thm}[Frankl and Tokushige, 2011]
\label{Thm:Frankl-Tokushige} Let $s\in\mathbb{N}$, let $k\le\frac{s-1}{s}n$, and let $\f_{1},\ldots,\f_{s}\subseteq{{[n]}\choose{k}}$
be $s$-wise cross-intersecting families. Then $\min\left\{ \left|\f_{1}\right|,\ldots,\left|\f_{s}\right|\right\} \le{{n-1}\choose{k-1}}$,
with equality if and only if all the families $\f_{1},\ldots,\f_{s}$ are equal to the same $\left(1,1\right)$-star.
\end{thm}

The second result is a lemma of Friedgut~\cite[Claim~3.1]{friedgut2008measure} which allows us to translate bounds on the size of a family
$\f \subset {{[n]}\choose{k}}$ into bounds on the biased measure of its monotonization $\f^{\uparrow}$.
\begin{lem}[Friedgut, 2008]
\label{lem:Friedgut}There exists an absolute constant $m_{0}$ such that the following holds. Let $k\le n$ be natural numbers,
let $\delta>0$, and write $p=\frac{k}{n}+m_{0}\sqrt{\frac{\log\left(1/\delta\right)}{n}}$. For any family  $\f\subseteq{{[n]}\choose{k}}$, we have $\mu_{p}\left(\f^{\uparrow}\right)\ge\mu\left(\f\right)-\delta$.
\end{lem}

In addition to these previous results, we will need three more propositions. The first proposition obtains an upper bound on the average of  $\frac{s-1}{s}$-biased measures of $s$-wise cross intersecting families, using a simple coupling argument.
\begin{prop}
\label{prop:average measure}Let $s,n\in\mathbb{N}$ be some integers, and suppose that $\f_{1},\ldots,\f_{s}\subseteq\pn$ are $s$-wise cross
intersecting families. Then
\[
\frac{1}{s}\sum_{i\in\left[s\right]}\mu_{\frac{s-1}{s}}\left(\f_{i}\right)\le\frac{s-1}{s}.
\]
\end{prop}
\begin{proof}
For each $i\in\left[n\right]$, let the random variable $\mathbf{X_i}$ be uniformly distributed in $(0,1]$, and for each $j\in\left[s\right]$, let
$\mathbf{A_{j}}:=\{i \in [n]: \mathbf{X_i} \not \in \left(\frac{i-1}{s},\frac{i}{s}\right]\}$. It is clear that each of the sets $\mathbf{A_{1}},\ldots,\mathbf{A_{s}}$ is distributed like a random set drawn from $[n]$ according to the $\frac{s-1}{s}$-biased measure, and on the other hand, that $\mathbf{A_{1}}\cap\cdots\cap \mathbf{A_{s}}=\emptyset$. Since $\f_{1},\ldots,\f_{s}\subseteq\pn$ are $s$-wise cross
intersecting, we have
\[
1=\Pr\left[\mathbf{A_{j}}\notin\f_{j}\mbox{ for some }j\in\left[s\right]\right]\le\sum_{j=1}^{s}\Pr\left[\mathbf{A_{j}}\notin\f_{j}\right]=\sum_{j=1}^{s}\left(1-\mu_{\frac{s-1}{s}}\left(\f_{j}\right)\right).
\]
The assertion follows by rearranging.
\end{proof}

The second proposition compares the measures of different slices of a family.
\begin{lem}
\label{lem:Technical computation} For any constants $\zeta>0,s\in\mathbb{N}$, there exists a constant $n_{0}\left(\zeta,s\right)$ such that the
following holds.

Let $n,k \in \mathbb{N}$ be such that $n>n_0$ and $\zeta\le\frac{k}{n}\le1-\zeta$, let $S \subseteq [n]$ be a set of size $s$, and let $i \in S$.
For any $\epsilon>0$, if a family $\f\subseteq{{[n]}\choose{k}}$ satisfies $\mu\left(\f\right)\ge\frac{k}{n}-\epsilon$ and
$\mu\left(\f_{\left\{ i\right\} }^{\emptyset}\right)\le\epsilon$, then
\[
\mu\left(\f_{S}^{\left\{ i\right\} }\right)\ge1-O_{\zeta,s}\left(\epsilon\right).
\]
\end{lem}
\begin{proof}
First we prove the claim for $|S|=1$, namely, we show that $\mu\left(\f_{\left\{ i\right\} }^{\left\{ i\right\} }\right)\ge1-O_{\zeta,s}\left(\epsilon\right)$. We have
\begin{align*}
\frac{k}{n}-\epsilon&\le\mu\left(\f\right) =\Pr_{\mathbf{A}\sim{{[n]}\choose{k}}}\left[\mathbf{A}\in\f\right]
  =\Pr_{\mathbf{A}\sim{{[n]}\choose{k}}}\left[i\in \mathbf{A}\right]\mu\left(\f_{\left\{ i\right\} }^{\left\{ i\right\} }\right)+\Pr_{\mathbf{A}\sim{{[n]}\choose{k}}}\left[i\notin \mathbf{A}\right]\mu\left(\f_{\left\{ i\right\} }^{\emptyset}\right)\\
 & \le\frac{k}{n}\mu\left(\f_{\left\{ i\right\} }^{\left\{ i\right\} }\right)+\left(1-\frac{k}{n}\right)\epsilon,
\end{align*}
and hence, $\mu\left(\f_{\left\{ i\right\} }^{\left\{ i\right\} }\right) \geq 1-\epsilon \left(\frac{2n}{k}-1\right) \geq 1- O_{\zeta,s}\left(\epsilon\right)$.
Now, suppose that $S=\{i_1,\ldots,i_s\}$ and $\mu(\f_{\{i_1\}}^{\emptyset}) \leq \epsilon$. Let
\[
\b={{[n] \setminus S}\choose{k-1}} \setminus \f_S^{\{i_1\}} \qquad \mbox{ and } \qquad \c={{[n]\setminus \{i_1\}}\choose{k-1}} \setminus \f_{\{i_1\}}^{\{i_1\}}.
\]
Provided that $n_{0}$ is sufficiently large, we have
\begin{align*}
O_{\zeta,s}\left(\epsilon\right)=\mu\left(\c\right) & =\Pr_{\mathbf{A}\sim {{\left[n\right]\backslash\left\{ i_{1}\right\} }\choose{k-1}}}\left[\mathbf{A}\in\c\right]=\sum_{T\subseteq S \setminus \{i_1\}}\Pr_{\mathbf{A}\sim {{\left[n\right]\backslash\left\{ i_{1}\right\} }\choose{k-1}}}\left[\mathbf{A}\cap (S \setminus \{i_1\})=T\right]\mu\left(\c_{S \setminus \{i_1\}}^{T}\right)\\
 & \ge\Pr_{\mathbf{A}\sim {{\left[n\right]\backslash\left\{ i_{1}\right\} }\choose{k-1}}} \left[\mathbf{A}\cap (S \setminus \{i_1\})=\emptyset \right]\mu\left(\c_{S \setminus \{i_1\}}^{\emptyset }\right)\\
 & =\Pr_{\mathbf{A}\sim {{\left[n\right]\backslash\left\{ i_{1}\right\} }\choose{k-1}}} \left[\mathbf{A}\cap (S \setminus \{i_1\})=\emptyset \right]\mu\left(\b\right)=\Omega_{s,\zeta}\left(\mu\left(\b\right)\right).
\end{align*}
Thus, $\mu\left(\f_S^{\left\{ i_1\right\} }\right) =1-\mu(\b) \geq 1- O_{\zeta,s}\left(\epsilon\right)$, as asserted.
\end{proof}

The third proposition provides a relation between biased measures of a monotone family with respect to different biases, and allows to deduce that if these measures satisfy a certain condition then the family can be approximated by a $(1,1)$-star.
\begin{prop}
\label{thm:stability for monotone families}For any constant $0<\zeta<1$, there exists a constant $M=M\left(\zeta\right)>1$ such that the following
holds. Let $p_{0},p_{1}\in\left(\zeta,1-\zeta\right)$ be such that $p_{0}<p_{1}-\zeta$, let $\epsilon>0$, and let $\f\subseteq\pn$ be a monotone
family.
\begin{enumerate}
\item If $\mpo\left(\f\right)\ge p_{0}\left(1-\epsilon\right),$ then $\mu_{p_{1}}\left(\f\right)\ge p_{1}\left(1-\epsilon^{M}\right)$.
\item If, in addition, we have $\mu_{p_{1}}\left(\f\right)\le p_{1}\left(1+\epsilon\right)$, then there exists a $\left(1,1\right)$-star $\s$ such that $\mpo\left(\f\backslash\s\right)\le O_{\zeta}(\epsilon^{M})$.
\end{enumerate}
\end{prop}

To prove Proposition~\ref{thm:stability for monotone families}, we need the two following results. The first is a special case of~\cite[Theorem~3.1]{ellis2016stability}.
\begin{prop}
\label{thm:stability for monotone families1} For any $\zeta>0$, there exist $M\left(\zeta\right),C\left(\zeta\right)>1$, such that the
following holds.

Let $p_{0},p'\in\left(\zeta,1-\zeta\right)$ be such that $p_{0}<p'-\frac{\zeta}{2}$. For any monotone family $\f\subseteq\pn$ that satisfies $\mu_{p'}\left(\f\right)\le p'$ and $\mu_{p_{0}}\left(\f\right)\ge p_{0}-\epsilon$, there exists a $(1,1)$-star $\s$ such that $\mu_{p_{0}}\left(\f\backslash\s\right)\le C\epsilon^{M}.$
\end{prop}

The second result is also taken from~\cite{ellis2016stability}: Its first part is Lemma~2.7(1) in~\cite{ellis2016stability}, and its second part is Lemma~3.6 in~\cite{ellis2016stability}.
\begin{lem}
\label{lem:2.7}Let $0<p_{0}<p_{1}<1$, let $x,t>0$, and let $\f\subseteq\p\left(\left[n\right]\right)$ be a monotone family.
\begin{enumerate}
\item Suppose that $\mu_{p_{1}}\left(\f\right)\le p_{1}^{t}$. Then $\mu_{p_{0}}\left(\f\right)\le p_{0}^{t}.$

\item Suppose that $\mu_{p_{1}}\left(\f\right)\le p_{1}^{t}(1-(1-p_1)^x)$. Then $\mu_{p_{0}}\left(\f\right)\le p_{0}^{t}(1-(1-p_0)^x).$
\end{enumerate}
\end{lem}

Recall that the \emph{dual family} of $\f \subseteq \pn$ is defined by $\f^{\dagger}=\left\{ A\,:\,\left[n\right]\backslash A\notin\f\right\}$.
It is easy to see that for any $0<p<1$ we have $\mu_{1-p}\left(\f^{\dagger}\right)=1-\mu_{p}\left(\f\right),$ and that for any $(1,1)$-star $\s$ we have  $\pn \setminus \left(\f\backslash\s\right)^{\dagger}=\s\backslash\f^{\dagger}$.
Applying Proposition~\ref{thm:stability for monotone families1} to the dual of a family $\f$, with $(1-p_1,1-p)$ in place of $(p_0,p')$, we obtain:
\begin{cor}\label{cor:Chvatal1}
For any $\zeta>0$, there exist $M\left(\zeta\right),C\left(\zeta\right)>1$, such that the following holds.

Let $p_{1},p\in\left(\zeta,1-\zeta\right)$ be such that $p_{1}>p+\frac{\zeta}{2}$. For any monotone family $\f\subseteq\pn$ that satisfies $\mu_{p}\left(\f\right)\ge p$ and $\mu_{p_{1}}\left(\f\right)\le p_{1}+\epsilon$, there exists a $(1,1)$-star $\s$ such that $\mu_{p_{1}}\left(\s\backslash\f\right)\le C\epsilon^{M}.$ Consequently,
\[
\mu_{p_1}\left(\f\backslash\s\right)=\mu_{p_1}\left(\f\right)-\mu_{p_1}\left(\s\right)+\mu_{p_1}\left(\s\backslash\f\right)\le (p_{1}+\epsilon)-p_1+C\epsilon^{M} = \epsilon+C\epsilon^M.
\]
\end{cor}

Now we are ready to prove Proposition~\ref{thm:stability for monotone families}.

\begin{proof}[Proof of Proposition~\ref{thm:stability for monotone families}]
The first part of the proposition follows instantly from Lemma~\ref{lem:2.7}(2). To prove the second part, write $\bar{p}=\frac{p_{0}+p_{1}}{2}.$ If $\mu_{\bar{p}}\left(\f\right)\le \bar{p},$ then the assertion follows from Proposition~\ref{thm:stability for monotone families1} (applied with $\bar{p}$ in place of $p'$). If $\mu_{\bar{p}}\left(\f\right)\ge \bar{p}$, then by Corollary~\ref{cor:Chvatal1} (applied with $\bar{p}$ in place of $p$), there exists a $(1,1)$-star $\s=\s_{\{i\}}$ such that $\mu_{p_1}\left(\f\backslash\s\right) \leq \epsilon+C'\epsilon^{M'}$ for some constants $C'(\zeta),M'(\zeta)>1$. This, in turn, implies
\[
\mu_{p_{1}}\left(\f_{\left\{ i\right\} }^{\emptyset}\right)\le O_{\zeta}\left(\epsilon\right).
\]
Note that the family $\f_{\left\{ i\right\} }^{\emptyset}$ is monotone. Hence, we can apply to it Lemma~\ref{lem:2.7}(1), to obtain
\[
\mu_{p_{0}}\left(\f\backslash\s\right)=\left(O_{\zeta}\left(\mu_{p_1}\left(\f_{\left\{ i\right\} }^{\emptyset}\right)\right)\right)^M \le C\epsilon^{M},
\]
 for some constants $C(\zeta),M(\zeta)>1.$ This completes the proof.
\end{proof}

\subsection{Stability result for the upper bound on the sizes of $s$-wise cross intersecting families}
\label{sec:sub:Chvatal:stability}

In this subsection we establish Step~2 of the proof of Theorem~\ref{thm:main-simplex} which asserts that any $s$-wise cross intersecting families $\f_{1},\ldots,\f_{s}$ whose measures are not must smaller than $\frac{k}{n}$, are all essentially contained in the same $\left(1,1\right)$-star. This can be viewed as a stability result for Theorem~\ref{Thm:Frankl-Tokushige}.
\begin{prop}
\label{prop:stability for s-wise intersecting families} For any
$\zeta,\epsilon>0$ and $s\in\mathbb{N}$, there exist constants $M=M\left(\zeta,s\right)>1$ and $n_{0}=n_{0}\left(\zeta,s,\epsilon\right)\in\mathbb{N}$ such
that the following holds.

Let $n>n_{0}$, let $k\in\left(\zeta n,\left(\frac{s-1}{s}-\zeta\right)n\right)$, and let $\f_{1},\ldots,\f_{s}\subseteq{{[n]}\choose{k}}$ be $s$-wise cross intersecting families such that
\[
\min\left\{ \left|\f_{1}\right|,\ldots,\left|\f_{s}\right|\right\} \ge \left(1-\epsilon\right) {{n-1}\choose{k-1}}.
\]
Then there exists a $\left(1,1\right)$-star $\s$ such that $\mu\left(\f_{i}\backslash\s\right)\le O_{s,\zeta}\left(\epsilon^{M}\right)$
for any $i\in\left[s\right]$.
\end{prop}

\begin{proof}
Throughout the proof we assume that $\epsilon$ is smaller than a sufficiently small constant depending only on $s,\zeta$, for otherwise the proposition holds trivially. We shall also assume that $n_{0}$ is sufficiently large.

Let $\f_1,\ldots,\f_s$ be families that satisfy the hypothesis of the proposition, and let $i \in [s]$.
Let $m_{0}$ be as in Lemma~\ref{lem:Friedgut}, and write $p_{0}=\frac{k}{n}+\sqrt{\frac{m_{0}\log n}{n}}$ and $p_{1}=\frac{s-1}{s}$.
By Lemma \ref{lem:Friedgut} (applied with $\delta=1/n$), we have
\[
\mpo\left(\f_{i}^{\uparrow}\right)\ge\frac{k}{n}\left(1-\epsilon\right)-\frac{1}{n}\ge p_{0}\left(1-2\epsilon\right),\label{eq:Step 1}
\]
provided that $n_{0}$ is sufficiently large.

Using again the assumption that $n_{0}$ is sufficiently large, we have $p_{0}<p_{1}-\frac{\zeta}{2}$. Applying Proposition~\ref{thm:stability for monotone families} with $\zeta/2$ in place of $\zeta$ and choosing $M$ appropriately, we obtain
\[
\mu_{p_{1}}\left(\f_{i}^{\uparrow}\right)\ge p_{1}\left(1-\left(2\epsilon\right)^{M}\right)=\frac{s-1}{s}\left(1-\left(2\epsilon\right)^{M}\right)>\frac{s-1}{s}-\epsilon=p_{1}-\epsilon,\label{eq:low}
\]
provided that $\epsilon$ is small enough. By Proposition~\ref{prop:average measure}, this implies
\[
\frac{s-1}{s}\ge\frac{1}{s}\sum_{j=1}^{s}\mu_{p_{1}}\left(\f_{j}^{\uparrow}\right)\ge\frac{1}{s}\mu_{p_{1}}\left(\f_{i}^{\uparrow}\right)+
\frac{s-1}{s}\left(p_{1}-\epsilon\right).\label{eq:up}
\]
Rearranging, we obtain
\[
\mu_{p_{1}}\left(\f_{i}^{\uparrow}\right)\le\frac{s-1}{s}+O_{s}\left(\epsilon\right).
\]
By Proposition \ref{thm:stability for monotone families}(2), this implies that
there exist a $\left(1,1\right)$-star $\s_{\{j_i\}}:=\left\{ A\in\pn:\, j_{i}\in A\right\} $,
and $M=M\left(\zeta,s\right)>1$ such that
\[
\mpo\left(\f_{i}^{\uparrow}\backslash\s_{\{j_i\}}\right)=O_{\zeta,s}\left(\epsilon^{M}\right).\label{eq:mup small}
\]
Write $\g:=\left(\f_{i}\right)_{\left\{ j_{i}\right\} }^{\emptyset}$. By Lemma \ref{lem:Friedgut}, we have
\begin{equation}
\mu\left(\f_{i}\backslash\s_{\{j_i\}}\right)=O\left(\mu\left(\g\right)\right)\le O\left(\mu_{p_{0}}\left(\g^{\uparrow}\right)\right)+O \left(\frac{1}{n}\right)=O\left(\mu_{p_{0}}
\left(\f_{i}^{\uparrow}\backslash\s_{\{j_i\}}\right)\right)+O\left(\frac{1}{n}\right)=O_{\zeta,s}\left(\epsilon^{M}\right),\label{eq:fi esentially contained in a dictatorship}
\end{equation}
provided that $n_{0}$ is large enough.

It now only remains to show that the $\left(1,1\right)$-stars $\{\s_{\{j_i\}}\}_{i=1,\ldots,s}$ are all equal. Let $S=\left\{ j_{1},\ldots,j_{s}\right\} $, and suppose on the contrary that $\left|S\right|>1$. Since $\f_1,\ldots,\f_s$ are $s$-wise cross-intersecting, this implies that the families
$\left(\f_{i}\right)_{S}^{\left\{ j_{i}\right\} }$ are $s$-wise cross intersecting as well. Hence, by Theorem~\ref{Thm:Frankl-Tokushige}, there exists $\ell$ such that
\begin{equation}\label{Eq:Chvatal-Aux1}
\mu \left(\left(\f_{\ell}\right)_{S}^{\left\{ j_{\ell}\right\} }\right) \leq \frac{k-1}{n-|S|}.
\end{equation}
However, since $\mu(\f_{\ell}) \geq \frac{k}{n} -O(\epsilon)$ by assumption and $\mu \left(\left(\f_{\ell}\right)_{j_{\ell}}^{\emptyset }\right) \leq
O_{\zeta,s}\left(\epsilon^{M}\right)$ by~\eqref{eq:fi esentially contained in a dictatorship}, Lemma~\ref{lem:Technical computation} implies that
$\mu\left(\f_{S}^{\left\{ j_{\ell}\right\} }\right)\ge1-O_{\zeta,s}\left(\epsilon\right)$, which contradicts~\eqref{Eq:Chvatal-Aux1} if $n_0$ is sufficiently large. This completes the proof.
\end{proof}

\subsection{The bootstrapping proposition}
\label{sec:sub:Chvatal:boot}

In this subsection we establish Step~3 of the proof of Theorem~\ref{thm:main-simplex}, namely, we show that if some families $\f_{1},\ldots,\f_{d+1}\subseteq{{[n]}\choose{k}}$ are cross free of a $d$-simplex and $\mu\left(\f_{i}\right)$ is close to 1 for any
$i\in\left[d\right]$, then the family $\f_{d+1}$ must be empty.
\begin{prop}
\label{Prop:Bootstrapping proposition}For any $d \in\mathbb{N}$ and $0 < \zeta <1$,
there exist constants $\epsilon_0,n_0$ that depend only on $d,\zeta$, such that the following holds.

Let $n>n_{0}$ and let $k_{1},k_{2} \in \left(\zeta n,\frac{d-1}{d}n+d+1\right)$.
Let $\f_{1}\subseteq {{[n]}\choose{k_1}},\ldots,\f_{d}\subseteq{{[n]}\choose{k_1}},\f_{d+1}\subseteq {{\left[n\right]}\choose{k_{2}}}$
be families that are cross free of a $d$-simplex, and suppose that $\mu\left(\f_{i}\right)\ge 1-\epsilon_{0}$
for all $i\in\left[d\right]$. Then $\f_{d+1}= \emptyset$.
\end{prop}
\begin{proof}
Let $\f_{1},\ldots,\f_{d+1}$ be as in the hypothesis of the proposition. Suppose on the contrary that $\f_{d+1} \neq \emptyset$ and let $E \in \f_{d+1}$. Let
$S \subseteq [n]$ be a set of size $l>d$ (to be determined below) that satisfies $|S\cap E|=d$, and write $S\cap E=\left\{ i_{1},\ldots,i_{d}\right\} $.

As the sets $\{S \setminus \{i_1\}, \ldots, S \setminus \{i_d\}, S \cap E\}$ constitute a $d$-simplex, it follows from Observation~\ref{obs:simplex}(1)
that the families
\[
\left(\f_{1}\right)_{S}^{S\backslash\left\{ i_{1}\right\} },\left(\f_{2}\right)_{S}^{S\backslash\left\{ i_{2}\right\} },\ldots,\left(\f_{d}\right)_{S}^{S\backslash\left\{ i_{d}\right\} },\left(\f_{d+1}\right)_{S}^{S\cap E}
\]
are $\left(d+1\right)$-wise cross intersecting. Since the family $\left(\f_{d+1}\right)_{S}^{S\cap E}$ is non-empty, this implies that the families
\[
\left(\f_{1}\right)_{S}^{S\backslash\left\{ i_{1}\right\} },\left(\f_{2}\right)_{S}^{S\backslash\left\{ i_{2}\right\} },\ldots,\left(\f_{d}\right)_{S}^{S\backslash\left\{ i_{d}\right\} }\subseteq {{\left[n\right]\backslash S}\choose{k_{1}-l+1}}
\]
are $d$-wise cross intersecting as well.

We now would like to apply Theorem~\ref{Thm:Frankl-Tokushige} to the families $\left(\f_{1}\right)_{S}^{S\backslash\left\{ i_{1}\right\} },\ldots, \left(\f_{d}\right)_{S}^{S\backslash\left\{ i_{d}\right\}}$. To do so, we should have $\frac{k_{1}-l+1}{n-l}\le\frac{d-1}{d}$, and hence we choose $l$ to be  the smallest integer for which this condition holds. (Since $k_1 \leq \frac{d-1}{d}n+d+1$, we clearly have $l=O_{d}\left(1\right)$). By Theorem~\ref{Thm:Frankl-Tokushige}, there exists $j \in [d]$ such that
\begin{equation}\label{Eq:Chvatal-Aux2}
\mu\left(\left(\f_{j}\right)_{S}^{S \setminus \{i_{j}\}}\right)\le\frac{k_{1}-l+1}{n-l}.
\end{equation}
On the other hand, we have
\begin{align*}
\epsilon_{0}&\ge1-\mu\left(\f_{j}\right)\ge \Pr_{\mathbf{A}\sim{{[n]}\choose{k}}}\left[\left(\mathbf{A}\cap S= S \setminus \{i_{j}\}\right) \wedge (\mathbf{A} \not \in \f_j) \right] \\
&= \Pr_{\mathbf{A}\sim{{[n]}\choose{k}}}\left[\mathbf{A}\cap S= S \setminus \{i_{j}\}\right] \Pr_{\mathbf{A}\sim{{[n]}\choose{k}}} \left[\mathbf{A} \not \in \f_j | \mathbf{A}\cap S= S \setminus \{i_{j}\} \right] \\
&=\Pr_{\mathbf{A}\sim{{[n]}\choose{k}}}\left[\mathbf{A}\cap S=S \setminus \left\{ i_{j}\right\} \right]\left(1-\mu\left(\left(\f_{j}\right)_{S}^{S \setminus\left\{ i_{j}\right\} }\right)\right)=\Omega_{s,\zeta}\left(1-\mu\left(\left(\f_{j}\right)_{S}^{S \setminus \left\{ i_{j}\right\} }\right)\right),
\end{align*}
and thus, $\mu\left(\left(\f_{j}\right)_{S}^{S \setminus \{i_{j}\}}\right)\ge 1-O_{s,\zeta}(\epsilon_0)$, which contradicts~\eqref{Eq:Chvatal-Aux2} provided $\epsilon_{0}$ is sufficiently small. This completes the proof.
\end{proof}

\subsection{Proof of the Erd\H{o}s-Chv\'{a}tal simplex conjecture for all $n>n_0(d)$}
\label{sec:sub:Chvatal:Sudoku}

In this subsection we combine the components presented in the previous subsections to prove Theorem~\ref{thm:main-simplex}, which completes the proof of the Erd\H{o}s-Chv\'{a}tal conjecture for all $n>n_0(d)$.

Let $d,\zeta$ be fixed constants, $n,k$ be natural numbers satisfying $\zeta n\le k\le\frac{d-1}{d}n$, and $\f\subseteq{{[n]}\choose{k}}$ be a family that is free of a $d$-simplex and satisfies $|\f| \geq (1-\epsilon_0){{n-1}\choose{k-1}}$. We would like to show that $\f$ is contained in a $(1,1)$-star, provided that $n$ is sufficiently large and $\epsilon_0$ is sufficiently small.

\begin{proof}[Proof of Theorem~\ref{thm:main-simplex}]
Set $\epsilon_0$ to be a small constant to be determined below. By Proposition~\ref{prop:Fairness}, there exists a set $S$ of size $d+1$ that is $\epsilon_{0}$-fair for $\f$. Assume without loss of generality that $S=\left[d+1\right]$. For each $i \in [d+1]$, we have
\begin{equation}\label{Eq:Chvatal-Aux3}
\mu\left(\f_{\left[d+1\right]}^{\left[d+1\right]\backslash\left\{ i\right\} }\right)\ge\mu\left(\f\right)-\epsilon_{0}\ge\frac{k}{n}-2\epsilon_{0}.
\end{equation}
The families $\{\f_{\left[d+1\right]}^{\left[d+1\right]\backslash\left\{ i\right\} }\}$
are $\left(d+1\right)$-wise cross intersecting. By Proposition~\ref{prop:stability for s-wise intersecting families} (applied with $s=d+1$), this implies that
there exists $M>1$ and a $\left(1,1\right)$-star $\s\subseteq\p\left(\left[n\right]\backslash\left[d+1\right]\right)$,
such that for any $i \in [d+1]$,
\begin{equation}\label{Eq:Chvatal-Aux4}
\mu\left(\f_{\left[d+1\right]}^{\left[d+1\right]\backslash\left\{ i\right\} }\backslash\s\right)=O_{d,\zeta}\left(\epsilon_{0}^{M}\right)=O_{d,\zeta}\left(\epsilon_{0}\right).
\end{equation}
Suppose w.l.o.g. that $\s=\s_{\{d+2\}}$.  Then for any $i\in\left[d+1\right]$, we have
\begin{equation}
\label{Eq:Chvatal-Aux5}
\mu\left(\f_{\left[d+2\right]}^{\left[d+2\right]\backslash\left\{ i\right\} }\right)=\frac{\mu\left(\f_{\left[d+1\right]}^{\left[d+1\right]\backslash\left\{ i\right\} }\right)-\mu\left(\f_{\left[d+1\right]}^{\left[d+1\right]\backslash\left\{ i\right\} }\backslash\s\right)}{\mu\left(\s\right)} \geq 1-O_{d,\zeta}\left(\epsilon_0\right),
\end{equation}
where the last inequality follows from the assumption on $\f$ and Equations~\eqref{Eq:Chvatal-Aux3} and~\eqref{Eq:Chvatal-Aux4}.

\medskip This establishes Step~2 of the proof of Theorem~\ref{thm:main-simplex}. Step~3 -- a bootstrapping argument -- was established in Proposition~\ref{Prop:Bootstrapping proposition} and will be used soon. Now we present the last step of the proof -- a `Sudoku step' in which we consider slices of the form $\f_{[d+2]}^B$ sequentially, and show that for any $B$ that does not contain $d+2$, the slice $\f_{[d+2]}^B$ is empty. This will show that $\f \subseteq \s$, completing the proof.

The first part of the `Sudoku step' asserts that for each element $E\in\f$ we either have $d+2\in E$ or $\left\{ 1,\ldots,d+1\right\} \subseteq E$.
\begin{claim}
\label{Claim 2} Provided that $\epsilon_{0}$ is sufficiently small and that $n_{0}$ is sufficiently large, we have
\[
\f_{\left\{ i,d+2\right\} }^{\emptyset}=\emptyset
\]
for any $i\in\left[d+1\right]$. In particular, $\mu\left(\f_{\left[d+2\right]}^{\left\{ d+2\right\} }\right)\ge1-O_{d,\zeta}\left(\epsilon_{0}\right).$ \end{claim}

\begin{proof}
To prove the first part of the claim, we show that the family $\f_{\left[d+2\right]}^{B}$
is empty for any $B\subseteq\left[d+2\right]\backslash\left\{ 1,d+2\right\} $.
Let $B\subseteq\left[d+2\right]\backslash\left\{ 1,d+2\right\} $.
Since the intersection of the sets $B,\left[d+2\right]\backslash\left\{ 2\right\} ,\ldots,\left[d+2\right]\backslash\left\{ d+1\right\} $
is empty, by Observation~\ref{obs:simplex}(2) the families
\[
\f_{\left[d+2\right]}^{B},\f_{\left[d+2\right]}^{\left[d+2\right]\backslash\left\{ 2\right\} },\ldots,\f_{\left[d+2\right]}^{\left[d+2\right]\backslash\left\{ d+1\right\} }
\]
are cross free of a $d$-simplex. By Proposition~\ref{Prop:Bootstrapping proposition} (which can applied due to~\eqref{Eq:Chvatal-Aux5}), this implies that the family $\f_{\left[d+2\right]}^{B}$ is empty, provided that $\epsilon_{0}$ is sufficiently small and $n_{0}$ is sufficiently large.

The `in particular' part holds since
\[
1-\mu\left(\f_{\left[d+2\right]}^{\left\{ d+2\right\} }\right)\le O\left(1-\mu\left(\f_{\left\{ 1,d+2\right\} }^{\left\{ d+2\right\} }\right)\right)=O\left(1-\frac{n-2}{k-1}\mu\left(\f_{\left\{ 1\right\} }^{\emptyset}\right)\right)\le O_{d,\zeta}\left(\epsilon_{0}\right),
\]
where the equality holds as all sets in $\f_{\{1\}}^{\emptyset}$ contain the element $d+2$ (by the first part of the claim), and the last inequality holds since $[d+1]$ is $\epsilon_0$-fair for $\f$. This completes the proof of the claim.
\end{proof}

We now complete the `Sudoku step' by showing that $\f_{[d+2]}^{B}=\emptyset$ for any $B$ such that $d+2 \not \in B$.

Since the intersection of the sets $B,\{d+2\},\ldots,\{d+2\}$ is empty, by Observation~\ref{obs:simplex}(2) the families
\[
\f_{\left[d+2\right]}^{B},\f_{\left[d+2\right]}^{\{d+2\}},\ldots,\f_{\left[d+2\right]}^{\{d+2\}}
\]
are cross free of a $d$-simplex. By Proposition~\ref{Prop:Bootstrapping proposition} (which can applied due to Claim~\ref{Claim 2}), this implies that $\f_{\left[d+2\right]}^{B}=\emptyset$, provided that $\epsilon_{0}$ is sufficiently small and $n_{0}$ is sufficiently large. This implies that $\f \subseteq \s_{\{d+2\}}$, and thus completes the proof of Theorem~\ref{thm:main-simplex}.
\end{proof}

\section*{Acknowledgements}

The authors are grateful to Gil Kalai and to David Ellis for valuable discussions and suggestions.

\renewcommand{\baselinestretch}{0.9}\normalsize
\bibliographystyle{plain}
\bibliography{Refs}
\renewcommand{\baselinestretch}{1.0}\normalsize

\end{document}